\pdfoutput=1 
\documentclass[10pt]{article}

\usepackage{amsmath, amsfonts, amsthm, amssymb, enumerate, url, float}
\usepackage[all,arc,color]{xy}
\usepackage{enumerate}
\usepackage{mathrsfs}
\usepackage{mathtools}
\usepackage[utf8]{inputenc}
\usepackage{tikz-cd}

\begingroup
    \makeatletter
    \@for\theoremstyle:=definition,remark,plain\do{%
        \expandafter\g@addto@macro\csname th@\theoremstyle\endcsname{%
            \addtolength\thm@preskip\parskip
            }%
        }
 \endgroup
 
   \usepackage[titletoc,toc,title]{appendix}
 
   \usepackage[parfill]{parskip}
   \usepackage{anysize}
   \marginsize{1in}{1in}{1in}{1in}
   \usepackage{thmtools}

\declaretheorem[name=Theorem,numberwithin=section]{thm}
\declaretheorem[name=Proposition,numberlike=thm]{prop}
\declaretheorem[name=Lemma,numberlike=thm]{lemma}
\declaretheorem[name=Corollary,numberlike=thm]{cor}

\declaretheorem[name=Definition,style=definition,qed=$\blacktriangle$,numberlike=thm]{defn}

\declaretheorem[name=Remark,style=definition,qed=$\blacktriangle$,numberlike=thm]{rmk}

\newcounter{commentCounter}
\newcommand{\bu}{\bullet}
\newcommand{\dd}{\mathrm{d}}
\newcommand{\cc}{\mathrm{c}}
\newcommand{\sta}{\star}
\newcommand{\cop}[1]{{#1}{}^{\cc}}
\newcommand{\dc}{\cop{\dd}}
\newcommand{\dop}[1]{{#1}{}^{\sta}}

\newcommand{\ds}{\dop{\dd}}

\newcommand{\ddc}{\dd \dc}

\newcommand{\vol}{\mathsf{vol}}

\newcommand{\G}{\mathrm{G}_2}
\newcommand{\Spin}[1]{\mathrm{Spin}(#1)}

\newcommand{\U}[1]{\mathrm{U}(#1)}
\newcommand{\SU}[1]{\mathrm{SU}(#1)}

\newcommand{\R}{\mathbb R}
\newcommand{\C}{\mathbb C}

\newcommand{\Z}{\mathbb Z}
\newcommand{\PR}{\mathbb P}

\newcommand{\ph}{\varphi}
\newcommand{\ps}{\psi}
\newcommand{\st}{\ast}
\newcommand{\hk}{\mathbin{\! \hbox{\vrule height0.3pt width5pt depth 0.2pt \vrule height5pt width0.4pt depth 0.2pt}}}

\newcommand{\tr}{\operatorname{Tr}}

\newcommand{\ddx}[1]{\frac{\del}{\del x^{#1}}}
\newcommand{\ddxs}[1]{\del_{#1}}
\newcommand{\dx}[1]{d x^{#1}}

\newcommand{\del}{\partial}
\newcommand{\delbar}{\overline{\partial}}

\newcommand{\DR}{\mathrm{dR}}

\newcommand{\im}{\operatorname{im}}

\newcommand{\curl}{\ensuremath{\operatorname{curl}}}
\newcommand{\grad}{\ensuremath{\operatorname{grad}}}
\newcommand{\dive}{\ensuremath{\operatorname{div}}}

\newcommand{\sym}{\ensuremath{S^2 (T^* M)}}
\newcommand{\symo}{\ensuremath{S^2_0 (T^* M)}}
\newcommand{\Sym}{\ensuremath{\mathcal{S}}}
\newcommand{\Symo}{\ensuremath{\mathcal{S}_0}}

\newcommand{\red}[1]{\textcolor{red}{#1}}
\newcommand{\blue}[1]{\textcolor{blue}{#1}}
\newcommand{\brown}[1]{\textcolor{brown}{#1}}
\newcommand{\magenta}[1]{\textcolor{magenta}{#1}}
\newcommand{\orange}[1]{\textcolor{orange}{#1}}
\newcommand{\cyan}[1]{\textcolor{cyan}{#1}}
\newcommand{\olive}[1]{\textcolor{olive}{#1}}
\newcommand{\teal}[1]{\textcolor{teal}{#1}}
\newcommand{\violet}[1]{\textcolor{violet}{#1}}
\newcommand{\gray}[1]{\textcolor{gray}{#1}}
\newcommand{\purple}[1]{\textcolor{purple}{#1}}

\makeatletter
\let\c@equation\c@thm
\makeatother
\numberwithin{equation}{section}

   \usepackage[colorlinks = true,
            linkcolor = blue,
            urlcolor  = blue,
            citecolor = blue,
            anchorcolor = blue]{hyperref}

\begin{document}

\title{The $\mathcal L_B$-cohomology on compact torsion-free $\G$~manifolds \\ and an application to `almost' formality}

\author{Ki Fung Chan \\ {\it Chinese University of Hong Kong} \\\tt{kifung@link.cuhk.edu.hk} \and Spiro Karigiannis \\ {\it Department of Pure Mathematics, University of Waterloo} \\ \tt{karigiannis@uwaterloo.ca} \and Chi Cheuk Tsang \\ {\it Chinese University of Hong Kong} \\ \tt{1155062770@link.cuhk.edu.hk} }

\date{January 18, 2018}

\maketitle

\begin{abstract}
We study a cohomology theory $H^{\bu}_{\ph}$, which we call the $\mathcal L_B$-cohomology, on compact torsion-free $\G$~manifolds. We show that $H^k_{\ph} \cong H^k_{\DR}$ for $k \neq 3, 4$, but that $H^k_{\ph}$ is infinite-dimensional for $k = 3,4$. Nevertheless there is a canonical injection $H^k_{\DR} \to H^k_{\ph}$. The $\mathcal L_B$-cohomology also satisfies a Poincar\'e duality induced by the Hodge star. The establishment of these results requires a delicate analysis of the interplay between the exterior derivative $\dd$ and the derivation $\mathcal L_B$, and uses both Hodge theory and the special properties of $\G$-structures in an essential way. As an application of our results, we prove that compact torsion-free $\G$~manifolds are `almost formal' in the sense that most of the Massey triple products necessarily must vanish.
\end{abstract}

\tableofcontents

\section{Introduction} \label{sec:intro}

Let $(M, \ph)$ be a manifold with $\G$-structure. Here $\ph$ is a smooth $3$-form on $M$ that is \emph{nondegenerate} in a certain sense that determines a Riemannian metric $g$ and a volume form $\vol$, hence a dual $4$-form $\ps$. We say that $(M, \ph)$ is a \emph{torsion-free} $\G$~manifold if $\nabla \ph = 0$. Note that this implies that $\nabla \ps = \dd \ph = \dd \ps = 0$ as well. In fact, it is now a classical result~\cite{FG} that the pair of conditions $\dd \ph = \dd \ps = 0$ are actually equivalent to $\nabla \ph = 0$.

The forms $\ph$ and $\ps$ can be used to construct a vector-valued $2$-form $B$ and a vector-valued $3$-form $K$, respectively, by raising an index using the metric. These vector-valued forms were studied in detail by Kawai--L\^e--Schwachh\"ofer in~\cite{KLS1} in the context of the Fr\"olicher--Nijenhuis bracket.

These vector-valued forms $B$ and $K$ induce \emph{derivations} $\mathcal L_B$ and $\mathcal L_K$ on the space $\Omega^{\bu}$ of forms on $M$, of degree $2$ and $3$, respectively. From these derivations we can define \emph{cohomology theories}. We call these the $\mathcal{L}_B$-cohomology, denoted $H^{\bu}_{\ph}$, and the $\mathcal{L}_K$-cohomology, denoted $H^{\bu}_{\ps}$. When $M$ is compact, the $\mathcal L_K$-cohomology was studied extensively by Kawai--L\^e--Schwachh\"ofer in~\cite{KLS2}. In the present paper we study in detail the $\mathcal L_B$-cohomology when $M$ is compact. Specifically, we compute $H^k_{\ph}$ for all $k$. The results are summarized in Theorem~\ref{thm:Hph}, which we restate here:

{\bf Theorem~\ref{thm:Hph}.} {\em The following relations hold.
\begin{itemize} \setlength\itemsep{-1mm}
\item $H^k_{\ph} \cong H^k_{dR}$ for $k=0,1,2,5,6,7$.
\item $H^k_{\ph}$ is infinite-dimensional for $k = 3,4$.
\item There is a canonical injection $\mathcal{H}^k \hookrightarrow H^k_{\ph}$ for all $k$.
\item The Hodge star induces isomorphisms $\st: H^k_{\ph} \cong H^{7-k}_{\ph}$.
\end{itemize}
}

The proof involves a very delicate analysis of the interplay between the exterior derivative $\dd$ and the derivation induced by $B$, and uses Hodge theory in an essential way.

As an application of our results, we study the question of \emph{formality} of compact torsion-free $\G$~manifolds. This is a longstanding open problem. It has been studied by many authors, including Cavalcanti~\cite{Cavalcanti-ddc}. In particular, the paper~\cite{Verbitsky} by Verbitsky has very close connections to the present paper. What is called $\dd_c$ in~\cite{Verbitsky} is $\mathcal L_B$ in the present paper. Verbitsky's paper contains many excellent ideas. Unfortunately, there are some gaps in several of the proofs in~\cite{Verbitsky}. Most important for us, there is a gap in the proof of~\cite[Proposition 2.19]{Verbitsky}, which is also used to prove~\cite[Proposition 2.20]{Verbitsky}, among several other results in~\cite{Verbitsky}. We give a different proof of this result, which is our Proposition~\ref{prop:quasi-isom}. We then use this to prove our Theorem~\ref{thm:almost-formal}, which essentially says that a compact torsion-free $\G$~manifold is `almost formal' in the sense that its de Rham complex is equivalent to a differential graded algebra with all differentials trivial except one.

A consequence of our Theorem~\ref{thm:almost-formal} is that almost all of the Massey triple products vanish on a compact torsion-free $\G$~manifold. This gives a new topological obstruction to the existence of torsion-free $\G$-structures on compact manifolds. The precise statement is the following:

{\bf Corollary~\ref{cor:Massey}.} {\em Let $M$ be a compact torsion-free $\G$~manifold. Consider cohomology classes $[\alpha]$, $[\beta]$, and $[\gamma] \in H^{\bu}_{\DR}$. If the Massey triple product $\langle [\alpha], [\beta], [\gamma] \rangle$ is defined and we have $|\alpha| + |\beta| \neq 4$ and $|\beta| + |\gamma| \neq 4$, then $\langle [\alpha], [\beta], [\gamma] \rangle = 0$.
}

We also prove the following stronger result in the case of full holonomy $\G$ (the ``irreducible'' case):

{\bf Theorem~\ref{thm:irreducibleMassey}.} {\em Let $M$ be a compact torsion-free $\G$~manifold with full holonomy $\G$, and consider cohomology classes $[\alpha]$, $[\beta]$, and $[\gamma] \in H^{\bu}_{\DR}$. If the Massey triple product $\langle [\alpha], [\beta], [\gamma] \rangle$ is defined, then $\langle [\alpha], [\beta], [\gamma] \rangle = 0$ except possibly in the case when $|\alpha| = |\beta| = |\gamma| = 2$.
}

The Massey triple products on a compact torsion-free $\G$~manifold are not discussed in~\cite{Verbitsky}.

{\bf Organization of the paper.} In the rest of this section, we discuss the domains of validity of the various results in this paper in Remark~\ref{rmk:when-torsion-free}, then we consider notation and conventions, and conclude with the statement of a trivial result from linear algebra that we use frequently.

Section~\ref{sec:main} is the heart of the paper, where we establish the various relations between the derivations $\dd$, $\iota_B$, $\iota_B$, $\mathcal L_B$, and $\mathcal L_K$. We begin with a brief summary of known facts about $\G$-structures that we will need in Section~\ref{sec:forms}. In Section~\ref{sec:dLap} we study the operators $\dd$ and $\Delta$ in detail. Some of the key results are Proposition~\ref{prop:dfigure}, which establishes Figure~\ref{figure:d}, and Corollary~\ref{cor:d-relations} and Proposition~\ref{prop:Laplacian} which establish second order differential identities. These have appeared before (without proof) in a paper of Bryant~\cite[Section 5.2]{Bryant}. But see Remark~\ref{rmk:Bryant}. A new and crucial result in Section~\ref{sec:dLap} is Theorem~\ref{thm:harmonic1} which relates the kernels of various operators on $\Omega^1$. In Section~\ref{sec:LBLK} we introduce the derivations $\iota_B$, $\iota_K$, $\mathcal L_B$, and $\mathcal L_K$ and study their basic properties. One of the highlights is Corollary~\ref{cor:LBLKfigures}, which establishes Figures~\ref{figure:LB} and~\ref{figure:LK}.

In Section~\ref{sec:cohom} we study and compute the $\mathcal L_B$-cohomology $H^{\bu}_{\ph}$ of a compact torsion-free $\G$~manifold. We use heavily both the results of Section~\ref{sec:main} and Hodge theory. This section culminates with the proof of Theorem~\ref{thm:Hph}. Then in Section~\ref{sec:formality} we apply the results of Section~\ref{sec:cohom} to study the Massey triple products of compact torsion-free $\G$~manifolds.

\begin{rmk} \label{rmk:when-torsion-free}
We summarize here the domains of validity of the various sections of the paper.
\begin{itemize} \setlength\itemsep{-1mm}
\item All results of Section~\ref{sec:forms} except the last one (Proposition~\ref{prop:Liederiv}), are valid for any $\G$-structure.
\item Proposition~\ref{prop:Liederiv} as well as \emph{the entirety of Section~\ref{sec:dLap}}, assume that $(M, \ph)$ is torsion-free.
\item In Section~\ref{sec:LBLK}, the results that only involve the algebraic derivations $\iota_B$ and $\iota_K$, up to and including Proposition~\ref{prop:iotaKfigure}, are valid for any $\G$-structure.
\item The rest of Section~\ref{sec:LBLK}, beginning with Corollary~\ref{cor:LBLKfigures}, uses the results of Section~\ref{sec:dLap} heavily and is only valid in the torsion-free setting.
\item The cohomology theories introduced in Section~\ref{sec:cohom-defn} make sense on any torsion-free $\G$~manifold. However, beginning in Section~\ref{sec:computeHph0123} and for the rest of the paper, we assume that $(M, \ph)$ is a \emph{compact} torsion-free $\G$~manifold, as we use Hodge theory throughout.
\end{itemize}
\end{rmk}

{\bf Notation and conventions.} We mostly follow the notation and conventions of~\cite{K-flows}, and we point out explicitly whenever our notation differs significantly. Let $(M, g)$ be an oriented smooth Riemannian $7$-manifold. Let $\{ e_1 , \ldots, e_7 \}$ be a local frame for $TM$ with dual coframe $\{ e^1, \ldots, e^7 \}$. It can be a local coordinate frame $\{ \ddx{1}, \ldots, \ddx{7} \}$ with dual coframe $\{ \dx{1}, \ldots, \dx{7} \}$ but this is not necessary. Note that the metric dual $1$-form of $e_i$ is $(e_i)^{\flat} = g_{ij} e^j$.

We employ the Einstein summation convention throughout. We write $\Lambda^k$ for the bundle $\Lambda^k (T^* M)$ and $\Omega^k$ for its space of smooth sections $\Gamma (\Lambda^k (T^* M))$. Then $\Lambda^{\bu} = \oplus_{k=1}^n \Lambda^k$ is the exterior algebra of $T^*M$ and $\Omega^{\bu} = \oplus_{k=0}^n \Omega^k$ is the space of smooth differential forms on $M$. Similarly, we use $\sym$ to denote the second symmetric power of $T^* M$, and $\Sym = \Gamma(\sym)$ to denote the space of smooth symmetric $2$-tensors on $M$. 

The Levi-Civita covariant derivative of $g$ is denoted by $\nabla$. Let $\nabla_p = \nabla_{e_p}$. The exterior derivative $d\alpha$ of a $k$-form $\alpha$ can be
written in terms of $\nabla$ as
\begin{equation} \label{eq:dd}
\begin{aligned}
\dd \alpha & = e^p \wedge \nabla_p \alpha, \\
(\dd \alpha)_{i_1 i_1 \cdots i_{k+1}} & = \sum_{j=1}^{k+1} \nabla_{i_j} \alpha_{i_1 \cdots \hat{i_j} \cdots i_k}.
\end{aligned}
\end{equation}
The adjoint $d^{\st}$ of $\dd$ with respect to $g$ satisfies $\ds = (-1)^k \st \dd \st$ on $\Omega^k$. It can be written in terms of $\nabla$ as
\begin{equation} \label{eq:ds}
\begin{aligned}
\ds \alpha & = - g^{pq} e_p \hk \nabla_q \alpha, \\
(\ds \alpha)_{i_1 \cdots i_{k-1}} & = - g^{pq} \nabla_p \alpha_{q i_1 \cdots i_{k-1}}.
\end{aligned}
\end{equation}

An element $h \in \Sym$ can be decomposed as $h = \tfrac{\tr_g h}{7} g + h^0$, where $\tr_g h = g^{ij} h_{ij}$ is the trace, and $h^0$ is the \emph{trace-free} component of $h$, which is orthogonal to $g$. We use $\symo$ to denote the bundle whose sections $\Symo = \Gamma(\symo)$ are the trace-free symmetric $2$-tensors. Finally, if $X$ is a vector field on $M$, we denote by $X^{\flat}$ the $1$-form metric dual to $X$ with respect to the metric $g$. Sometimes we abuse notation and write $X^{\flat}$ as simply $X$ when there is no danger of confusion.

We write $H^k_{\DR}$ for the $k^{\text{th}}$ de Rham cohomology over $\R$ and $\mathcal H^k$ for the space of harmonic $k$-forms. If $[\alpha]$ is a cohomology class, then $|\alpha|$ denotes the degree of any of its representative differential forms. That is, if $[\alpha] \in H^k_{\DR}$, then $|\alpha| = k$.

We use $C^{\bu}$ to denote a $\Z$-graded complex of real vector spaces. A degree $k$ map $P$ of the complex $C^{\bu}$ maps $C^i$ into $C^{i+k}$, and we write
\begin{equation} \label{eq:complexes}
\begin{aligned}
(\ker P)^i & = \ker (P : C^i \to C^{i+k}), \\
(\im P)^i & = \im (P : C^{i-k} \to C^i).
\end{aligned}
\end{equation}

\begin{lemma} \label{lemma:linalg}
We state two trivial results from linear algebra that we use several times in Section~\ref{sec:cohom}.
\begin{enumerate}[(i)]
\item Let $V \subseteq U \subseteq (V \oplus W)$ be nested subspaces. Then $U = V \oplus (W \cap U)$.
\item Let $U = A \oplus B \oplus C$ be a direct sum decomposition of a vector space into complementary subspaces $A, B, C$. Let $V, W$ be subspaces of $U$ such that $V = A' \oplus B' \oplus C'$ and $W = A'' \oplus B'' \oplus C''$ where $A',A''$ are subspaces of $A$, and $B',B''$ are subspaces of $B$, and $C',C''$ are subspaces of $C$. Then $V \cap W = (A' \cap A'') \oplus (B' \cap B'') \oplus (C' \cap C'')$.
\end{enumerate}
\end{lemma}

{\bf Acknowledgments.} These results were obtained in 2017 as part of the collaboration between the COSINE program organized by the Chinese University of Hong Kong and the URA program organized by the University of Waterloo. The authors thank both universities for this opportunity. Part of the writing was done while the second author held a Fields Research Fellowship at the Fields Institute. The second author thanks the Fields Institute for their hospitality. The authors also thank the anonymous referee for pointing out that we had actually also established Theorem~\ref{thm:irreducibleMassey}, which is the stronger version of Corollary~\ref{cor:Massey} in the case of full $\G$ holonomy.

\section{Natural derivations on torsion-free $\G$~manifolds} \label{sec:main}

We first review some facts about torsion-free $\G$~manifolds and the decomposition of the exterior derivative $\dd$. Then we define two derivations on $\Omega^{\bu} $ and discuss their properties.

\subsection{$\G$-structures and the decomposition of $\Omega^{\bu}$} \label{sec:forms}

Let $(M^7, \ph)$ be a manifold with a $\G$-structure. Here $\ph$ is the positive $3$-form associated to the $\G$-structure, and we use $\ps$ to denote the dual $4$-form $\ps = \st \ph$ with respect to the metric $g$ induced by $\ph$. We will use the sign/orientation convention for $\G$-structures of~\cite{K-flows}. In this section we collect some facts about $\G$-structures, taken from~\cite{K-flows}, that we will need. We recall the fundamental relation between $\ph$ and $g$, which allows one to extract the metric from the $3$-form. This is:
\begin{equation} \label{eq:fund-eq}
(X \hk \ph) \wedge (Y \hk \ph) \wedge \ph = - 6 g(X, Y) \vol.
\end{equation}

\begin{lemma} \label{lemma:identities}
The tensors $g$, $\ph$, $\ps$ satisfy the following contraction identities in a local frame:
\begin{align*}
\ph_{ijk} \ph_{abc} g^{kc} & = g_{ia} g_{jb} - g_{ib} g_{ja} - \ps_{ijab}, \allowdisplaybreaks\\
\ph_{ijk} \ph_{abc} g^{jb} g^{kc} & = 6 g_{ia}, \allowdisplaybreaks\\
\ph_{ijk} \ph_{abc} g^{ia} g^{jb} g^{kc} & = 42, \allowdisplaybreaks\\
 \ph_{ijk} \ps_{abcd} g^{kd} & = g_{ia} \ph_{jbc} + g_{ib} \ph_{ajc} + g_{ic} \ph_{abj} - g_{aj} \ph_{ibc} - g_{bj} \ph_{aic} - g_{cj} \ph_{abi}, \allowdisplaybreaks\\
\ph_{ijk} \ps_{abcd} g^{jc} g^{kd} & = - 4 \ph_{iab}, \allowdisplaybreaks\\
\ph_{ijk} \ps_{abcd} g^{ib} g^{jc} g^{kd} & = 0, \allowdisplaybreaks\\
\ps_{ijkl} \ps_{abcd} g^{ld} & = -\ph_{ajk} \ph_{ibc} - \ph_{iak} \ph_{jbc} - \ph_{ija} \ph_{kbc} \allowdisplaybreaks\\
& \qquad {} + g_{ia} g_{jb} g_{kc} + g_{ib} g_{jc} g_{ka} + g_{ic} g_{ja} g_{kb} - g_{ia} g_{jc} g_{kb} - g_{ib} g_{ja} g_{kc} - g_{ic} g_{jb} g_{ka} \allowdisplaybreaks\\
& \qquad {} -g_{ia} \ps_{jkbc} - g_{ja} \ps_{kibc} - g_{ka} \ps_{ijbc} + g_{ab} \ps_{ijkc} - g_{ac} \ps_{ijkb}, \allowdisplaybreaks\\
\ps_{ijkl} \ps_{abcd} g^{kc} g^{ld} & = 4 g_{ia} g_{jb} - 4 g_{ib} g_{ja} - 2 \ps_{ijab}, \allowdisplaybreaks\\
\ps_{ijkl} \ps_{abcd} g^{jb} g^{kc} g^{ld} & = 24 g_{ia}, \allowdisplaybreaks\\
\ps_{ijkl} \ps_{abcd} g^{ia} g^{jb} g^{kc} g^{ld} & = 168.
\end{align*}
\end{lemma}
\begin{proof}
This is proved in Lemmas A.12, A.13, and A.14 of~\cite{K-flows}.
\end{proof}

For $k = 0, \ldots, 7, $ the bundle $\Lambda^k := \Lambda^k (T^* M)$ decomposes as follows:
\begin{equation} \label{eq:bundle-decomp}
\begingroup
\renewcommand*{\arraystretch}{1.2}
\begin{matrix}
\Lambda^0 & = & \Lambda^0_1, \\
\Lambda^1 & = & & & \Lambda^1_7, \\
\Lambda^2 & = & & & \Lambda^2_7 & \oplus & \Lambda^2_{14}, \\
\Lambda^3 & = & \Lambda^3_1 & \oplus & \Lambda^3_7 & & & \oplus & \Lambda^3_{27}, \\
\Lambda^4 & = & \Lambda^4_1 & \oplus & \Lambda^4_7 & & & \oplus & \Lambda^4_{27}, \\
\Lambda^5 & = & & & \Lambda^5_7 & \oplus & \Lambda^5_{14}, \\
\Lambda^6 & = & & & \Lambda^6_7, \\
\Lambda^7 & = & \Lambda^7_1.
\end{matrix}
\endgroup
\end{equation}
Here $\Lambda^k_l$ is a rank $l$ subbundle of $\Lambda^k$, and the decomposition is orthogonal with respect to $g$. Moreover, we have $\Lambda^{7-k}_l = \st \Lambda^k_l$. In fact there are isomorphisms $\Lambda^k_l \cong \Lambda^{k'}_l$, so the bundles in the same vertical column of~\eqref{eq:bundle-decomp} are all isomorphic. Moreover, the Hodge star $\st$ and the operations of wedge product with $\ph$ or with $\ps$ all commute with the projections $\pi_l$ for $l = 1, 7, 14, 27$.

We will denote by $\Omega^k_l$ the space of smooth sections of $\Lambda^k_l$. The isomorphisms $\Lambda^k_l \cong \Lambda^{k'}_l$ induce isomorphisms $\Omega^k_l \cong \Omega^{k'}_l$. The descriptions of the $\Omega^k_l$ and the particular identifications that we choose to use in this paper are given explicitly as follows:

\begin{equation} \label{eq:forms-isom}
\begin{aligned}
\Omega^0_1 & = C^{\infty} (M), \\
\Omega^1_7 & = \Gamma (T^* M) \cong \Gamma (TM), \\
\Omega^2_7 & = \{ X \hk \ph : \, X \in \Gamma(TM) \} \cong \Omega^1_7 \\
\Omega^2_{14} & = \{ \beta \in \Omega^2 : \, \beta \wedge \ps = 0 \} = \{ \beta \in \Omega^2 : \, \beta_{pq} g^{pi} q^{qj} \ph_{ijk} = 0 \}, \\
\Omega^3_1 & = \{ f \ph : \, f \in C^{\infty}(M) \} \cong \Omega^0_1, \\
\Omega^3_7 & = \{ X \hk \ps : \, X \in \Gamma(TM) \} \cong \Omega^1_7, \\
\Omega^3_{27} & = \{ \beta \in \Omega^3 : \, \beta \wedge \ph = 0 \text{ and } \beta \wedge \ps = 0 \} = \{ h_{ip} g^{pk} \dx{i} \wedge (\ddxs{k} \hk \ph) : \, h \in \Symo \} \\
\Omega^k_l & = \{ \st \beta: \, \beta \in \Omega^{7-k}_l \}, \quad \text{for $k = 4, 5, 6, 7$}.
\end{aligned}
\end{equation}

\begin{rmk} \label{rmk:not-isom}
We emphasize that the particular identifications we have chosen in~\eqref{eq:forms-isom} are \emph{not isometric}. Making them isometric identifications would require introducing irrational constant factors but this will not be necessary. See also Remark~\ref{rmk:adjoints}.
\end{rmk}

We will denote by $\pi^k_l$ the orthogonal projection $\pi^k_l : \Omega^k \to \Omega^k_l$. We note for future reference that $\beta \in \Omega^3_1 \oplus \Omega^3_{27}$ if and only if $\beta \perp (X \hk \ps)$ for all $X$, and $\beta \in \Omega^3_7 \oplus \Omega^3_{27}$ if and only if $\beta \perp \ph$. In a local frame these observations are
\begin{equation} \label{eq:omega3stuff}
\begin{aligned}
\beta \in \Omega^3_1 \oplus \Omega^3_{27} \, & \longleftrightarrow \, \beta_{ijk} g^{ia} g^{jb} g^{kc} \ps_{abcd} = 0, \\
\beta \in \Omega^3_7 \oplus \Omega^3_{27} \, & \longleftrightarrow \, \beta_{ijk} g^{ia} g^{jb} g^{kc} \ph_{abc} = 0.
\end{aligned}
\end{equation}
Similarly we have that $\gamma \in \Omega^4_1 \oplus \Omega^4_{27}$ if and only if $\gamma \perp (\ph \wedge X)$ for all $X$, and $\gamma \in \Omega^4_7 \oplus \Omega^4_{27}$ if and only if $\gamma \perp \ps$. In a local frame these observations are
\begin{equation} \label{eq:omega4stuff}
\begin{aligned}
\gamma \in \Omega^4_1 \oplus \Omega^4_{27} \, & \longleftrightarrow \, \gamma_{ijkl} g^{ia} g^{jb} g^{kc} \ph_{abc} = 0, \\
\gamma \in \Omega^4_7 \oplus \Omega^4_{27} \, & \longleftrightarrow \, \gamma_{ijkl} g^{ia} g^{jb} g^{kc} g^{ld} \ps_{abcd} = 0.
\end{aligned}
\end{equation}

\begin{lemma} \label{lemma:identities2}
The following identities hold:
\begin{align*}
\st (\ph \wedge X^{\flat}) & = X \hk \ps, & \st ( \ps \wedge X^{\flat}) & = X \hk \ph, \\
\ps \wedge \st (\ph \wedge X^{\flat}) & = 0, & \ph \wedge \st (\ps \wedge X^{\flat}) & = -2 \ps \wedge X^{\flat}, \\
\ph \wedge ( X \hk \ph) & = -2 \st (X \hk \ph), & \ps \wedge (X \hk \ph) & = 3 \st X^{\flat}, \\
\ph \wedge (X \hk \ps) & = -4 \st X^{\flat}, & \ps \wedge ( X \hk \ps) & = 0.
\end{align*}
\end{lemma}
\begin{proof}
This is part of Proposition A.3 in~\cite{K-flows}. 
\end{proof}

\begin{lemma} \label{lemma:identities3}
Identify $\Omega^1 \cong \Gamma (TM)$ using the metric. The \emph{cross product} $\times : \Omega^1 \times \Omega^1 \to \Omega^1$ is defined by $X \times Y = Y \hk X \hk \ph = \st (X \wedge Y \wedge \ps)$. It satisfies the identity
\begin{equation*}
X \times (X \times Y) = - g(X, X) Y + g(X, Y) X.
\end{equation*}
\end{lemma}
\begin{proof}
This is part of Lemma A.1 in~\cite{K-flows}. 
\end{proof}

In terms of a local frame, we define a map $\ell_{\ph} : \Gamma(T^* M \otimes T^* M) \to \Omega^3$ by
\begin{equation} \label{eq:ellphdefn}
\ell_{\ph} A = A_{ip} g^{pq} e^i \wedge (e_q \hk \ph).
\end{equation}
In components, we have
\begin{equation*}
(\ell_{\ph} A)_{ijk} = A_{ip} g^{pq} \ph_{qjk} + A_{jp} g^{pq} \ph_{iqk} + A_{kp} g^{pq} \ph_{ijq}.
\end{equation*}
Analogous to~\eqref{eq:ellphdefn}, we define $\ell_{\ps} :  \Gamma(T^* M \otimes T^* M) \to \Omega^4$ by
\begin{equation} \label{eq:ellpsdefn}
\ell_{\ps} A = A_{ip} g^{pq} e^i \wedge (e_q \hk \ps).
\end{equation}
In components, we have
\begin{equation*}
(\ell_{\ps} A)_{ijkl} = A_{ip} g^{pq} \ps_{qjkl} + A_{jp} g^{pq} \ps_{iqkl} + A_{kp} g^{pq} \ps_{ijql} + A_{lp} g^{pq} \ps_{ijkq}.
\end{equation*}
It is easy to see that when $A = g$ is the metric, then
\begin{equation} \label{eq:ellg}
\ell_{\ph} g = 3 \ph, \qquad \qquad \ell_{\ps} g = 4 \ps.
\end{equation}

In~\cite[Section 2.2]{K-flows} the map $\ell_{\ph}$ is written as $D$, but we use $\ell_{\ph}$ to avoid confusion with the many instances of `$D$' throughout the present paper to denote various natural linear first order differential operators. We can orthogonally decompose sections of $\Gamma(T^* M \otimes T^* M)$ into symmetric and skew-symmetric parts, which then further orthogonally decompose as
\begin{equation*}
\Gamma(T^* M \otimes T^* M) = \Omega^0_1 \oplus \Symo \oplus \Omega^2_7 \oplus \Omega^2_{14}.
\end{equation*}
 In~\cite[Section 2.2]{K-flows} it is shown that $\ell_{\ph}$ has kernel $\Omega^2_{14}$ and maps $\Omega^0_1$, $\Symo$, and $\Omega^2_7$ isomorphically onto $\Omega^3_1$, $\Omega^3_{27}$, and $\Omega^3_7$, respectively. One can similarly show that $\ell_{\ps}$ has kernel $\Omega^2_{14}$ and maps $\Omega^0_1$, $\Symo$, and $\Omega^2_7$ isomorphically onto $\Omega^4_1$, $\Omega^4_{27}$, and $\Omega^4_7$, respectively. (See also~\cite{KLL} for a detailed proof.) In particular, we note for future references that
\begin{equation} \label{eq:ellon14}
\begin{aligned}
\beta \in \Omega^2_{14} & \iff (\ell_{\ph} \beta)_{ijk} = \beta_{ip} g^{pq} \ph_{qjk} + \beta_{jp} g^{pq} \ph_{iqk} + \beta_{kp} g^{pq} \ph_{ijq} = 0, \\ & \iff (\ell_{\ps} \beta)_{ijkl} = \beta_{ip} g^{pq} \ps_{qjkl} + \beta_{jp} g^{pq} \ps_{iqkl} + \beta_{kp} g^{pq} \ps_{ijql} + \beta_{lp} g^{pq} \ps_{ijkq} = 0.
\end{aligned}
\end{equation}

When restricted to $\Sym$, the map $\ell_{\ph}$ is denoted by $i$ in~\cite{K-flows}. We use $\ell_{\ph}$ rather than $i$, to avoid confusion with the algebraic derivations $\iota_B$ and $\iota_K$ that we introduce later in Section~\ref{sec:LBLK}.

\begin{lemma} \label{lemma:starell}
Let $h \in \Symo$. Then $\st (\ell_{\ph} h) = - \ell_{\ps} h$.
\end{lemma}
\begin{proof}
This is part of Proposition 2.14 in~\cite{K-flows}.
\end{proof}

The next two propositions will be crucial to establish properties of the algebraic derivations $\iota_B$ and $\iota_K$ in Section~\ref{sec:LBLK}.

\begin{prop} \label{prop:special}
Let $h = h_{ij} e^i e^j$ be a \emph{symmetric} $2$-tensor. The following identities hold:
\begin{equation} \label{eq:specialprop}
\begin{aligned}
h^{pq} (e_p \hk \ph) \wedge (e_q \hk \ph) & = - 2 (\tr_g h) \ps + 2 \ell_{\ps} h, \\
h^{pq} (e_p \hk \ph) \wedge (e_q \hk \ps) & = 0, \\
h^{pq} (e_p \hk \ps) \wedge (e_q \hk \ps) & = 0.
\end{aligned}
\end{equation}
\end{prop}
\begin{proof}
Let $\alpha \in \Omega^k$ and $\beta \in \Omega^l$. Then we have
\begin{equation*}
(e_p \hk \alpha) \wedge (e_q \hk \beta) = e_p \hk \big( \alpha \wedge (e_q \hk \beta) \big) - (-1)^k \alpha \wedge (e_p \hk e_q \hk \beta).
\end{equation*}
Since the second term above is skew in $p,q$, when we contract with the symmetric tensor $h^{pq}$ we obtain
\begin{equation} \label{eq:specialtemp}
h^{pq} (e_p \hk \alpha) \wedge (e_q \hk \beta) = h^{pq} e_p \hk \big( \alpha \wedge (e_q \hk \beta) \big).
\end{equation}
We will repeatedly use the identities from Lemma~\ref{lemma:identities2}. When $\alpha = \beta = \ps$ in~\eqref{eq:specialtemp}, we have $\ps \wedge (e_q \hk \ps) = 0$, establishing the third equation in~\eqref{eq:specialprop}. When $\alpha = \ph$ and $\beta = \ps$ in~\eqref{eq:specialtemp}, we have
\begin{equation*}
\ph \wedge (e_q \hk \ps) = - 4 \st (g_{qm} e^m),
\end{equation*}
and hence using that $X \hk (\st \alpha) = - \st (X^{\flat} \wedge \alpha)$ for $\alpha \in \Omega^1$, we find
\begin{align*}
h^{pq} (e_p \hk \ph) \wedge (e_q \hk \ps) & = h^{pq} e_p \hk ( - 4 \st g_{qm} e^m ) = -4 h^{pq} g_{qm} e_p \hk (\st e^m) \\
& = + 4 h^{pq} g_{qm} \st \big( (e_p)^{\flat} \wedge e^m \big) = 4 h^{pq} g_{qm} g_{pl} \st (e^l \wedge e^m) \\
& = 4 h_{lm} \st (e^l \wedge e^m) = 0,
\end{align*}
establishing the second equation in~\eqref{eq:specialprop}. Finally, when $\alpha = \beta = \ph$ in~\eqref{eq:specialtemp}, we have
\begin{equation*}
\ph \wedge (e_q \hk \ph) = - 2 \st (e_q \hk \ph) = - 2 \big( \ps \wedge (e_q)^{\flat} \big) = - 2 g_{qm} e^m \wedge \ps,
\end{equation*}
and hence using~\eqref{eq:ellpsdefn} we find
\begin{align*}
h^{pq} (e_p \hk \ph) \wedge (e_q \hk \ph) & = h^{pq} e_p \hk (- 2 g_{qm} e^m \wedge \ps ) = -2 h^{pq} g_{qm} e_p \hk (e^m \wedge \ps) \\
& = - 2 h^{pq} g_{qm} \delta^m_p \ps + 2 h^{pq} g_{qm} e^m \wedge (e_p \hk \ps) \\
& = - 2 h^{pq} g_{pq} \ps + 2 h_{ml} g^{lp} e^m \wedge (e_p \hk \ps) = - 2 (\tr_g h) \ps + 2 \ell_{\ps} h,
\end{align*}
establishing the first equation in~\eqref{eq:specialprop}.
\end{proof}

\begin{prop} \label{prop:special2}
For any fixed $m$, the following identities hold:
\begin{equation} \label{eq:specialprop2}
\begin{aligned}
g^{pq} (e_p \hk \ph) \wedge (e_q \hk e_m \hk \ph) & = 3 (e_m \hk \ps), \\
g^{pq} (e_p \hk \ph) \wedge (e_q \hk e_m \hk \ps) & = -3 \st (e_m \hk \ps), \\
g^{pq} (e_p \hk \ps) \wedge (e_q \hk e_m \hk \ph) & = -3 \st (e_m \hk \ps), \\
g^{pq} (e_p \hk \ps) \wedge (e_q \hk e_m \hk \ps) & = 4 \st (e_m \hk \ph).
\end{aligned}
\end{equation}
\end{prop}
\begin{proof}
In this proof, we use $e^{ijk}$ to denote $e^i \wedge e^j \wedge e^k$ and similarly for any number of indices. First, we compute
\begin{align*}
g^{pq} (e_p \hk \ph) \wedge (e_q \hk e_m \hk \ph) & = \tfrac{1}{2} (g^{pq} \ph_{pij} \ph_{mqk}) e^{ijk} \\
& = \tfrac{1}{2} ( g_{ik} g_{jm} - g_{im} g_{jk} - \ps_{ijkm}) e^{ijk} \\
& = 0 - 0 - \tfrac{1}{2} \ps_{ijkm} e^{ijk} = 3 (\tfrac{1}{6} \ps_{mijk} e^{ijk}) = 3 (e_m \hk \ps),
\end{align*}
establishing the first equation in~\eqref{eq:specialprop2}.

Similarly we compute
\begin{align*}
g^{pq} (e_p \hk \ph) \wedge (e_q \hk e_m \hk \ps) & = \tfrac{1}{4} (g^{pq} \ph_{pij} \ps_{mqkl}) e^{ijkl} \\
& = \tfrac{1}{4} ( g_{ik} \ph_{jlm} + g_{il} \ph_{kjm} + g_{im} \ph_{klj} - g_{jk} \ph_{ilm} - g_{jl} \ph_{kim} - g_{jm} \ph_{kli} ) e^{ijkl} \\
& = 0 + 0 + \tfrac{1}{4} g_{im} \ph_{klj} e^{ijkl} + 0 + 0 - \tfrac{1}{4} g_{jm} \ph_{kli} e^{ijkl} = \tfrac{1}{2} g_{im} \ph_{jkl} e^{ijkl} \\
& = 3 (g_{mi} e^i) \wedge ( \tfrac{1}{6} \ph_{jkl} e^{jkl} ) = 3 (e_m)^{\flat} \wedge \ph = - 3 \st (e_m \hk \ps),
\end{align*}
establishing the second equation in~\eqref{eq:specialprop2}. Now let $h = g$ in the second equation of~\eqref{eq:specialprop}. Taking the interior product of $g^{pq} (e_p \hk \ph) \wedge (e_q \hk \ps) = 0$ with $e_m$, we obtain
\begin{equation*}
g^{pq} (e_m \hk e_p \hk \ph) \wedge (e_q \hk \ps) + g^{pq} (e_p \hk \ph) \wedge (e_m \hk e_q \hk \ps) = 0,
\end{equation*}
which, after rearrangement and relabelling of indices, becomes
\begin{equation*}
g^{pq} (e_p \hk \ph) \wedge (e_q \hk e_m \hk \ps) = g^{pq} (e_p \hk \ps) \wedge (e_q \hk e_m \hk \ph),
\end{equation*}
establishing the third equation in~\eqref{eq:specialprop2}.

Finally we compute
\begin{align*}
g^{pq} (e_p \hk \ps) \wedge (e_q \hk e_c \hk \ps) & = \tfrac{1}{12} (g^{pq} \ps_{pijk} \ps_{cqab}) e^{ijkab} \\
& = -\tfrac{1}{12} (g^{pq} \ps_{ijkp} \ps_{abcq} ) e^{ijkab} \\ 
& = -\tfrac{1}{12} \bigg( -\ph_{ajk} \ph_{ibc} - \ph_{iak} \ph_{jbc} - \ph_{ija} \ph_{kbc} + g_{ia} g_{jb} g_{kc} + g_{ib} g_{jc} g_{ka} \\
& \qquad \qquad {} + g_{ic} g_{ja} g_{kb} - g_{ia} g_{jc} g_{kb} - g_{ib} g_{ja} g_{kc} - g_{ic} g_{jb} g_{ka}  -g_{ia} \ps_{jkbc} \\
& \qquad \qquad {} - g_{ja} \ps_{kibc} - g_{ka} \ps_{ijbc} + g_{ab} \ps_{ijkc} - g_{ac} \ps_{ijkb} \bigg) e^{ijkab}.
\end{align*}
The first three terms above combine, and all the remaining terms except the last one vanish. Thus using Lemma~\ref{lemma:identities2} we have
\begin{align*}
g^{pq} (e_p \hk \ps) \wedge (e_q \hk e_c \hk \ps) & = -\tfrac{1}{12} ( -3 \ph_{ajk} \ph_{ibc} - g_{ac} \ps_{ijkb} ) e^{ijkab} \\
& = -\tfrac{1}{4} (\ph_{ajk} e^{ajk}) \wedge (\ph_{cib} e^{ib}) - \tfrac{1}{12} (g_{ac} e^a) \wedge (\ps_{ijkb} e^{ijkb}) \\
& = -3 ( \tfrac{1}{6} \ph_{ajk} e^{ajk} ) \wedge ( \tfrac{1}{2} \ph_{cib} e^{ib} ) - 2 (g_{ca} a^a) \wedge ( \tfrac{1}{24} \ps_{ijkb} e^{ijkb} ) \\
& = -3 \ph \wedge (e_c \hk \ph) - 2 (e_c)^{\flat} \wedge \ps = 6 \st (e_c \hk \ph) - 2 \st (e_c \hk \ph) = 4 \st (e_c \hk \ph),
\end{align*}
establishing the fourth equation in~\eqref{eq:specialprop2}.
\end{proof}

For the rest of this section and all of the next section, we assume $(M, \ph)$ is torsion-free. See also Remark~\ref{rmk:when-torsion-free}.

\begin{prop} \label{prop:Liederiv}
Suppose $(M, \ph)$ is a torsion-free $\G$~manifold. Then $\st (\pi_{27} \mathcal L_X \ph) = - \pi_{27} \mathcal L_X \ps$ for any vector field $X$. 
\end{prop}
\begin{proof}
Because $\ph$ and $\ps$ are both parallel, from~\cite[equation (1.7)]{K-flows} we have
\begin{equation*}
(\mathcal L_X \ph) = (\nabla_i X_p) g^{pq} e^i \wedge (e_q \hk \ph), \qquad (\mathcal L_X \ps) = (\nabla_i X_p) g^{pq} e^i \wedge (e_q \hk \ps).
\end{equation*}
Applying $\pi_{27}$ to both of the above expressions and using Lemma~\ref{lemma:starell} yields the desired result.
\end{proof}

\subsection{The exterior derivative $\dd$ and the Hodge Laplacian $\Delta$} \label{sec:dLap}

In this section we analyze the exterior derivative $\dd$ and the Hodge Laplacian $\Delta$ on a manifold with torsion-free $\G$-structure. Much, but not all, of the results in this section have appeared before, without proof, in~\cite[Section 5.2]{Bryant}. See Remark~\ref{rmk:Bryant} for details. Theorem~\ref{thm:harmonic1}, which relates kernels of various operators on $\Omega^1$, is fundamental to the rest of the paper and appears to be new.

We first define three first order operators on torsion-free $\G$~manifolds, which will be used to decompose $\dd : \Omega^k \to \Omega^{k+1}$ into components. More details can be found in~\cite[Section 4]{Knotes}.

\begin{defn} \label{defn:ops}
Let $(M, \ph)$ be a torsion-free $\G$~manifold. We define the following first order linear differential operators:
\begin{align*}
\grad : \Omega^0_1 & \to \Omega^1_7, & f & \mapsto \dd f, \\
\dive : \Omega^1_7 & \to \Omega^0_1, & X & \mapsto - \ds X, \\
\curl : \Omega^1_7 & \to \Omega^1_7, & X & \mapsto \st (\ps \wedge \dd X).
\end{align*}
In a local frame, these operators have the following form:
\begin{equation} \label{eq:opslocal}
(\grad f)_k = \nabla_k f, \qquad \dive X = g^{ij} \nabla_i X_j, \qquad (\curl X)_k = (\nabla_i X_j) g^{ip} g^{jq} \ph_{pqk}.
\end{equation}
\end{defn}

\begin{defn} \label{defn:Doperators}
Denote by $D^{l}_{m}$ the composition
\begin{equation*}
D^l_m : \Omega^k_l \hookrightarrow \Omega^k \xrightarrow{\dd} \Omega^{k+1} \twoheadrightarrow \Omega^{k+1}_{m},
\end{equation*}
where $k$ is the smallest integer such that this composition makes sense. Here the surjection is the projection $\pi^{k+1}_m$. That is, $D^l_m = \pi^{k+1}_m \circ {\left. {\dd} \right|}_{\Omega^k_l}$.
\end{defn}

\begin{prop} \label{prop:vanishingDs}
The operators $D^1_1$, $D^1_{14}$, $D^{14}_1$, $D^1_{27}$, $D^{27}_1$, and $D^{14}_{14}$ are all zero.
\end{prop}
\begin{proof}
It is clear from~\eqref{eq:bundle-decomp} that $D^{14}_{14} = 0$. The operators $D^1_1 : \Omega^3_1 \to \Omega^4_1$ and $D^1_{27} : \Omega^3_1 \to \Omega^4_{27}$ are both zero because $\dd (f \ph) = (\dd f) \wedge \ps \in \Omega^4_7$. Similarly, since $\dd (f \ps) = (\dd f) \wedge \ps \in \Omega^5_7$, we also have $D^1_{14} = 0$. If $\beta \in \Omega^2_{14}$, then $\beta \wedge \ps = 0$, so $(\dd \beta) \wedge \ps = 0$, and thus $\pi_1 (\dd \beta) = 0$, hence $D^{14}_1 = 0$. Similarly if $\beta \in \Omega^3_{27}$, then $\beta \wedge \ph = 0$, so $(\dd \beta) \wedge \ph = 0$, and thus $\pi_1 (\dd \beta) = 0$, hence $D^{27}_1 = 0$.
\end{proof}

\begin{figure}[H]
\begin{equation*}
\xymatrix{
\Omega^0_1 \ar@[red][drr]^{\red{D^1_7}} & & & & & & & & \\
& & \Omega^1_7 \ar@[orange][d]^{\orange{D^7_7}} \ar@[brown][drr]^{\brown{D^7_{14}}} & & & & & & \\
& & \Omega^2_7 \ar@[blue][dll]^{\blue{D^7_1}} \ar@/_/@[magenta][drrrr]^{\magenta{D^7_{27}}} \ar@[orange][d]^{\orange{-\frac{3}{2}D^7_7}} & & \Omega ^2_{14} \ar@[cyan][lld]^{\cyan{D^{14}_7}} \ar@[violet][rrd]^{\violet{D^{14}_{27}}} & & & & \\
\Omega^3_1 \ar@/_/@[red][drr]_{\red{-D^1_7}} & & \Omega^3_7 \ar@/_/@[blue][dll]_{\blue{\frac{4}{3}D^7_1}} \ar@[orange][d]^{\orange{-\frac{3}{2} D^7_7}} \ar@/^/@[magenta][drrrr]^{\magenta{-D^7_{27}}} & & & & \Omega^3_{27} \ar@/^/@[teal][dllll]^{\teal{D^{27}_7}} \ar@[olive][d]^{\olive{D^{27}_{27}}} & & \\
\Omega^4_1 \ar@[red][drr]^{\red{D^1_7}} & & \Omega^4_7 \ar@[orange][d]^{\orange{2 D^7_7}} \ar@[brown][drr]^{\brown{-D^7_{14}}} & & & & \Omega^4_{27} \ar@/_/@[teal][dllll]^{\teal{\frac{4}{3} D^{27}_7}} \ar@[gray][dll]^{\gray{D^{27}_{14}}} & \\
& & \Omega^5_7 \ar@[orange][d]^{\orange{3 D^7_7}} & & \Omega ^5_{14} \ar@[cyan][dll]^{\cyan{4 D^{14}_7}} & & & & \\
& & \Omega^6_7 \ar@[blue][dll]^{\blue{\frac{7}{3} D^7_1}} & & & & & & \\
\Omega^7_1 & & & & & & & &
} 
\end{equation*}
\caption{Decomposition of the exterior derivative $\dd$ into components} \label{figure:d}
\end{figure}

\begin{prop} \label{prop:dfigure}
With respect to the identifications described in~\eqref{eq:forms-isom}, the components of the exterior derivative $\dd$ satisfy the relations given in Figure~\ref{figure:d}.
\end{prop}
\begin{proof}
We will use repeatedly the contraction identities of Lemma~\ref{lemma:identities} and the descriptions~\eqref{eq:forms-isom} of the $\Omega^k_l$ spaces.

(i) We establish the relations for $\pi_7 \dd \pi_1 : \Omega^k_1 \to \Omega^{k+1}_7$ for $k=0,3,4$.

$k=0$: Let $f \in \Omega^0_1$. By Definition~\ref{defn:Doperators}, we have $D^1_7 f = \dd f$.

$k=3$: Let $\beta = f \ph \in \Omega^3_1$. Since $\dd \beta = (\dd f) \wedge \ph \in \Omega^4_7$, we have $\pi_7 (\dd \beta) = - \ph \wedge (\dd f) = - \st ( (\dd f) \hk \ps)$, so $\pi_7 \dd \pi_1 : \Omega^3_1 \to \Omega^4_7$ is identified with $-D^1_7$.

$k=4$: Let $\gamma = f \ps \in \Omega^4_1$. Since $\dd \gamma = (\dd f) \wedge \ps \in \Omega^5_7$, we have $\pi_7 (\dd \gamma) = \ps \wedge (\dd f) = \st ( (\dd f) \hk \ph)$, so $\pi_7 \dd \pi_1 : \Omega^4_1 \to \Omega^5_7$ is identified with $D^1_7$.

(ii) We establish the relations for $\pi_1 \dd \pi_7 : \Omega^k_7 \to \Omega^{k+1}_1$ for $k=2,3,6$.

$k=2$: Let $\alpha = X \hk \ph \in \Omega^2_7$. Then $(\pi_1 \dd \alpha)_{ijk} = f \ph_{ijk}$ for some function $f$. Using~\eqref{eq:omega3stuff} we compute
\begin{align*}
(\dd \alpha)_{ijk} g^{ia} g^{jb} g^{kc} \ph_{abc} & = (\pi_1 \dd \alpha)_{ijk} g^{ia} g^{jb} g^{kc} \ph_{abc} = f \ph_{ijk} g^{ia} g^{jb} g^{kc} \ph_{abc}  = 42 f, \\
& = ( \nabla_i \alpha_{jk} + \nabla_j \alpha_{ki} + \nabla_k \alpha_{ij} ) g^{ia} g^{jb} g^{kc} \ph_{abc} = 3 (\nabla_i \alpha_{jk}) g^{ia} g^{jb} g^{kc} \ph_{abc},
\end{align*}
and thus $f = \tfrac{3}{42} (\nabla_i \alpha_{jk}) g^{ia} g^{jb} g^{kc} \ph_{abc}$. Substituting $\alpha_{jk} = X^m \ph_{mjk}$ we obtain
\begin{equation*}
f = \tfrac{3}{42} (\nabla_i X^m) \ph_{mjk} \ph_{abc} g^{ia} g^{jb} g^{kc} = \tfrac{18}{42} (\nabla_i X^m) g^{ia} g_{ma} = \tfrac{3}{7} \nabla_i X^i,
\end{equation*}
and comparing with Definition~\ref{defn:ops} we find that
\begin{equation} \label{eq:D71-1}
D^7_1 X = \pi_1 \dd (X \hk \ph) = f \ph = (\tfrac{3}{7} \dive X) \ph.
\end{equation}

$k=3$: Let $\beta = X \hk \ps \in \Omega^3_7$. Then $(\pi_1 \dd \beta)_{ijkl} = f \ps_{ijkl}$ for some function $f$. Using~\eqref{eq:omega4stuff} we compute
\begin{align*}
(\dd \beta)_{ijkl} g^{ia} g^{jb} g^{kc} g^{ld} \ps_{abcd} & = (\pi_1 \dd \beta)_{ijkl} g^{ia} g^{jb} g^{kc} g^{ld} \ps_{abcd} = f \ps_{ijkl} g^{ia} g^{jb} g^{kc} g^{ld} \ps_{abcd}  = 168 f, \\
& = ( \nabla_i \beta_{jkl} - \nabla_j \beta_{ikl} + \nabla_k \beta_{ijl} - \nabla_l \beta_{ijk} ) g^{ia} g^{jb} g^{kc} g^{ld} \ps_{abcd} \\
& = 4 (\nabla_i \beta_{jkl}) g^{ia} g^{jb} g^{kc} g^{ld} \ps_{abcd},
\end{align*}
and thus $f = \tfrac{4}{168} (\nabla_i \beta_{jkl}) g^{ia} g^{jb} g^{kc} g^{ld} \ps_{abcd}$. Substituting $\beta_{jkl} = X^m \ps_{mjkl}$ we obtain
\begin{equation*}
f = \tfrac{4}{168} (\nabla_i X^m) \ps_{mjkl} \ps_{abcd} g^{ia} g^{jb} g^{kc} g^{ld} = \tfrac{4 \cdot 24}{168} (\nabla_i X^m) g^{ia} g_{ma} = \tfrac{4}{7} \nabla_i X^i,
\end{equation*}
and comparing with~\eqref{eq:D71-1} we find that $\pi_1 \dd \pi_7 : \Omega^3_7 \to \Omega^4_1$ is identified with $\tfrac{4}{3} D^7_1$.

$k=6$: Let $\st X \in \Omega^6_7$. Then $\pi_1 \dd (\st X) = \dd \st X = \st^2 \dd \st X = - \st (\ds X) = - (\ds X) \vol$, where we have used $\ds = - \st \dd \st$ on odd forms. Comparing with Definition~\ref{defn:ops} and~\eqref{eq:D71-1} we find that $\pi_1 \dd \pi_7 : \Omega^6_7 \to \Omega^7_1$ is identified with $\tfrac{7}{3} D^7_1$.

(iii) We establish the relations for $\pi_7 \dd \pi_7 : \Omega^k_7 \to \Omega^{k+1}_7$ for $k=1,2,3,4,5$.

$k=1$: Let $X \in \Omega^1_7$. Then $(\pi_7 \dd X)_{ij} = Y^m \ph_{mij}$ for some vector field $Y$. We compute
\begin{align*}
(\dd X)_{ij} g^{ia} g^{jb} \ph_{kab} & = (\pi_7 \dd X)_{ij} g^{ia} g^{jb} \ph_{kab} = Y^m \ph_{mij} g^{ia} g^{jb} \ph_{kab}  = 6 Y_k, \\
& = ( \nabla_i X_j - \nabla_j X_i ) g^{ia} g^{jb} \ph_{kab} = 2 (\nabla_i X_j) g^{ia} g^{jb} \ph_{abk},
\end{align*}
from which it follows from Definition~\ref{defn:Doperators} that
\begin{equation} \label{eq:D77-1}
D^7_7 X = \pi_7 \dd X = Y = \tfrac{1}{3} \curl X.
\end{equation}

$k=2$: Let $\alpha = X \hk \ph \in \Omega^2_7$. Then $(\pi_7 \dd \alpha)_{ijk} = Y^m \ps_{mijk}$ for some vector field $Y$. Using~\eqref{eq:omega3stuff} we compute
\begin{align*}
(\dd \alpha)_{ijk} g^{ia} g^{jb} g^{kc} \ps_{labc} & = (\pi_7 \dd \alpha)_{ijk} g^{ia} g^{jb} g^{kc} \ps_{labc} = Y^m \ps_{mijk} g^{ia} g^{jb} g^{kc} \ps_{labc}  = 24 Y_l, \\
& = ( \nabla_i \alpha_{jk} + \nabla_j \alpha_{ki} + \nabla_k \alpha_{ij} ) g^{ia} g^{jb} g^{kc} \ps_{labc} = 3 (\nabla_i \alpha_{jk}) g^{ia} g^{jb} g^{kc} \ps_{labc},
\end{align*}
and thus $Y_l = \tfrac{1}{8} (\nabla_i \alpha_{jk}) g^{ia} g^{jb} g^{kc} \ps_{labc}$. Substituting $\alpha_{jk} = X^m \ph_{mjk}$ we obtain
\begin{equation*}
Y_l = \tfrac{1}{8} (\nabla_i X^m) \ph_{mjk} \ps_{labc} g^{ia} g^{jb} g^{kc} = -\tfrac{4}{8} (\nabla_i X^m) g^{ia} \ph_{mla} = - \tfrac{1}{2} \curl X,
\end{equation*}
and comparing with~\eqref{eq:D77-1} we find that $\pi_7 \dd \pi_7 : \Omega^2_7 \to \Omega^3_7$ is identified with $-\tfrac{3}{2} D^7_7$.

$k=3$: Let $\beta = X \hk \ps \in \Omega^3_7$. Then $\pi_7 (\dd \beta) = \st (Y \hk \ps) = \ph \wedge Y$ for some vector field $Y$. We have $(\pi_7 \dd \beta)_{ijkl} = \ph_{ijk} Y_l - \ph_{ijl} Y_k + \ph_{ikl} Y_j - \ph_{jkl} Y_i$. Using~\eqref{eq:omega3stuff} we compute
\begin{align*}
(\dd \beta)_{ijkl} g^{ia} g^{jb} g^{kc} \ph_{abc} & = (\pi_7 \dd \beta)_{ijkl} g^{ia} g^{jb} g^{kc} \ph_{abc} \\
& = (\ph_{ijk} Y_l - \ph_{ijl} Y_k + \ph_{ikl} Y_j - \ph_{jkl} Y_i) g^{ia} g^{jb} g^{kc} \ph_{abc} \\
& = 42 Y_l - 3 \ph_{ijl} Y_k  g^{ia} g^{jb} g^{kc} \ph_{abc} = 42 Y_l - 3(6 Y_k g^{kc} g_{lc}) = 24 Y_l.
\end{align*}
But we also have
\begin{align*}
(\dd \beta)_{ijkl} g^{ia} g^{jb} g^{kc} \ph_{abc} & = ( \nabla_i \beta_{jkl} - \nabla_j \beta_{ikl} + \nabla_k \beta_{ijl} - \nabla_l \beta_{ijk} ) g^{ia} g^{jb} g^{kc} \ph_{abc} \\
& = 3 (\nabla_i \beta_{jkl}) g^{ia} g^{jb} g^{kc} \ph_{abc} - (\nabla_l \beta_{ijk}) g^{ia} g^{jb} g^{kc} \ph_{abc}.
\end{align*}
Substituting $\beta_{ijk} = X^m \ps_{mijk}$ we obtain
\begin{align*}
(\dd \beta)_{ijkl} g^{ia} g^{jb} g^{kc} \ph_{cab} & = 3 (\nabla_i X^m) \ps_{mljk} g^{ia} g^{jb} g^{kc} \ph_{abc} - (\nabla_l X^m) \ps_{mijk} g^{ia} g^{jb} g^{kc} \ph_{abc} \\
& = - 12 (\nabla_i X^m) g^{ia} \ph_{aml} - 0,
\end{align*}
and thus $Y_l = -\tfrac{12}{24} (\nabla_i X^m) g^{ia} \ph_{aml} = - \tfrac{1}{2} \curl X$. Comparing with~\eqref{eq:D77-1} we find that $\pi_7 \dd \pi_7 : \Omega^3_7 \to \Omega^4_7$ is identified with $-\tfrac{3}{2} D^7_7$.

$k=4$: Let $\gamma = \st( X \hk \ps) = \ph \wedge X \in \Omega^4_7$. Then $\pi_7 (\dd \gamma) = \pi_7 \dd (\ph \wedge X) = -\pi_7 (\ph \wedge \dd X) = \st( Y \hk \ph)$ for some vector field $Y$. We compute
\begin{equation*}
\st( Y \hk \ph ) = -\pi_7 (\ph \wedge \dd X)  = - \ph \wedge (\pi_7 \dd X) = 2 \st (\pi_7 \dd X).
\end{equation*}
Comparing with~\eqref{eq:D77-1} we find that $\pi_7 \dd \pi_7 : \Omega^4_7 \to \Omega^5_7$ is identified with $2 D^7_7$.

$k=5$: Let $\eta = \st( X \hk \ph) = \ps \wedge X \in \Omega^5_7$. Then $\pi_7 (\dd \eta) = \dd \eta = \dd (\ps \wedge X) = \ps \wedge \dd X = \st Y$ for some vector field $Y$. Using Definition~\ref{defn:ops}, we compute
\begin{equation*}
Y = \st( \ps \wedge \dd X) = \curl X.
\end{equation*}
Comparing with~\eqref{eq:D77-1} we find that $\pi_7 \dd \pi_7 : \Omega^5_7 \to \Omega^6_7$ is identified with $3 D^7_7$.

(iv) We establish the relations for $\pi_{14} \dd \pi_7 : \Omega^k_7 \to \Omega^{k+1}_{14}$ for $k=1,4$.

$k=1$: Let $X \in \Omega^1_7$. By definition, we have $\pi_{14} \dd X = D^7_{14} X$.

$k=4$. Let $\gamma = \st (X \hk \ps) = \ph \wedge X \in \Omega^4_7$. Then $\dd \gamma = - \ph \wedge (\dd X)$, so $\pi_{14} \dd \gamma = - \pi_{14} (\ph \wedge \dd X) = - \ph \wedge (\pi_{14} \dd X) = - \st (\pi_{14} \dd X)$. Thus we find that $\pi_{14} \dd \pi_7 : \Omega^4_7 \to \Omega^5_{14}$ is identified with $- D^7_{14}$.

(v) We establish the relations for $\pi_7 \dd \pi_{14} : \Omega^k_{14} \to \Omega^{k+1}_7$ for $k=2,5$.

$k=1$: Let $\alpha \in \Omega^2_{14}$. By definition, we have $\pi_7 \dd \alpha = D^{14}_7 X$.

$k=4$. Let $\eta = \st \beta \in \Omega^5_{14}$ where $\beta \in \Omega^2_{14}$. We have $\st \beta = \ph \wedge \beta$, so $\pi_7 \dd (\st \beta) = \dd (\st \beta) = - \ph \wedge \dd \beta \in \Omega^6_7$. We can write $\pi_7 \dd \beta = Y \hk \ps \in \Omega^3_7$ for some vector field $Y$. Then using Lemma~\ref{lemma:identities2} we find $\pi_7 \dd (\st \beta) = - \ph \wedge \dd \beta = - \ph \wedge (\pi_7 \dd \beta) = - \ph \wedge (Y \hk \ps) = 4 \st Y$. Thus we find that $\pi_7 \dd \pi_{14} : \Omega^5_{14} \to \Omega^6_7$ is identified with $4 D^{14}_7$.

(vi) We establish the relations for $\pi_{27} \dd \pi_7 : \Omega^k_7 \to \Omega^{k+1}_{27}$ for $k=2,3$.

$k=2$: Let $\alpha \in \Omega^2_7$. By definition, we have $\pi_{27} \dd \alpha = D^7_{27} \alpha$.

$k=3$: Let $\beta = X \hk \ps \in \Omega^4_7$. Then $\pi_{27} \dd \beta = \pi_{27} \dd (X \hk \ps) = \pi_{27} \mathcal L_X \ps$. Consider $\alpha = X \hk \ph$. Then similarly we have $\pi_{27} \dd \alpha = \pi_{27} \mathcal L_X \ph$. By Proposition~\ref{prop:Liederiv}, we have $\pi_{27} \dd (X \hk \ps) = - \st (\pi_{27} \dd (X \hk \ph))$. Thus we find that $\pi_{27} \dd \pi_7 : \Omega^3_{27} \to \Omega^4_7$ is identified with $- D^7_{27}$.

(vii) We establish the relations for $\pi_7 \dd \pi_{27} : \Omega^k_{27} \to \Omega^{k+1}_7$ for $k=3,4$.

$k=3$: Let $\beta = \ell_{\ph} h \in \Omega^3_{27}$, where $h \in \Symo$. Then $\pi_7 (\dd \beta) = \st(Y \hk \ps) = \ph \wedge Y$ for some vector field $Y$. We have $(\pi_7 \dd \beta)_{ijkl} = \ph_{ijk} Y_l - \ph_{ijl} Y_k + \ph_{ikl} Y_j - \ph_{jkl} Y_i$. Using~\eqref{eq:omega4stuff} we compute
\begin{align*}
(\dd \beta)_{ijkl} g^{ia} g^{jb} g^{kc} \ph_{abc} & = (\pi_7 \dd \beta)_{ijkl} g^{ia} g^{jb} g^{kc} \ph_{abc} \\
& = (\ph_{ijk} Y_l - \ph_{ijl} Y_k + \ph_{ikl} Y_j - \ph_{jkl} Y_i) g^{ia} g^{jb} g^{kc} \ph_{abc} \\
& = 42 Y_l - 3 \ph_{ijl} Y_k  g^{ia} g^{jb} g^{kc} \ph_{abc} = 42 Y_l - 3(6 Y_k g^{kc} g_{lc}) = 24 Y_l.
\end{align*}
But we also have
\begin{align*}
(\dd \beta)_{ijkl} g^{ia} g^{jb} g^{kc} \ph_{abc} & = ( \nabla_i \beta_{jkl} - \nabla_j \beta_{ikl} + \nabla_k \beta_{ijl} - \nabla_l \beta_{ijk} ) g^{ia} g^{jb} g^{kc} \ph_{abc} \\
& = 3 (\nabla_i \beta_{jkl}) g^{ia} g^{jb} g^{kc} \ph_{abc} - (\nabla_l \beta_{ijk}) g^{ia} g^{jb} g^{kc} \ph_{abc}.
\end{align*}
Substituting $\beta_{ijk} = h_{ip} g^{pq} \ph_{qjk} + h_{jp} g^{pq} \ph_{qki} + h_{kp} g^{pq} \ph_{qjk}$ we obtain
\begin{align*}
24 Y_l & = (\dd \beta)_{ijkl} g^{ia} g^{jb} g^{kc} \ph_{cab} \\
& = 3 (\nabla_i (h_{jp} g^{pq} \ph_{qkl} + h_{kp} g^{pq} \ph_{qlj} + h_{lp} g^{pq} \ph_{qjk})) g^{ia} g^{jb} g^{kc} \ph_{abc} \\
& \qquad {} - (\nabla_l (h_{ip} g^{pq} \ph_{qjk} + h_{jp} g^{pq} \ph_{qki} + h_{kp} g^{pq} \ph_{qjk})) g^{ia} g^{jb} g^{kc} \ph_{abc} \\
& = 6 (\nabla_i h_{jp}) g^{pq} g^{ia} g^{jb} (g^{kc} \ph_{lqk} \ph_{abc}) + 3 (\nabla_i h_{lp}) g^{pq} g^{ia} (g^{jb} g^{kc} \ph_{qjk} \ph_{abc}) \\
& \qquad {} - 3 (\nabla_l h_{ip}) g^{pq} g^{ia} (g^{jb} g^{kc} \ph_{qjk} \ph_{abc}).
\end{align*}
We further simplify this as
\begin{align*}
24 Y_l & = 6 (\nabla_i h_{jp}) g^{pq} g^{ia} g^{jb} (g_{la} g_{qb} - g_{lb} g_{qa} - \ps_{lqab}) + 3 (\nabla_i h_{lp}) g^{pq} g^{ia} (6 g_{qa}) - 3 (\nabla_l h_{ip}) g^{pq} g^{ia} (6 g_{qa}) \\
& = 6 (\nabla_l h_{jp}) g^{jp} - 6 (\nabla_i h_{lp} )g ^{ip} - 0 + 18 (\nabla_i h_{lp}) g^{ip} - 18 (\nabla_l h_{ip}) g^{ip} \\
& = 6 \nabla_l (\tr h) - 6 (\nabla_{i} h_{jl}) g^{ij} + 18 (\nabla_i h_{jl}) g^{ij} - 18 \nabla_l (\tr h) = 12 g^{ij} (\nabla_{i} h_{jl}),
\end{align*}
and thus $Y_l = \tfrac{1}{2} g^{ij} (\nabla_{i} h_{jl})$. It follows from Definition~\ref{defn:Doperators} that
\begin{equation} \label{eq:D277-1}
D^{27}_7 h = \pi_7 \dd (\ell_{\ph} h) = \st ( Y \hk \ps), \qquad \text{ where $Y_l = \tfrac{1}{2} g^{ij} (\nabla_{i} h_{jl})$}.
\end{equation}

$k=4$: Let $\gamma = \st (\ell_{\ph} h) \in \Omega^4_{27}$, where $h \in \Symo$. Then $\pi_7 (\dd \gamma) = \st(Y \hk \ph)$ for some vector field $Y$. Taking Hodge star of both sides, we have $Y \hk \ph = \st \pi_7 (\dd \st \ell_{\ph} h) = \pi_7 \st \dd \st (\ell_{\ph} h) = -\pi_7 \ds (\ell_{\ph} h)$. Thus we have
\begin{align*}
-(\ds (\ell_{\ph} h))_{ij} g^{ia} g^{jb} \ph_{kab} & = -(\pi_7 \ds (\ell_{\ph} h))_{ij} g^{ia} g^{jb} \ph_{kab} = (Y \hk \ph)_{ij} g^{ia} g^{jb} \ph_{kab} \\
& = Y^m \ph_{mij} g^{ia} g^{jb} \ph_{kab} = 6 Y_k.
\end{align*}
But we also have
\begin{align*}
-(\ds (\ell_{\ph} h))_{ij} g^{ia} g^{jb} \ph_{kab} & = g^{pq} (\nabla_p (\ell_{\ph} h)_{qij}) g^{ia} g^{jb} \ph_{kab} \\
& = g^{pq} (\nabla_p (h_{ql} g^{lm} \ph_{mij} + h_{il} g^{lm} \ph_{mjq} + h_{jl} g^{lm} \ph_{mqi})) g^{ia} g^{jb} \ph_{kab} \\
& = g^{pq} (\nabla_p h_{ql}) g^{lm} (g^{ia} g^{jb} \ph_{mij} \ph_{kab}) + 2 g^{pq} (\nabla_p h_{il}) g^{lm} g^{ia} (g^{jb} \ph_{qmj} \ph_{kab}) \\
& = g^{pq} (\nabla_p h_{ql}) g^{lm} (6 g_{mk}) + 2 g^{pq} (\nabla_p h_{il}) g^{lm} g^{ia} (g_{qk} g_{ma} - g_{qa} g_{mk} - \ps_{qmka}) \\
& = 6 g^{pq} (\nabla_p h_{qk}) + 2 \nabla_k (\tr h) - 2 g^{ip} (\nabla_p h_{ik}) - 0 = 4 g^{ij} (\nabla_{i} h_{jk})
\end{align*}
Thus we have $Y_k = \tfrac{2}{3} g^{ij} (\nabla_{i} h_{jk}) = \tfrac{4}{3} (\tfrac{1}{2} g^{ij} (\nabla_{i} h_{jk}))$. Comparing with~\eqref{eq:D277-1} we find that $\pi_7 \dd \pi_{27} : \Omega^4_{27} \to \Omega^5_7$ is identified with $\tfrac{4}{3} D^{27}_7$.
\end{proof}

\begin{cor} \label{cor:d-relations}
The operators of Definition~\ref{defn:Doperators} satisfy the following fourteen relations:
\begin{equation} \label{eq:d-relations}
\begin{aligned}
D^7_7 D^1_7 & = 0, \qquad \qquad & D^7_{14} D^1_7 & = 0, \\
D^7_1 D^7_7 & = 0, \qquad \qquad & \tfrac{3}{2} D^7_7 D^7_7 - D^{14}_7 D^7_{14} & = 0, \\
- D^1_7 D^7_1 + \tfrac{9}{4} D^7_7 D^7_7 + D^{27}_7 D^7_{27} & = 0, \qquad \qquad & \tfrac{3}{2} D^7_{14} D^7_7 - D^{27}_{14} D^7_{27} & = 0, \\
\tfrac{3}{2} D^7_{27} D^7_7 + D^{27}_{27} D^7_{27} & = 0, \qquad \qquad & D^7_{27} D^7_7 + D^{14}_{27} D^7_{14} & = 0, \\
D^7_1 D^{14}_7 & = 0, \qquad \qquad & \tfrac{3}{2} D^7_7 D^{14}_7 - D^{27}_7 D^{14}_{27} & = 0, \\
D^7_{27} D^{14}_7 - D^{27}_{27} D^{14}_{27} & = 0, \qquad \qquad & D^7_7 D^{27}_7 + D^{14}_7 D^{27}_{14} & = 0, \\
\tfrac{3}{2} D^7_7 D^{27}_7 + D^{27}_7 D^{27}_{27} & = 0, \qquad \qquad & D^7_{14} D^{27}_7 - D^{27}_{14} D^{27}_{27} & = 0.
\end{aligned}
\end{equation}
\end{cor}
\begin{proof}
These relations all follow from Figure~\ref{figure:d} and the fact that $\dd^2 = 0$, by computing $\pi_{l'} \dd^2 \pi_{l} : \Omega^k_l \to \Omega^{k+2}_{l'}$ for all $l, l' \in \{ 1, 7, 14, 27 \}$ and all $k = 0, \ldots, 5$. Some of the relations arise multiple times this way. Moreover, there are \emph{two distinct relations} for $(l, l') = (7,7)$, $(7,27)$, and $(27,7)$.
\end{proof}

\begin{cor} \label{cor:adjoints}
Consider the maps $D^l_m : \Omega^k_l \to \Omega^{k+1}_m$ introduced in Definition~\ref{defn:Doperators}. Recall these were only defined for the smallest integer $k$ where the composition makes sense. The formal adjoint is a map $(D^l_m)^* : \Omega^{k+1}_m \to \Omega^l_m$. With respect to the identifications described in~\eqref{eq:forms-isom}, these adjoint maps are given by
\begin{equation} \label{eq:adjoints}
\begin{aligned}
(D^1_7)^* & = -\tfrac{7}{3} D^7_1, \qquad & (D^7_7)^* & = 3 D^7_7, \qquad & (D^7_{14})^* & = 4 D^{14}_7, \\
(D^7_1)^* & = - D^1_7, \qquad & (D^7_{27})^* & = - \tfrac{4}{3} D^{27}_7, \qquad & (D^{14}_7)^* & = D^7_{14}, \\
(D^{14}_{27})^* & = - D^{27}_{14}, \qquad & (D^{27}_7)^* & = - D^7_{27}, \qquad & (D^{27}_{27})^* & = D^{27}_{27}, \\
(D^{27}_{14})^* & = - D^{14}_{27}.
\end{aligned}
\end{equation} 
\end{cor}
\begin{proof}
These follow from Figure~\ref{figure:d} and the facts that $\ds = (-1)^k \st \dd \st$ on $\Omega^k$ and that $\st$ is compatible with the identifications given in~\eqref{eq:forms-isom}.
\end{proof}

\begin{rmk} \label{rmk:adjoints}
One has to be very careful with the ``equations'' in~\eqref{eq:adjoints}. In particular, taking the adjoint of both sides of an equation in~\eqref{eq:adjoints} in general violates $P^{**} = P$. This is because these are not really \emph{equalities}, but identifications, and recall that unfortunately the identifications in~\eqref{eq:bundle-decomp} are not isometries, as explained in Remark~\ref{rmk:not-isom}. However, this will not cause us any problems, because the notation $D^l_m$ will always only refer to the maps introduced in Definition~\ref{defn:Doperators}, and we will never have need to consider the adjoints of any other components of $\dd$.
\end{rmk}

We can now describe the Hodge Laplacian $\Delta = \dd \ds + \ds \dd$ on each summand $\Omega^k_l$ in terms of the operators of Definition~\ref{defn:Doperators}.

\begin{prop} \label{prop:Laplacian}
On $\Omega^k_l$, the Hodge Laplacian $\Delta$ can be written as follows:
\begin{equation} \label{eq:Laplacian}
\begin{aligned}
\left. \Delta \right|_{\Omega^k_1} & = -\tfrac{7}{3} D^7_1 D^1_7 & & \text{for $k=0,3,4,7$}, \\
\left. \Delta \right|_{\Omega^k_7} & = 9 D^7_7 D^7_7 - \tfrac{7}{3} D^1_7 D^7_1 & & \text{for $k=1,2,3,4,5,6$}, \\
\left. \Delta \right|_{\Omega^k_{14}} & = 5 D^7_{14} D^{14}_7 - D^{27}_{14} D^{14}_{27} & & \text{for $k=2,5$}, \\
\left. \Delta \right|_{\Omega^k_{27}} & = -\tfrac{7}{3} D^7_{27} D^{27}_7 - D^{14}_{27} D^{27}_{14} + (D^{27}_{27})^2 & & \text{for $k=3,4$}.
\end{aligned}
\end{equation}
\end{prop}
\begin{proof}
Recall that $\ds = (-1)^k \st \dd \st$ on $\Omega^k$ and that $\st$ is compatible with the identifications given in~\eqref{eq:forms-isom}. The expressions in~\eqref{eq:Laplacian} can be checked on a case-by-case basis using these facts, Figure~\ref{figure:d}, and the relations in Corollary~\ref{cor:d-relations}. Note that one can show from general principles that $\Delta_{\dd}$ preserves the splittings~\eqref{eq:bundle-decomp} when $\ph$ is parallel, which we always assume. (See~\cite{Joyce} for details.) However, the proof of the present proposition gives an explicit verification of this fact, viewing it as a consequence of the fundamental relations~\eqref{eq:d-relations}.
\end{proof}

\begin{rmk} \label{rmk:TF}
We emphasize that for Proposition~\ref{prop:dfigure}, Corollary~\ref{cor:d-relations}, and Propostion~\ref{prop:Laplacian}, the torsion-free assumption is essential, as the proofs frequently made use of $\nabla \ph = \nabla \ps = \dd \ph = \dd \ps = 0$. For $\G$-structures with torsion, there would be many additional terms involving torsion, and in particular the Laplacian $\Delta$ would \emph{not} preserve the splittings~\eqref{eq:bundle-decomp}. See also Remark~\ref{rmk:when-torsion-free}.
\end{rmk}

\begin{rmk} \label{rmk:Bryant}
As mentioned in the introduction, the results of Proposition~\ref{prop:dfigure}, Corollary~\ref{cor:d-relations}, and Propostion~\ref{prop:Laplacian} have appeared before in~\cite[Section 5.2, Tables 1--3]{Bryant}, where Bryant says the results follow by routine computation. We have presented all the details for completeness and for readers to be able to use the computational techniques for possible future applications. Note that one has to be careful to compare our results with~\cite{Bryant}. First, we use a different orientation convention, which effectively replaces $\st$ by $-\st$ and $\ps$ by $-\ps$, although Bryant denotes the $3$-form by $\sigma$. Secondly, we use slightly different \emph{identifications} between the spaces $\Omega^k_l$ for different values of $k$. Finally, Bryant defines the ``fundamental'' operators differently. For example, Bryant's $\dd^7_7$ is our $3 D^7_7$, and Bryant's $- \tfrac{3}{7} \dd^7_1$ is our $D^7_1$. We did notice at least one typographical error in~\cite{Bryant}. The equation $\dd (\alpha \wedge \st_{\sigma} \sigma) = - \st_{\sigma} \dd^7_7 \alpha$ in Table 1 is inconsistent with the definition $\dd^7_7 \alpha = \st_{\sigma} ( \dd (\alpha \wedge \st_{\sigma} \sigma))$ on the previous page, since $(\st_{\sigma})^2 = +1$, not $-1$.
\end{rmk}

From now on we assume $M$ is \emph{compact}, as we will be using Hodge theory throughout. Moreover, we can integrate by parts, so if $P$ is a linear operator on forms, then $P \alpha = 0 \iff P^* P \alpha = 0$, which we will use often. The next result relates the kernel of the operators in Definition~\ref{defn:Doperators} with harmonic $1$-forms. This result is \emph{fundamental} and is used often in the rest of the paper.

\begin{thm} \label{thm:harmonic1}
We have $\ker D^7_7 = \ker D^7_{14}$. Furthermore, let $\mathcal H^1 = \ker \left. \Delta \right|_{\Omega^1}$ denote the space of harmonic $1$-forms. Then we have
\begin{equation} \label{eq:harmonic1}
\begin{aligned}
\mathcal H^1 & = \ker D^7_1 \cap \ker D^7_7 \cap \ker D^7_{14} \\
& = \ker D^7_1 \cap \ker D^7_7 \\
& = \ker D^7_1 \cap \ker D^7_{27} \\
& = \ker D^7_7 \cap \ker D^7_{27}.
\end{aligned}
\end{equation}
In particular, the intersection of \emph{any two of the three} spaces $\ker D^7_1$, $\ker D^7_7$, $\ker D^7_{27}$ is $\mathcal H^1$.
\end{thm}
\begin{proof}
From Corollary~\ref{cor:adjoints}, on $\Omega^1_7$ we have that $\ds = (D^1_7)^* : \Omega^1_7 \to \Omega^0_1$ equals $-\tfrac{7}{3} D^7_1$, and thus
\begin{align*}
\mathcal{H}^1 &= (\ker \dd)^1 \cap (\ker \ds)^1 = \ker (D^7_7 +D^7_{14}) \cap \ker (-\tfrac{7}{3} D^7_1)\\ &= \ker D^7_7 \cap \ker D^7_{14} \cap \ker D^7_1,
\end{align*}
establishing the first equality in~\eqref{eq:harmonic1}.

Similarly from Corollary~\ref{cor:adjoints}, we have $(D^7_7)^* = 3 D^7_7$ and $(D^7_{14})^* = 4 D^{14}_7$. Hence, using $D^{14}_7 D^7_{14} = \tfrac{3}{2} D^7_7 D^7_7$ from~\eqref{eq:d-relations}, we have
\begin{align*}
D^7_7 \alpha = 0 &\iff (D^7_7)^* D^7_7 \alpha = 3 D^7_7 D^7_7 \alpha = 0 \\
& \iff D^{14}_7 D^7_{14} \alpha = \tfrac{1}{4} (D^7_{14})^* D^7_{14} \alpha = 0 \\
& \iff D^7_{14} \alpha = 0.
\end{align*}
Thus we deduce that $\ker D^7_7 = \ker D^7_{14}$ as claimed, and hence the second equality in~\eqref{eq:harmonic1} follows.

Finally, from Corollary~\ref{cor:adjoints} we have $(D^7_1)^* = - D^1_7$ and $(D^7_{27})^* = - \tfrac{4}{3} D^{27}_7$ and $(D^7_7)^* = 3 D^7_7$. Thus the relation $- D^1_7 D^7_1 + \tfrac{9}{4} D^7_7 D^7_7 + D^{27}_7 D^7_{27} = 0$ from~\eqref{eq:d-relations} can be written as
\begin{equation*}
(D^7_1)^* D^7_1 + \tfrac{3}{4} (D^7_7)^* D^7_7 - \tfrac{3}{4} (D^7_{27})^* D^7_{27} = 0.
\end{equation*}
From the above relation we easily deduce again by integration by parts that any two of $D^7_1 \alpha = 0$, $D^7_7 \alpha = 0$, $D^7_{27} \alpha = 0$ implies the third, establishing the remaining equalities in~\eqref{eq:harmonic1}.
\end{proof}

\subsection{The derivations $\mathcal L_B$ and $\mathcal L_K$ and their properties} \label{sec:LBLK}

We begin with a brief discussion of derivations on $\Omega^{\bu}$ arising from vector-valued forms on a general $n$-manifold $M$. A good reference for this material is~\cite{KMS}. We use notation similar to~\cite{CKT,dKS}.

Let $\Omega^r_{TM} = \Gamma(\Lambda^r (T^* M) \otimes TM)$ be the space of vector-valued $r$-forms on $M$. Given an element $K \in \Omega^r_{TM}$, it induces two derivations on $\Omega^{\bu}$. They are the \emph{algebraic derivation} $\iota_K$, of degree $r-1$, and the \emph{Nijenhuis-Lie derivation} $\mathcal{L}_K$, of degree $r$. They are defined as follows. Let $\{ e_1, \ldots, e_n \}$ be a (local) tangent frame with dual coframe $\{ e^1, \ldots, e^n \}$. Then locally $K = K^j e_j$ where each $K^j$ is an $r$-form. The operation $\iota_K : \Omega^k \to \Omega^{k+r-1}$ is defined to be
\begin{equation} \label{eq:alg-derivation}
\iota_K \alpha = K^j \wedge (e_j \hk \alpha),
\end{equation}
where $e_j \hk \cdot$ is the interior product with $e_j$. The operation $\iota_K$ is well-defined and is a derivation on $\Omega^{\bu}$. Moreover, $\iota_K$ vanishes on functions, so $\iota_K (h \alpha) = h (\iota_K \alpha)$ for any $h \in \Omega^0$ and $\alpha \in \Omega^k$, which justifies why $\iota_K$ is called algebraic. If $Y \in \Omega^1$, then
\begin{equation} \label{eq:alg-derivation1}
(\iota_K Y) (X_1, \ldots, X_r) = Y( K(X_1, \ldots, X_r) ).
\end{equation}
The operation $\mathcal{L}_K : \Omega^k \to \Omega^{k+r}$ is defined to be
\begin{equation} \label{eq:Lie-derivation}
\mathcal{L}_K \alpha =  \iota_K (\dd \alpha) - (-1)^{r-1} \dd (\iota_K \alpha) = [ \iota_K, \dd ] \alpha.
\end{equation}
That is, $\mathcal{L}_K$ is the graded commutator of $\iota_K$ and $\dd$. The graded Jacobi identity on the space of graded linear operators on $\Omega^{\bu}$ and $\dd^2 = 0$ together imply that
\begin{equation} \label{eq:commdL}
[ \dd, \mathcal{L}_K ] = \dd \mathcal{L}_K - (-1)^r \mathcal{L}_K \dd = 0.
\end{equation}

From now on, let $g$ be a Riemannian metric on $M$.

\begin{lemma} \label{lemma:iotaKfromform}
Let $K \in \Omega^r_{TM}$ be obtained from an $(r+1)$-form $\eta$ by raising the last index. That is, $g( K(X_1, \ldots, X_r), X_{r+1}) = \eta(X_1, \ldots, X_{r+1})$. In a local frame we have $K_{i_1 \cdots i_r}^q = \eta_{i_1 \cdots i_r p} g^{pq}$. The operator $\iota_K$ is of degree $r-1$. For any $\alpha \in \Omega^k$, the $(k+r-1)$-form $\iota_K \alpha$ is given by
\begin{equation} \label{eq:iotaKframe}
\iota_K \alpha = (-1)^r g^{pq} (e_p \hk \eta) \wedge (e_q \hk \alpha).
\end{equation}
\end{lemma}
\begin{proof}
In a local frame we have $K = \tfrac{1}{k!} K_{i_1 \cdots i_r}^q e^{i_1} \wedge \cdots \wedge e^{i_r} \otimes e_q$, and thus from~\eqref{eq:alg-derivation} we have
\begin{align*}
\iota_K \alpha & = \tfrac{1}{k!} K_{i_1 \cdots i_r}^q e^{i_1} \wedge \cdots \wedge e^{i_r} \wedge (e_q \hk \alpha) \\
& = \tfrac{1}{k!} \eta_{i_1 \cdots i_r p} g^{pq} e^{i_1} \wedge \cdots \wedge e^{i_r} \wedge (e_q \hk \alpha) \\
& = (-1)^r g^{pq} \big( \tfrac{1}{k!} \eta_{p i_1 \cdots i_r} e^{i_1} \wedge \cdots \wedge e^{i_r} \big) \wedge (e_q \hk \alpha) \\
& = (-1)^r g^{pq} (e_p \hk \eta) \wedge (e_q \hk \alpha)
\end{align*}
as claimed.
\end{proof}

\begin{cor} \label{cor:iotaKzero}
Let $K$ be as in Lemma~\ref{lemma:iotaKfromform}. If $\alpha \in \Omega^{n-(r-1)}$, then $\iota_K \alpha = 0$ in $\Omega^n$.
\end{cor}
\begin{proof}
Let $\alpha \in \Omega^{n-(r-1)}$. Since $e_p \hk \eta \in \Omega^r$, the form $(e_p \hk \eta) \wedge \alpha$ is of degree $(n+1)$ and hence zero. Taking the interior product with $e_q$, we have
\begin{equation*}
0 = e_q \hk \big( (e_p \hk \eta) \wedge \alpha \big) = (e_q \hk e_p \hk \eta) \wedge \alpha + (-1)^r (e_p \hk \eta) \wedge (e_q \hk \alpha).
\end{equation*}
Thus, by the skew-symmetry of $e_q \hk e_p \hk \eta$ in $p,q$, we find from~\eqref{eq:iotaKframe} that
\begin{equation*}
\iota_K \alpha = (-1)^r g^{pq} (e_p \hk \eta) \wedge (e_q \hk \alpha) = - g^{pq} (e_q \hk e_p \hk \eta) \wedge \alpha = 0
\end{equation*}
as claimed.
\end{proof}

\begin{cor} \label{cor:iotaKstar}
Let $K$ be as in Lemma~\ref{lemma:iotaKfromform}. Then the adjoint $\iota_K^*$ is a degree $-(r-1)$ operator on $\Omega^{\bu}$ and satisfies
\begin{equation} \label{eq:iotaKstar}
\iota_K^* \beta = (-1)^{nk + rk + nr + n + 1} \st \iota_K \st \beta \qquad \text{ for $\beta \in \Omega^k$}.
\end{equation}
\end{cor}
\begin{proof}
Let $\alpha \in \Omega^{k-(r-1)}$ and $\beta \in \Omega^k$. Then $\alpha \wedge \st \beta \in \Omega^{n-(r-1)}$, so by Lemma~\ref{lemma:iotaKfromform} we have $\iota_K (\alpha \wedge \st \beta) = 0$. Since $\iota_K$ is a derivation of degree $r-1$, and $\iota_K \st \beta$ is an $(n-k+r-1)$-form, this can be written as
\begin{align*}
0 & = (\iota_K \alpha) \wedge \st \beta + (-1)^{(r-1)(k-(r-1))} \alpha \wedge (\iota_K \st \beta) \\
& = g(\iota_K \alpha, \beta) \vol + (-1)^{rk + k + r + 1} \alpha \wedge (-1)^{(n-k+r-1)(k-r+1)} \st (\st \iota_K \st \beta) \\
& = g(\iota_K \alpha, \beta) \vol + (-1)^{rk + k + r + 1} (-1)^{k + r + 1 + nk + nr + n} g(\alpha, \st \iota_K \st \beta) \vol \\
& = g(\iota_K \alpha, \beta) \vol + (-1)^{nk + rk + nr + n} g(\alpha, \st \iota_K \st \beta) \vol,
\end{align*}
and hence $\iota_K^* \beta = (-1)^{nk + rk + nr + n + 1} \st \iota_K \st \beta$ as claimed.
\end{proof}

Now let $(M, \ph)$ be a manifold with $\G$-structure. In particular, $n=7$ from now on.

\begin{defn} \label{defn:BK}
From the $\G$-structure $\ph$ on $M$, we obtain two particular vector-valued forms $B \in \Omega^2_{TM}$ and $K \in \Omega^3_{TM}$ by raising the last index on the forms $\ph$ and $\ps$, respectively. That is,
\begin{equation*}
g( B(X, Y), Z) = \ph(X, Y, Z), \qquad g( K(X, Y, Z), W) = \ps (X, Y, Z, W).
\end{equation*}
In local coordinates we have
\begin{equation*}
B_{ij}^q = \ph_{ijp} g^{pq}, \qquad K_{ijk}^q = \ps_{ijkp} g^{pq}.
\end{equation*}
The vector-valued $2$-form $B$ is also called the \emph{cross product} induced by $\ph$, and, up to a factor of $-\tfrac{1}{2}$, the vector-valued $3$-form $K$ is called the \emph{associator}. (See~\cite[p.116]{HL} for details.) Thus $\iota_B$ and $\iota_K$ are algebraic derivations on $\Omega^{\bu}$ of degrees $1$ and $2$, respectively. We also have the associated Nijenhuis-Lie derivations $\mathcal L_B$ and $\mathcal L_K$. From~\eqref{eq:Lie-derivation} we have
\begin{equation} \label{eq:LBK}
\mathcal L_B = \iota_B \dd + \dd \iota_B, \qquad \mathcal L_K = \iota_K \dd - \dd \iota_K.
\end{equation}
The operators $\mathcal L_B$ and $\mathcal L_K$ are of degree $2$ and $3$, respectively.
\end{defn}

\begin{rmk} \label{rmk:chi}
In much of the literature the associator $K$ is denoted by $\chi$, but we are following the convention of~\cite{CKT, dKS} of denoting vector-valued forms by capital Roman letters.
\end{rmk}

\begin{prop} \label{prop:derivationsstar}
Let $\iota_B$, $\iota_K$, $\mathcal L_B$, and $\mathcal L_K$ be as in Definition~\ref{defn:BK}. Then on $\Omega^k$, we have
\begin{equation} \label{eq:derivationsstar}
\begin{aligned}
\iota_B^* & = (-1)^k \st \iota_B \st, & \qquad \iota_K^* & = -\st \iota_K \st, \\
\mathcal L_B^* & = - \st \mathcal L_B \st, & \qquad \mathcal L_K^* & = (-1)^k \st \mathcal L_K \st.
\end{aligned}
\end{equation}
\end{prop}
\begin{proof}
The first pairs of equations follow from~\eqref{eq:iotaKstar} with $n=7$ and $r=2,3$, respectively. In odd dimensions, $\ds = (-1)^k \st \dd \st$ on $k$-forms, and $\st^2 = 1$. The second pair of equations follows from these facts and taking adjoints of~\eqref{eq:LBK}.
\end{proof}

The operations $\iota_B$ and $\iota_K$ are morphisms of $\G$-representations, and in fact they are constants on $\Omega^l_{l'}$ after our identifications~\eqref{eq:forms-isom}. We will prove this in Propositions~\ref{prop:iotaBfigure} and~\ref{prop:iotaKfigure}, but first we need to collect several preliminary results.

\begin{lemma} \label{lemma:iotaBKprelim1}
Let $f \in \Omega^0$ and $X \in \Omega^1$. The following identities hold:
\begin{equation} \label{eq:iotaBKprelimeq1}
\begin{aligned}
\iota_B f & = 0, & \qquad \iota_K f & = 0, \\
\iota_B X & = X \hk \ph, & \qquad \iota_K X & = - X \hk \ps. 
\end{aligned}
\end{equation}
\end{lemma}
\begin{proof}
The first pair of equations are immediate since any algebraic derivation vanishes on functions. Letting $\alpha = X$ in~\eqref{eq:iotaKframe} gives $\iota_K X = (-1)^r g^{pq} (e_p \hk \eta) \wedge X_q = (-1)^r X^p e_p \hk \eta = (-1)^r X \hk \eta$. The second pair of equations now follows using $r = 2$ for $\eta = \ph$ and $r = 3$ for $\eta = \ps$.
\end{proof}

\begin{lemma} \label{lemma:iotaBKprelim2}
The following identities hold:
\begin{equation} \label{eq:iotaBKprelimeq2}
\begin{aligned}
\iota_B \ph & = - 6 \ps, & \qquad \iota_K \ph & = 0, \\
\iota_B \ps & = 0, & \qquad \iota_K \ps & = 0.
\end{aligned}
\end{equation}
\end{lemma}
\begin{proof}
To establish each of these, we use~\eqref{eq:iotaKframe} and Proposition~\ref{prop:special} with $h = g$. First, using~\eqref{eq:ellg} and $\tr_g g = 7$, we have
\begin{equation*}
\iota_B \ph = g^{pq} (e_p \hk \ph) \wedge (e_q \hk \ph) = - 2 (\tr_g g) \ps + 2 \ell_{\ps} g = - 14 \ps + 8 \ps = -6 \ps.
\end{equation*}
Similarly from Proposition~\ref{prop:special} we find that
\begin{equation*}
\iota_B \ps = g^{pq} (e_p \hk \ph) \wedge (e_q \hk \ps) = 0,
\end{equation*}
and hence also $\iota_K \ph = - g^{pq} (e_p \hk \ps) \wedge (e_q \hk \ph) = - \iota_B \ps = 0$. Finally, again from Proposition~\ref{prop:special} we deduce that
\begin{equation*}
\iota_K \ps = -g^{pq} (e_p \hk \ps) \wedge (e_q \hk \ps) = 0
\end{equation*}
as well.
\end{proof}

\begin{lemma} \label{lemma:iotaBKprelim3}
Let $X \in \Omega^1$. The following identities hold:
\begin{equation} \label{eq:iotaBKprelimeq3}
\begin{aligned}
\iota_B (X \hk \ph) & = 3 (X \hk \ps), & \qquad \iota_K (X \hk \ph) & = 3 \st (X \hk \ps), \\
\iota_B (X \hk \ps) & = -3 \st (X \hk \ps), & \qquad \iota_K (X \hk \ps) & = -4 \st (X \hk \ph).
\end{aligned}
\end{equation}
\end{lemma}
\begin{proof}
Let $X = X^m e_m$. By linearity of derivations and~\eqref{eq:iotaKframe} we have
\begin{align*}
\iota_B (X \hk \beta) & = X^m g^{pq} (e_p \hk \ph) \wedge (e_q \hk e_m \hk \beta), \\
\iota_K (X \hk \beta) & = - X^m g^{pq} (e_p \hk \ps) \wedge (e_q \hk e_m \hk \beta).
\end{align*}
The equations in~\eqref{eq:iotaBKprelimeq3} now follow immediately from Proposition~\ref{prop:special2}.
\end{proof}

\begin{lemma} \label{lemma:iotaBKprelim4}
Let $\beta \in \Omega^2_{14}$. The following identities hold:
\begin{equation} \label{eq:iotaBKprelimeq4}
\iota_B \beta = 0, \qquad \qquad \iota_K \beta = 0.
\end{equation}
\end{lemma}
\begin{proof}
We use the notation of Proposition~\ref{prop:special2}. Let $\beta \in \Omega^2_{14}$. Using~\eqref{eq:iotaKframe} and~\eqref{eq:ellon14} we compute
\begin{align*}
\iota_B \beta & = g^{pq} (e_p \hk \ph) \wedge (e_q \hk \beta) \\
& = \tfrac{1}{2} g^{pq} \ph_{pij} \beta_{qk} e^{ijk} \\
& = - \tfrac{1}{6} (\beta_{kq} g^{qp} \ph_{pij} + \beta_{iq} g^{qp} \ph_{pjk} + \beta_{jq} g^{qp} \ph_{pki} ) e^{ijk} = 0.
\end{align*}
Similarly, again using~\eqref{eq:iotaKframe} and~\eqref{eq:ellon14} we compute
\begin{align*}
\iota_K \beta & = -g^{pq} (e_p \hk \ps) \wedge (e_q \hk \beta) \\
& = -\tfrac{1}{6} g^{pq} \ps_{pijk} \beta_{ql} e^{ijkl} \\
& = + \tfrac{1}{24} (\beta_{lq} g^{qp} \ps_{pijk} - \beta_{iq} g^{qp} \ps_{pljk} - \beta_{jq} g^{qp} \ps_{pilk} - \beta_{kq} g^{qp} \ps_{pijl} ) e^{ijkl} = 0
\end{align*}
as claimed.
\end{proof}

We are now ready to establish the actions of $\iota_B$ and $\iota_K$ on the summands of $\Omega^{\bu}$ with respect to the identifications~\eqref{eq:forms-isom}.

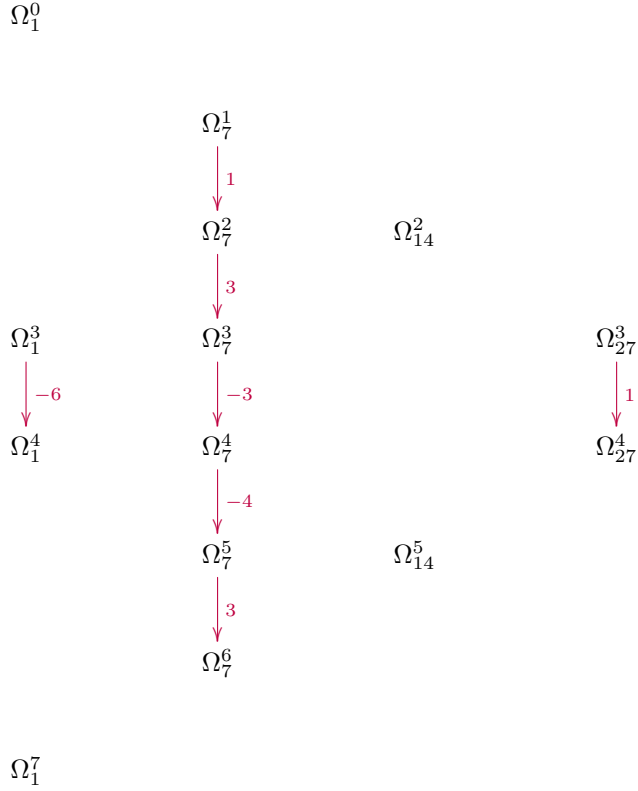
\begin{figure}[H]
\begin{equation*}
\xymatrix {
\Omega^0_1 & & & & & & \\
& & \Omega^1_7 \ar@[purple][d]^{\purple{1}} & & & & \\
& & \Omega^2_7 \ar@[purple][d]^{\purple{3}} & & \Omega^2_{14} & & \\
\Omega^3_1 \ar@[purple][d]^{\purple{-6}} & & \Omega^3_7 \ar@[purple][d]^{\purple{-3}} & & & & \Omega^3_{27} \ar@[purple][d]^{\purple{1}} \\
\Omega^4_1 & & \Omega^4_7 \ar@[purple][d]^{\purple{-4}} & & & & \Omega ^4_{27} \\
& & \Omega^5_7 \ar@[purple][d]^{\purple{3}} & & \Omega^5_{14} & & \\
& & \Omega^6_7 & & & & \\
\Omega^7_1 & & & & & &
}
\end{equation*}
\caption{Decomposition of the algebraic derivation $\iota_B$ into components} \label{figure:iotaB}
\end{figure}

\begin{prop} \label{prop:iotaBfigure}
With respect to the identifications described in~\eqref{eq:forms-isom}, the components of the operator $\iota_B$ satisfy the relations given in Figure~\ref{figure:iotaB}.
\end{prop}
\begin{proof}
The derivation $\iota_B$ is of degree $1$, so it vanishes on $\Omega^7$. Moreover, by Corollary~\ref{cor:iotaKzero} is also vanishes on $\Omega^6$. We establish the rest of Figure~\ref{figure:iotaB} by each vertical column.

$\Omega^k_1$ column: This follows from~\eqref{eq:iotaBKprelimeq1} and~\eqref{eq:iotaBKprelimeq2}. In particular, the map $\iota_B : \Omega^3_1 \to \Omega^4_1$ is identified with multiplication by $-6$.

$\Omega^k_7$ column: The map $\iota_B : \Omega^1_7 \to \Omega^2_7$ is identified with multiplication by $1$ by~\eqref{eq:iotaBKprelimeq1}. The maps $\iota_B : \Omega^2_7 \to \Omega^3_7$ and $\iota_B : \Omega^3_7 \to \Omega^4_7$ are identified with multiplication by $3$ and $-3$, respectively, by~\eqref{eq:iotaBKprelimeq3}. Let $\st (X \hk \ps) = \ph \wedge X \in \Omega^4_7$. Then
\begin{align*}
\iota_B \big( \st (X \hk \ps) \big) & = \iota_B (\ph \wedge X) = (\iota_B \ph) \wedge X - \ph \wedge (\iota_B X) \\
& = (-6 \ps) \wedge X - \ph \wedge (X \hk \ph) = - 6 \st (X \hk \ph) + 2 \st (X \hk \ph) \\
& = - 4 \st (X \hk \ph),
\end{align*}
and hence the map $\iota_B : \Omega^4_7 \to \Omega^5_7$ is identified with multiplication by $-4$. Finally, let $\st (X \hk \ph) = \ps \wedge X \in \Omega^5_7$. Then
\begin{align*}
\iota_B \big( \st (X \hk \ph) \big) & = \iota_B (\ps \wedge X) = (\iota_B \ps) \wedge X + \ps \wedge (\iota_B X) \\
& = 0 + \ps \wedge (X \hk \ph) = 3 \st X,
\end{align*}
and hence the map $\iota_B : \Omega^5_7 \to \Omega^6_7$ is identified with multiplication by $3$.

$\Omega^k_{14}$ column: The map $\iota_B$ on $\Omega^2_{14}$ is zero by Lemma~\ref{lemma:iotaBKprelim4}. Let $\mu = \st \beta \in \Omega^5_{14}$ where $\beta \in \Omega^2_{14}$. Then $\mu = \st \beta = \ph \wedge \beta$, so $\iota_B \mu = (\iota_B \ph) \wedge \beta - \ph \wedge (\iota_B \beta) = - 6 \ps \wedge \beta - 0 = 0$, by the description of $\Omega^2_{14}$ in~\eqref{eq:forms-isom}.

$\Omega^k_{27}$ column: Let $\gamma = \ell_{\ph} h \in \Omega^3_{27}$, where $h \in \symo$. By~\eqref{eq:ellphdefn} we have $\gamma = h_{kl} g^{lm} e^k \wedge (e_m \hk \ph)$. Since $\iota_B$ is algebraic, we can pull out functions, and using~\eqref{eq:iotaBKprelimeq1} and~\eqref{eq:iotaBKprelimeq3} we compute
\begin{align*}
\iota_B \gamma & = \iota_B \big( h_{kl} g^{lm} e^k \wedge (e_m \hk \ph) \big) \\
& = h_{kl} g^{lm} \big( (\iota_B e^k) \wedge (e_m \hk \ph) - e^k \wedge \iota_B (e_m \hk \ph) \big) \\
& = h_{kl} g^{lm} \big( g^{kp} (e_p \hk \ph) \wedge (e_m \hk \ph) - e^k \wedge (3 e_m \hk \ps) \big) \\
& = h^{pm} (e_p \hk \ph) \wedge (e_m \hk \ph) - 3 h_{kl} g^{lm} e^k \wedge (e_m \hk \ps).
\end{align*}
By~\eqref{eq:specialprop} and~\eqref{eq:ellpsdefn}, since $\tr_g h = 0$, the above expression is
\begin{equation*}
\iota_B \gamma = 2 \ell_{\ps} h - 3 \ell_{\ps} h = - \ell_{\ps} h.
\end{equation*}
Using Lemma~\ref{lemma:starell}, we conclude that $\iota_B (\ell_{\ph} h) = \st (\ell_{\ph} h)$, and thus the map $\iota_B : \Omega^3_{27} \to \Omega^4_{27}$ is identified with multiplication by $1$. Finally, let $\eta = \ell_{\ps} h \in \Omega^4_{27}$, where $h \in \symo$. By~\eqref{eq:ellpsdefn} we have $\eta = h_{kl} g^{lm} e^k \wedge (e_m \hk \ps)$. Computing as before, we find
\begin{align*}
\iota_B \eta & = \iota_B \big( h_{kl} g^{lm} e^k \wedge (e_m \hk \ps) \big) \\
& = h_{kl} g^{lm} \big( (\iota_B e^k) \wedge (e_m \hk \ps) - e^k \wedge \iota_B (e_m \hk \ps) \big) \\
& = h_{kl} g^{lm} \big( g^{kp} (e_p \hk \ph) \wedge (e_m \hk \ps) - e^k \wedge (-3 \st( e_m \hk \ps) ) \big) \\
& = h^{pm} (e_p \hk \ph) \wedge (e_m \hk \ps) + 3 h_{kl} g^{lm} e^k \wedge (\ph \wedge (e_m)^{\flat}).
\end{align*}
Using~\eqref{eq:specialprop}, the above expression becomes
\begin{equation*}
\iota_B \eta = 0 + 3 h_{kl} g^{lm} e^k \wedge \ph \wedge (g_{mp} e^p) = - 3 h_{kp} e^k \wedge e^p \wedge \ph = 0,
\end{equation*}
so the map $\iota_B$ on $\Omega^4_{27}$ is zero.
\end{proof}

\begin{figure}[H]
\begin{equation*}
\xymatrix {
\Omega^0_1 & & & \\
& \Omega^2_7 \ar@[purple][d]^{\purple{3}} & \Omega^{2}_{14} & \\
\Omega^4_1 & \Omega^4_7 \ar@[purple][d]^{\purple{4}} & & \Omega ^4_{27} \\
& \Omega^6_7 & & \\
} \qquad \qquad
\xymatrix {
 & \Omega^1_7 \ar@[purple][d]^{\purple{-1}} & & \\
\Omega^3_1 & \Omega^3_7 \ar@[purple][d]^{\purple{-4}} & & \Omega^3_{27} \\
& \Omega^5_7 & \Omega^5_{14} & \\
\Omega^7_1 & & & \\
}
\end{equation*}
\caption{Decomposition of the algebraic derivation $\iota_K$ into components} \label{figure:iotaK}
\end{figure}
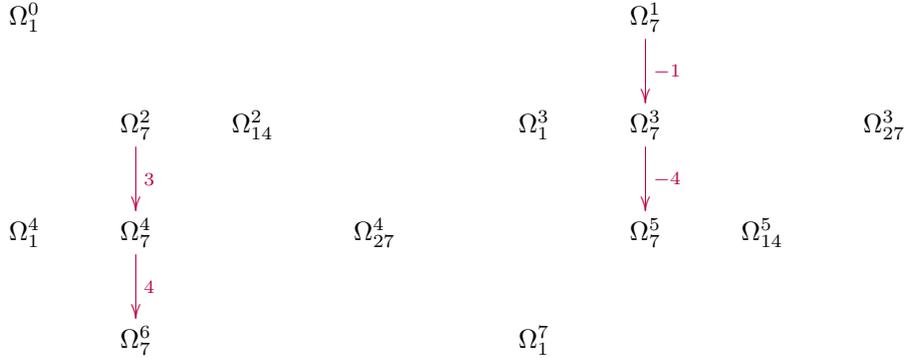

\begin{prop} \label{prop:iotaKfigure}
With respect to the identifications described in~\eqref{eq:forms-isom}, the components of the operator $\iota_K$ satisfy the relations given in Figure~\ref{figure:iotaK}.
\end{prop}
\begin{proof}
The derivation $\iota_K$ is of degree $2$, so it vanishes on $\Omega^6$ and $\Omega^7$. Moreover, by Corollary~\ref{cor:iotaKzero} is also vanishes on $\Omega^5$. We establish the rest of Figure~\ref{figure:iotaB} by each vertical column. Note that $\iota_K$ preserves the parity (even/odd) of forms.

$\Omega^k_1$ column: This follows from~\eqref{eq:iotaBKprelimeq1} and~\eqref{eq:iotaBKprelimeq2}.

$\Omega^k_7$ column: The map $\iota_K : \Omega^1_7 \to \Omega^3_7$ is identified with multiplication by $-1$ by~\eqref{eq:iotaBKprelimeq1}. The maps $\iota_K : \Omega^2_7 \to \Omega^4_7$ and $\iota_K : \Omega^3_7 \to \Omega^5_7$ are identified with multiplication by $3$ and $-4$, respectively, by~\eqref{eq:iotaBKprelimeq3}. Let $\st (X \hk \ps) = \ph \wedge X \in \Omega^4_7$. Then, since $\iota_K$ is an even derivation,
\begin{align*}
\iota_K \big( \st (X \hk \ps) \big) & = \iota_K (\ph \wedge X) = (\iota_K \ph) \wedge X + \ph \wedge (\iota_K X) \\
& = 0 + \ph \wedge (-X \hk \ps) = -\ph \wedge (X \hk \ps) = 4 \st X
\end{align*}
and hence the map $\iota_K : \Omega^4_7 \to \Omega^6_7$ is identified with multiplication by $4$.

$\Omega^k_{14}$ column: The map $\iota_K$ on $\Omega^2_{14}$ is zero by Lemma~\ref{lemma:iotaBKprelim4}.

$\Omega^k_{27}$ column: Let $\gamma = \ell_{\ph} h \in \Omega^3_{27}$, where $h \in \symo$. By~\eqref{eq:ellphdefn} we have $\gamma = h_{kl} g^{lm} e^k \wedge (e_m \hk \ph)$. Computing as in the proof of Proposition~\ref{prop:iotaKfigure}, we find that
\begin{align*}
\iota_K \gamma & = \iota_K \big( h_{kl} g^{lm} e^k \wedge (e_m \hk \ph) \big) \\
& = h_{kl} g^{lm} \big( (\iota_K e^k) \wedge (e_m \hk \ph) - e^k \wedge \iota_K (e_m \hk \ph) \big) \\
& = h_{kl} g^{lm} \big( -g^{kp} (e_p \hk \ps) \wedge (e_m \hk \ph) - e^k \wedge (3 \st( e_m \hk \ps) ) \big) \\
& = -h^{pm} (e_p \hk \ps) \wedge (e_m \hk \ph) - 3 h_{kl} g^{lm} e^k \wedge \ph \wedge (e_m)^{\flat}.
\end{align*}
The first term vanishes by~\eqref{eq:specialprop} and the second term vanishes as it is $-3 h_{kl} g^{lm} g_{mp} e^k \wedge \ph \wedge e^p = 3 h_{kp} e^k \wedge e^p \wedge \ph = 0$.
Thus the map $\iota_K$ vanishes on $\Omega^3_{27}$. Finally, let $\eta = \ell_{\ps} h \in \Omega^4_{27}$, where $h \in \symo$. By~\eqref{eq:ellpsdefn} we have $\eta = h_{kl} g^{lm} e^k \wedge (e_m \hk \ps)$. Computing as before, we find
\begin{align*}
\iota_K \eta & = \iota_K \big( h_{kl} g^{lm} e^k \wedge (e_m \hk \ps) \big) \\
& = h_{kl} g^{lm} \big( (\iota_K e^k) \wedge (e_m \hk \ps) - e^k \wedge \iota_K (e_m \hk \ps) \big) \\
& = h_{kl} g^{lm} \big( -g^{kp} (e_p \hk \ps) \wedge (e_m \hk \ps) - e^k \wedge (-4 \st( e_m \hk \ph) ) \big) \\
& = -h^{pm} (e_p \hk \ps) \wedge (e_m \hk \ps) + 4 h_{kl} g^{lm} e^k \wedge (\ps \wedge (e_m)^{\flat}).
\end{align*}
Again, the first term vanishes by~\eqref{eq:specialprop} and the second term vanishes as it is $4 h_{kl} g^{lm} g_{mp} e^k \wedge \ps \wedge e^p = 4 h_{kp} e^k \wedge e^p \wedge \ps = 0$.
Thus the map $\iota_K$ vanishes on $\Omega^4_{27}$.
\end{proof}

\begin{figure}[H]
\begin{equation*}
\xymatrix {
\Omega^0_1\ar@[red][dr]^{\red{D^1_7}} & & & \\
& \Omega^2_7 \ar@[blue][dl]_{\blue{-2 D^7_1}} \ar@[magenta][drr]_{\magenta{-2 D^7_{27}}} & \Omega^{2}_{14} \ar@/^/@[violet][dr]^{\violet{D^{14}_{27}}} \ar@/^/@[cyan][dl]_{\cyan{-3 D^{14}_7}} & \\
\Omega^4_1 \ar@[red][dr]_{\red{3 D^1_7}} & \Omega^4_7 \ar@[orange][d]^{\orange{-6 D^7_7}} & & \Omega ^4_{27} \ar@[teal][dll]^{\teal{4 D^{27}_7}} \\
& \Omega^6_7 & &
} \qquad \qquad
\xymatrix {
& \Omega^1_7 \ar@[blue][dl]_{\blue{D^7_1}} \ar@[orange][d]^{\orange{\frac{3}{2} D^7_7}} \ar@[magenta][drr]^{\magenta{D^7_{27}}} & & \\
\Omega^3_1 \ar@[red][dr]_{\red{-2 D^1_7}} & \Omega^3_7 \ar@/^/@[brown][dr]_{\brown{3 D^7_{14}}} & & \Omega^3_{27} \ar@[teal][dll]_{\teal{\, \, -\frac{8}{3} D^{27}_7}} \ar@[gray][dl]^{\gray{D^{27}_{14}}} \\
& \Omega^5_7 \ar@[blue][dl]^{\blue{7 D^7_1}} & \Omega^5_{14} & \\
\Omega^7_1 & & &
}
\end{equation*}
\caption{Decomposition of the Nihenhuis-Lie derivation $\mathcal L_B$ into components} \label{figure:LB}
\end{figure}

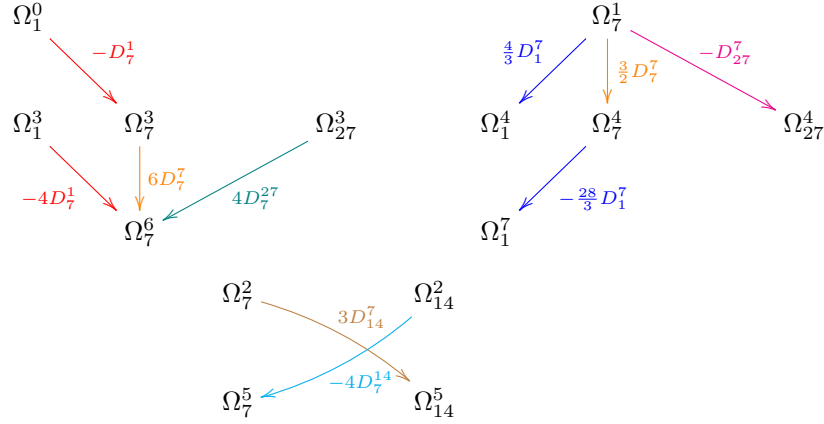
\begin{figure}[H]
\begin{equation*}
\xymatrix {
\Omega^0_1 \ar@[red][dr]^{\red{-D^1_7}} & & & \\
\Omega^3_1 \ar@[red][dr]_{\red{-4 D^1_7}} & \Omega^3_7 \ar@[orange][d]^{\orange{6 D^7_7}} & & \Omega^3_{27} \ar@[teal][dll]^{\teal{4 D^{27}_7}} \\
& \Omega^6_7 & & \\
} \qquad \qquad
\xymatrix {
& \Omega^1_7 \ar@[blue][dl]_{\blue{\frac{4}{3} D^7_1}} \ar@[orange][d]^{\orange{\frac{3}{2} D^7_7}} \ar@[magenta][drr]^{\magenta{-D^7_{27}}} & & \\
\Omega^4_1 & \Omega^4_7 \ar@[blue][dl]^{\blue{-\frac{28}{3} D^7_1}} & & \Omega ^4_{27} \\
\Omega^7_1 & & & \\
}
\end{equation*}
\begin{equation*}
\xymatrix {
& & \Omega^2_7 \ar@/^/@[brown][drr]^{\brown{3 D^7_{14}}} & & \Omega^2_{14} \ar@/^/@[cyan][dll]^{\cyan{\!\! -4 D^{14}_7}} & & & & \\
& & \Omega^5_7 & & \Omega^5_{14} & & & & \\
}
\end{equation*}
\caption{Decomposition of the Nijenhuis-Lie derivation $\mathcal L_K$ into components} \label{figure:LK}
\end{figure}

From now on in the paper, we always assume that $(M, \ph)$ is torsion-free. See also Remark~\ref{rmk:when-torsion-free}.

\begin{cor} \label{cor:LBLKfigures}
With respect to the identifications described in~\eqref{eq:forms-isom}, the components of the operators $\mathcal L_B$ and $\mathcal L_K$ satisfy the relations given in Figures~\ref{figure:LB} and~\ref{figure:LK}.
\end{cor}
\begin{proof}
This is straightforward to verify from Figures~\ref{figure:d},~\ref{figure:iotaB}, and~\ref{figure:iotaK} using the equations in~\eqref{eq:LBK}.
\end{proof}

Next we discuss some properties of $\mathcal L_B$ and $\mathcal L_K$.

\begin{lemma} \label{lemma:LBLKformula}
Let $\alpha$ be a form. In a local frame, the actions of $\mathcal L_B$ and $\mathcal L_K$ is given by
\begin{equation} \label{eq:LBLKframe}
\begin{aligned}
\mathcal L_B \alpha & = g^{pq} (e_p \hk \ph) \wedge (\nabla_q \alpha), \\
\mathcal L_K \alpha & = - g^{pq} (e_p \hk \ps) \wedge (\nabla_q \alpha).
\end{aligned}
\end{equation}
\end{lemma}
\begin{proof}
It is clear that both expressions in~\eqref{eq:LBLKframe} are independent of the choice of frame. To establish these expressions at $x \in M$, we choose a local frame determined by Riemannian normal coordinates centred at $x$. In particular, at the point $x$ we have $\nabla_p e_j = $ and $\nabla_p e^j = 0$. Recalling that $M$ is torsion-free, so $\nabla \ph = 0$, using~\eqref{eq:LBK},~\eqref{eq:iotaKframe}, and~\eqref{eq:dd} at the point $x$ we compute
\begin{align*}
\mathcal L_B \alpha & = (\iota_B \dd + \dd \iota_B) \alpha \\
& = \iota_B (e^m \wedge \nabla_m \alpha) + e^m \wedge \nabla_m (\iota_B \alpha) \\
& = g^{pq} (e_p \hk \ph) \wedge \big( e_q \hk (e^m \wedge \nabla_m \alpha) \big) + e^m \wedge \nabla_m \big( g^{pq} (e_p \hk \ph) \wedge (e_q \hk \alpha) \big) \\
& = g^{pq} (e_p \hk \ph) \wedge \big( \delta^m_q \nabla_m \alpha - e^m \wedge (e_q \hk \nabla_m \alpha) \big) + g^{pq} e^m \wedge (e_p \hk \ph) \wedge (e_q \hk \nabla_m \alpha) \\
& = g^{pq} (e_p \hk \ph) \wedge \nabla_q \alpha,
\end{align*}
establishing the first equation in~\eqref{eq:LBLKframe}. The other equation is proved similarly using $\nabla \ps = 0$.
\end{proof}

\begin{cor} \label{cor:LBLKds}
For any for $\alpha$, we have
\begin{equation} \label{eq:LBLKds}
\begin{aligned}
\mathcal L_B \alpha & = - \ds(\ph \wedge \alpha) - \ph \wedge \ds \alpha, \\
\mathcal L_K \alpha & = \ds (\ps \wedge \alpha) - \ps \wedge \ds \alpha.
\end{aligned}
\end{equation}
\end{cor}
\begin{proof}
Consider a local frame determined by Riemannian normal coordinates centred at $x \in M$ as in the proof of Lemma~\ref{lemma:LBLKformula}. Using~\eqref{eq:LBLKframe} and~\eqref{eq:ds}, we compute
\begin{align*}
\mathcal L_B \alpha & = g^{pq} (e_p \hk \ph) \wedge (\nabla_q \alpha) \\
& = g^{pq} \big( e_p \hk ( \ph \wedge \nabla_q \alpha) + \ph \wedge (e_p \hk \nabla_q \alpha) \big) \\
& = g^{pq} e_p \hk \nabla_q (\ph \wedge \alpha) + \ph \wedge ( g^{pq} e_p \hk \nabla_q \alpha) \\
& = - \ds (\ph \wedge \alpha) - \ph \wedge (\ds \alpha),
\end{align*}
establishing the first equation in~\eqref{eq:LBLKds}. The other equation in proved similarly.
\end{proof}

\begin{prop} \label{prop:LBLKproperties}
The derivations $\mathcal L_B$ and $\mathcal L_K$ satisfy the following identities:
\begin{align} \label{eq:LBLKdscom}
\mathcal L_B \ds & = \ds \mathcal L_B, & \mathcal L_K \ds & = -\ds \mathcal L_K, \\ \label{eq:LBLKLap}
\mathcal L_B \Delta & = \Delta \mathcal L_B, & \mathcal L_K \Delta & = \Delta \mathcal L_K, \\
\label{eq:LcomposeL}
\mathcal L_B \mathcal L_K = \mathcal L_K \mathcal L_B & = 0, & (\mathcal L_K)^2 & = 0,
\end{align}
and
\begin{equation} \label{eq:LBLKonH}
\mathcal L_B = \mathcal L_K = 0 \text{ on } \mathcal{H}^k \quad \text{if $M$ is compact}.
\end{equation}
\end{prop}
\begin{proof}
The identities in~\eqref{eq:LBLKdscom}--\eqref{eq:LcomposeL} can be verified directly from the Figures~\ref{figure:d},~\ref{figure:LB}, and~\ref{figure:LK} using $\ds = (-1)^k \st \dd \st$ on $\Omega^k$ and $\Delta = \dd \ds + \ds \dd$, the identities in Corollary~\ref{cor:d-relations}, and recalling that our identifications were chosen compatible with $\st$.

However, we now give an alternative proof of the first equation in~\eqref{eq:LBLKdscom} that is less tedious and more illuminating. A similar proof establishes the second equation in~\eqref{eq:LBLKdscom}. (In fact this proof can be found in~\cite{KLS2}). Using~\eqref{eq:LBLKds} and $(\ds)^2 = 0$, we compute
\begin{align*}
\mathcal L_B \ds \alpha & = - \ds (\ph \wedge \ds \alpha) - \ph \wedge \big( \ds (\ds \alpha) \big) \\
& = \ds \big( - \ph \wedge (\ds \alpha) - \ds (\ph \wedge \alpha) \big) \\
& = \ds \mathcal L_B \alpha.
\end{align*}
The equations in~\eqref{eq:LBLKLap} can also be established from~\eqref{eq:LBLKdscom},~\eqref{eq:commdL}, and $\Delta = \dd \ds + \ds \dd$.

Equation~\eqref{eq:LBLKonH} can be similarly verified using Figures~\ref{figure:d},~\ref{figure:LB}, and~\ref{figure:LK}, noting that in the compact case, the space $\mathcal H^k$ of harmonic $k$-forms coincides with the space of $\dd$-closed and $\ds$-closed $k$-forms.
\end{proof}

\begin{rmk} \label{rmk:LBLKcomms}
For a $k$-form $\gamma$, let $L_{\gamma}$ be the linear operator of degree $k$ on $\Omega^{\bu}$ given by $L_{\gamma} \alpha = \gamma \wedge \alpha$. In terms of graded commutators, in the torsion-free case Corollary~\ref{cor:LBLKds} says that $[ \ds, L_{\ph} ] = - \mathcal L_B$ and $[ \ds, L_{\ps} ] = \mathcal L_K$, and Proposition~\ref{prop:LBLKproperties} says that $[ \ds, \mathcal L_B ] = [ \ds, \mathcal L_K ] = 0$, $[ \Delta, \mathcal L_B ] = [ \Delta, \mathcal L_K ] = 0$, and $[ \mathcal L_B, \mathcal L_K ] = [ \mathcal L_K, \mathcal L_K ] = 0$. (In fact the first equation in~\eqref{eq:LcomposeL} is actually stronger than $[ \mathcal L_B, \mathcal L_K ] = 0$.) These graded commutators and others are considered more generally for $\G$~manifolds with torsion in~\cite{K} using the general framework developed in~\cite{dKS} in the case of $\U{m}$-structures.
\end{rmk}

\section{The $\mathcal L_B$-cohomology $H^{\bu}_{\ph}$ of $M$ and its computation} \label{sec:cohom}

In this section we define two cohomologies on a torsion-free $\G$~manifold using the derivations $\mathcal L_B$ and $\mathcal L_K$. The cohomology determined by $\mathcal L_K$ was studied extensively by Kawai--L\^e--Schwachh\"ofer in~\cite{KLS2}. We recall one of the main results of~\cite{KLS2} on the $\mathcal L_K$-cohomology, stated here as Theorem~\ref{thm:KLS}.  We then proceed to compute the cohomology determined by $\mathcal L_B$. This section culminates with the proof of Theorem~\ref{thm:Hph}, which is our analogue of Theorem~\ref{thm:KLS} for the $\mathcal L_B$-cohomology. An application to formality of compact torsion-free $\G$~manifolds is given in Section~\ref{sec:formality}.

\subsection{Cohomologies determined by $\mathcal L_B$ and $\mathcal L_K$} \label{sec:cohom-defn}

Recall from~\eqref{eq:LcomposeL} that $(\mathcal L_K)^2 = 0$. This observation motivates the following definition.

\begin{defn} \label{defn:Hps}
For any $0 \leq k \leq 7$, we define
\begin{equation*}
H^k_{\ps} : = \frac{\ker(\mathcal L_K: \Omega^k \to \Omega^{k+3})}{\im(\mathcal L_K: \Omega^{k-3} \to \Omega^k)}.
\end{equation*}
We call these groups the $\mathcal L_K$-cohomology groups.
\end{defn}
The $\mathcal L_K$-cohomology is studied extensively in~\cite{KLS2}. Here is one of the main results of~\cite{KLS2}.

\begin{thm}[Kawai--L\^e--Schwachh\"ofer~\cite{KLS2}] \label{thm:KLS}
The following relations hold.
\begin{itemize} \setlength\itemsep{-1mm}
\item $H^k_{\ps} \cong H^k_{\DR}$ for $k=0,1,6,7$.
\item $H^k_{\ps}$ is infinite-dimensional for $k=2,3,4,5$.
\item There is a canonical injection $\mathcal{H}^k \hookrightarrow H^k_{\ps}$ for all $k$.
\item The Hodge star induces isomorphisms $\st: H^k_{\ps} \cong H^{7-k}_{\ps}$.
\end{itemize}
\end{thm}
\begin{proof}
This is part of~\cite[Theorem 1.1]{KLS2}.
\end{proof}

From Figure~\ref{figure:LB} and~\eqref{eq:d-relations} we see that in general $(\mathcal L_B)^2 \neq 0$. Because of this, we \emph{cannot} directly copy the definition of $H^k_{\ps}$ to define $\mathcal L_B$-cohomology groups. However, we can make the following definition.

\begin{defn} \label{defn:Hph}
For any $0 \leq k \leq 7$, we define
\begin{equation} \label{eq:Hphdefn}
H^k_{\ph} := \frac{\ker(\mathcal L_B: \Omega^k \to \Omega^{k+2})}{\im(\mathcal L_B: \Omega^{k-2} \to \Omega^k) \cap \ker(\mathcal L_B: \Omega^k \to \Omega^{k+2})}.
\end{equation}
We call these groups the $\mathcal L_B$-cohomology groups.
\end{defn}

In Sections~\ref{sec:computeHph0123} and~\ref{sec:computeHph4567} we compute these $\mathcal L_B$-cohomology groups and then in Section~\ref{sec:Hphthm} we prove Theorem~\ref{thm:Hph}, which is the analogue to Theorem~\ref{thm:KLS}.

\subsection{Computation of the groups $H^0_{\ph}$, $H^1_{\ph}$, $H^2_{\ph}$, and $H^3_{\ph}$} \label{sec:computeHph0123}

From now on we always assume that $(M, \ph)$ is a \emph{compact} torsion-free $\G$~manifold as we use Hodge theory frequently. See also Remark~\ref{rmk:when-torsion-free}.

\begin{rmk} \label{rmk:IBP}
In particular we will often use the following observations. (There is no summation over $l, l', l''$ in this remark. The symbols $l, l', l'' \in \{ 1, 7, 14, 27 \}$ are not indices.) By Corollary~\ref{cor:adjoints}, we have $D^{l'}_{l} = c (D^{l}_{l'})^*$ for some $c \neq 0$. Thus, by integration by parts,
\begin{equation*}
\text{whenever $D^{l'}_l D^l_{l'} \omega = 0$ for some $\omega$, then $D^l_{l'} \omega = 0$.}
\end{equation*}
More generally, by Corollary~\ref{cor:adjoints} an equation of the form $a D^{l'}_l D^l_{l'} \omega + b D^{l''}_l D^l_{l''} \omega = 0$ can be rewritten as $\tilde a (D^{l}_{l'})^* D^{l}_{l'} \omega + \tilde b (D^{l}_{l''})^* D^{l}_{l''} \omega = 0$ for some $\tilde a, \tilde b$. If $\tilde a, \tilde b$ \emph{have the same sign}, then again by integration by parts we conclude that both $D^l_{l'} \omega = 0$ and $D^l_{l''} \omega = 0$.
\end{rmk}

In the first two Propositions we establish that $H^k_{\ph} \cong H^k_{\DR}$ for $k = 0,1,2$.

\begin{prop} \label{prop:Hph01}
We have $H^0_{\ph} = \mathcal{H}^0 $ and $H^1_{\ph} = \mathcal{H}^1$.
\end{prop}
\begin{proof}
From Figure~\ref{figure:LB} and Figure~\ref{figure:d}, we observe that
\begin{align*}
\im (\mathcal L_B: \Omega^{-2} \to \Omega^0) & = 0, \\
\ker (\mathcal L_B: \Omega^0 \to \Omega^2) & = \ker (D^1_7) = \mathcal{H}^0,
\end{align*}
and thus that $H^0_{\ph} = \mathcal{H}^0$.

Similarly, using Figure~\ref{figure:LB} and Theorem~\ref{thm:harmonic1}, we observe that
\begin{align*}
\im (\mathcal L_B: \Omega^{-1} \to \Omega^1) & = 0, \\
\ker (\mathcal L_B: \Omega^1 \to \Omega^3) & = \ker (D^7_1) \cap \ker (D^7_7) \cap \ker (D^7_{27}) = \mathcal{H}^1
\end{align*}
and hence $H^1_{\ph} = \mathcal{H}^1$.
\end{proof}

In the remainder of this section and the next we will often use the notation introduced in~\eqref{eq:complexes}.

\begin{prop} \label{prop:Hph2}
We have $H^2_{\ph} \cong \mathcal{H}^2 $.
\end{prop}
\begin{proof}
We first show that the denominator in~\eqref{eq:Hphdefn} is trivial. Let $\omega \in (\ker \mathcal L_B)^2 \cap ( \im \mathcal L_B)^2$. Then by Figure~\ref{figure:LB} we have
\begin{equation*}
\omega = \mathcal L_B f = D^1_7 f \quad \text{for some } f \in \Omega^0_1
\end{equation*}
and also that
\begin{equation*}
0 = \mathcal L_B \omega = -2 D^7_1 (D^1_7 f) - 2 D^7_{27} (D^1_7 f).
\end{equation*}
Projecting onto the $\Omega^4_1$ component, we find that $D^7_1 D^1_7 f = 0$. By Remark~\ref{rmk:IBP}, we deduce that $\omega = D^1_7 f = 0$. Thus we have shown that $(\ker \mathcal L_B)^2 \cap ( \im \mathcal L_B)^2 = 0$. Hence $H^2_{\ph} = (\ker \mathcal L_B)^2$.

Write $\omega = \omega_7 + \omega_{14} \in \Omega^2_7 \oplus \Omega^2_{14}$. By Figure~\ref{figure:LB} we have
\begin{equation} \label{eq:Hph2temp}
\omega \in (\ker \mathcal L_B)^2 \iff
\left\{
\begin{aligned}
-2 D^7_1 \omega_7 & = 0, \\
-3 D^{14}_7 \omega_{14} & = 0, \\
-2 D^7_{27} \omega_7 + D^{14}_{27} \omega_{14} & = 0.
\end{aligned}
\right\}
\end{equation}
Taking $D^{27}_7$ of the third equation in~\eqref{eq:Hph2temp}, using Corollary~\ref{cor:d-relations} to write $D^{27}_7 D^{14}_{27} = \tfrac{3}{2} D^7_7 D^{14}_7$, and using the second equation in~\eqref{eq:Hph2temp}, we find that
\begin{align*}
0 & = D^{27}_7 (-2 D^7_{27} \omega_7 + D^{14}_{27} \omega_{14}) \\
& = -2 D^{27}_7 D^7_{27} \omega_7 + \tfrac{3}{2} D^7_7 D^{14}_7 \omega_{14} = -2 D^{27}_{7} D^7_{27} \omega_7,
\end{align*}
implying by Remark~\ref{rmk:IBP} that $D^7_{27} \omega_7 = 0$. Therefore we have established that
\begin{equation*}
\omega \in (\ker \mathcal L_B)^2 \iff 
\left\{
\begin{aligned}
D^7_1 \omega_7 & = 0, \\
D^{14}_7 \omega_{14} & = 0, \\
D^7_{27} \omega_7 & = 0, \\
D^{14}_{27} \omega_{14} & = 0,
\end{aligned}
\right\}
\iff
\left\{
\begin{aligned}
\omega_7 & \in \mathcal{H}^2_7 \cong \mathcal{H}^1_7 \text{ by Theorem~\ref{thm:harmonic1}}, \\
\omega_{14} & \in \mathcal{H}^2_{14} \text{ by Figure~\ref{figure:d} and Corollary~\ref{cor:adjoints}}.
\end{aligned}
\right\}
\end{equation*}
We conclude that $H^2_{\ph} = (\ker \mathcal L_B)^2 = \mathcal{H}^2$.
\end{proof}

\begin{prop} \label{prop:Hph3}
We have $H^3_{\ph} = \mathcal{H}^3 \oplus \big( (\im \ds)^3 \cap (\ker \mathcal L_B)^3 \big)$.
\end{prop}
\begin{proof}
We first show that the denominator in~\eqref{eq:Hphdefn} is trivial. Let $\omega \in (\ker \mathcal L_B)^3 \cap ( \im \mathcal L_B)^3$. Then by Figure~\ref{figure:LB} we have
\begin{equation*}
\omega = \mathcal L_B \alpha = D^7_1 \alpha + \tfrac{3}{2} D^7_7 \alpha + D^7_{27} \alpha \quad \text{for some } \alpha \in \Omega^1_7.
\end{equation*}
Also, using Corollary~\ref{cor:d-relations} to write $D^{27}_{14} D^{7}_{27} = \tfrac{3}{2} D^7_{14} D^7_7$ and $D^{27}_7 D^7_{27} = D^1_7 D^7_1 - \tfrac{9}{4} D^7_7 D^7_7$, we have that
\begin{align*}
0 = \mathcal L_B \omega & = -2 D^1_7 (D^7_1 \alpha) + 3 D^7_{14} (\tfrac{3}{2} D^7_7 \alpha) + (- \tfrac{8}{3} D^{27}_7 + D^{27}_{14}) (D^7_{27} \alpha) \\
& = -2 D^1_7 D^7_1 \alpha - \tfrac{8}{3} D^{27}_7 D^7_{27} \alpha + \tfrac{9}{4} D^7_{14} D^7_7 \alpha + D^{27}_{14} D^7_{27} \alpha \\
& = - 2 D^1_7 D^7_1 \alpha - \tfrac{8}{3} (D^1_7 D^7_1 \alpha - \tfrac{9}{4} D^7_7 D^7_7 \alpha) + \tfrac{9}{4} D^7_{14} D^7_7 \alpha + \tfrac{3}{2} D^7_{14} D^7_7 \alpha \\
& = (-\tfrac{14}{3} D^1_7 D^7_1 \alpha + 6 D^7_7 D^7_7 \alpha) + \tfrac{15}{4} D^7_{14} D^7_7 \alpha.
\end{align*}
Projecting onto the $\Omega^5_7$ component, we find that
\begin{equation*}
-\tfrac{14}{3} D^1_7 D^7_1 \alpha + 6 D^7_7 D^7_7 \alpha = 0.
\end{equation*}
Using Corollary~\ref{cor:adjoints}, the above expression becomes
\begin{equation*}
\tfrac{14}{3} (D^7_1)^* D^7_1 \alpha + 2 (D^7_7)^* D^7_7 \alpha = 0,
\end{equation*}
and hence by Remark~\ref{rmk:IBP} we deduce that $D^7_1 \alpha = 0$ and $D^7_7 \alpha = 0$. By Theorem~\ref{thm:harmonic1}, we then have $D^7_{27} \alpha = 0$ automatically. Therefore we have shown that $(\ker \mathcal L_B)^3 \cap ( \im \mathcal L_B)^2 = 0$, and so $H^3_{\ph} = (\ker \mathcal L_B)^3$.

Write $\omega = \omega_1 + \omega_7 + \omega_{27} \in \Omega^3_1 \oplus \Omega^3_7 \oplus \Omega^3_{27}$. By Figure~\ref{figure:LB} we have
\begin{equation} \label{eq:Hph3temp}
\omega \in (\ker \mathcal L_B)^3 \iff
\left\{
\begin{aligned}
-2 D^1_7 \omega_1 - \tfrac{8}{3} D^{27}_7 \omega_{27} & = 0, \\
3 D^7_{14} \omega_7 + D^{27}_{14} \omega_{27} & = 0.
\end{aligned}
\right\}
\end{equation}
Taking $D^{14}_7$ of the second equation in~\eqref{eq:Hph3temp}, using Corollary~\ref{cor:d-relations} to write $D^{14}_7 D^{27}_{14} = - D^7_7 D^{27}_7$ and $D^7_7 D^1_7 = 0$, and using $D^{27}_7 \omega_{27} = - \tfrac{3}{4} D^1_7 \omega_1$ from the first equation in~\eqref{eq:Hph3temp}, we find that
\begin{align*}
0 & = D^{14}_7 (3 D^7_{14} \omega_7 + D^{27}_{14} \omega_{27}) \\
& = 3 D^{14}_7 D^7_{14} \omega_7 - D^7_7 D^{27}_7 \omega_{27} \\
& = 3 D^{14}_7 D^7_{14} \omega_7 + \tfrac{3}{4} D^7_7 D^1_7 \omega_1 = 3 D^{14}_7 D^7_{14} \omega_7,
\end{align*}
implying by Remark~\ref{rmk:IBP} that $D^7_{14} \omega_7 = 0$. Therefore we have established that
\begin{equation} \label{eq:Hph3temp2}
\omega \in (\ker \mathcal L_B)^3 \iff 
\left\{
\begin{aligned}
2 D^1_7 \omega_1 + \tfrac{8}{3} D^{27}_7 \omega_{27} & = 0, \\
D^7_{14} \omega_7 & = 0, \\
D^{27}_{14} \omega_{27} & = 0,
\end{aligned}
\right\}
\overset{\text{Theorem~\ref{thm:harmonic1}}}{\iff}
\left\{
\begin{aligned}
2 D^1_7 \omega_1 + \tfrac{8}{3} D^{27}_7 \omega_{27} & = 0, \\
D^7_7 \omega_7 & = 0, \\
D^{27}_{14} \omega_{27} & = 0.
\end{aligned}
\right\}
\end{equation}
From $\ds = - \st \dd \st$ on $\Omega^3$ and Figure~\ref{figure:d} we find that
\begin{equation} \label{eq:Hph3temp3}
\ds \omega = 0 \iff
\left\{
\begin{aligned}
D^1_7 \omega_1 + 2 D^7_7 \omega_7 + \tfrac{4}{3} D^{27}_7 \omega_{27} & = 0, \\
-D^7_{14} \omega_7 + D^{27}_{14} \omega_{27} & = 0.
\end{aligned}
\right\}
\end{equation}
Now equations~\eqref{eq:Hph3temp2} and~\eqref{eq:Hph3temp3} together imply that $(\ker \mathcal L_B)^3 \subseteq (\ker \ds)^3$. By the Hodge theorem we have $(\ker \ds)^3 = \mathcal{H}^3 \oplus (\im \ds)^3$, and by~\eqref{eq:LBLKonH} we have $\mathcal{H}^3 \subset (\ker \mathcal L_B)^3$. Thus
\begin{equation*}
\mathcal{H}^3 \subseteq (\ker \mathcal{L}_B)^3 \subseteq \mathcal{H}^3 \oplus (\im \ds)^3.
\end{equation*}
Applying Lemma~\ref{lemma:linalg}(i) we conclude that $H^3_{\ph} =(\ker \mathcal{L}_B)^3 = \mathcal{H}^3 \oplus \big( (\im \ds )^3 \cap (\ker \mathcal L_B)^3 \big)$. 
\end{proof}

We have thus far computed half of the $\mathcal L_B$-cohomology groups $H^k_{\ph}$, for $k = 0,1,2,3$. The other half, for $k=4,5,6,7$, will be computed rigorously in Section~\ref{sec:computeHph4567}. However, we can predict the duality result that $H^k_{\ph} \cong H^{7-k}_{\ph}$ by the following formal manipulation:
\begin{align*}
H^k_{\ph} & = \frac{(\ker \mathcal L_B)^k}{(\im \mathcal L_B)^k \cap (\ker \mathcal L_B)^k} \cong \frac{(\ker \mathcal L_B)^k + (\im \mathcal L_B)^k}{(\im \mathcal L_B)^k} \text{ by the second isomorphism theorem} \\
& \cong \frac{(\ker \mathcal L_B)^{7-k} + (\im \mathcal L_B)^{7-k}}{(\im \mathcal L_B)^{7-k}} \text{ by applying $\st$ and using equation~\eqref{eq:derivationsstar}} \\
& \overset{\text{\red{(!)}}}{=} \frac{\big( (\im \mathcal L_B)^{7-k} \big)^{\perp} + \big( (\ker \mathcal L_B)^{7-k} \big)^{\perp}}{\big( (\ker \mathcal L_B)^{7-k} \big)^{\perp}} \\
& = \frac{\big( (\im \mathcal L_B)^{7-k} \cap (\ker \mathcal L_B)^{7-k} \big)^{\perp}}{\big( (\ker \mathcal L_B)^{7-k} \big)^{\perp}} \text{ by properties of orthogonal complement} \\
& \overset{\text{\red{(!!)}}}{\cong} \frac{(\ker \mathcal L_B)^{7-k}}{(\im \mathcal L_B)^{7-k} \cap (\ker \mathcal L_B)^{7-k}}=H^{7-k}_{\ph}.
\end{align*}
Note that the above formal manipulation is not a rigorous proof of duality because at step \red{(!)}, we do not have $\im P^* = (\ker P)^{\perp}$ in general for an arbitrary operator $P$, and step \red{(!!)} is also not justified. Because $\Omega^k$ is not complete with respect to the $\mathcal{L}^2$-norm, the usual Hilbert space techniques do not apply. We will use elliptic operator theory to give a rigorous computation of $H^k_{\ph}$ for $k = 4,5,6,7$, in the next section.

\subsection{Computation of the groups $H^4_{\ph}$, $H^5_{\ph}$, $H^6_{\ph}$, and $H^7_{\ph}$} \label{sec:computeHph4567}

The material on regular operators in this section is largely based on Kawai--L\^e--Schwachh\"ofer~\cite{KLS2}.

\begin{defn} \label{defn:regular}
Let $P$ be a linear differential operator of degree $r$ on $\Omega^{\bu}$. Then $P : \Omega^{k - r} \to \Omega^k$ is said to be \emph{regular} if $\Omega^k = \im P \oplus \ker P^*$, where by $\ker P^*$ we mean the kernel of the formal adjoint $P^* : \Omega^k \to \Omega^{k - r}$ with respect to the $L^2$ inner product. The operator $P$ is said to be elliptic, overdetermined elliptic, underdetermined elliptic, if the principal symbol $\sigma_{\xi} (P)$ of $P$ is bijective, injective, surjective, respectively, for all $\xi \neq 0$.
\end{defn}

\begin{rmk} \label{rmk:regular}
It is a standard result in elliptic operator theory (see~\cite[p.464; 32 Corollary]{Besse}) that elliptic, overdetermined elliptic, and underdetermined elliptic operators are all regular.
\end{rmk}

\begin{prop} \label{prop:LBregular}
The operator $\mathcal L_B: \Omega^{k-2} \to \Omega^k$ is regular for all $k = 0, \ldots, 9$.
\end{prop}
\begin{proof}
Consider the symbol $P = \sigma_{\xi} (\mathcal L_B)$. By~\eqref{eq:LBLKframe}, this operator is $P (\omega) = (\xi \hk \ph) \wedge \omega$. Note that this is an algebraic (pointwise) map and thus at each point it is a linear map between finite-dimensional vector spaces. We will show that $P : \Omega^{k-2} \to \Omega^k$ is injective for $k = 0,1,2,3,4$ and surjective for $k=5,6,7,8,9$. The claim will then follow by Remark~\ref{rmk:regular}.

First we claim that injectivity of $P : \Omega^{k-2} \to \Omega^k$ for $k=0,1,2,3,4$ implies surjectivity of $P : \Omega^{k-2} \to \Omega^k$ for $k=5,6,7,8,9$. Suppose $P : \Omega^{k-2} \to \Omega^k$ is injective. Then the dual map $P^* : \Omega^k \to \Omega^{k-2}$ is surjective. But we have
\begin{equation*}
P^* = ( \sigma_{\xi} (\mathcal L_B) )^* = \sigma_{\xi} (\mathcal L_B^*),
\end{equation*}
and by~\eqref{eq:derivationsstar} this equals $\sigma_{\xi} (- \st \mathcal L_B \st) = - \st \sigma_{\xi} (\mathcal L_B) \st = - \st P \st$. Since $\st : \Omega^l \to \Omega^{7-l}$ is bijective, and we have that $\st P \st : \Omega^k \to \Omega^{k-2}$ is surjective, we deduce that $P : \Omega^{(9-k) - 2} \to \Omega^{9-k}$ is surjective. But $9-k \in \{ 5, 6, 7, 8, 9 \}$ if $k = \{ 0, 1, 2, 3, 4 \}$. Thus the claim is proved. 

It remains to establish injectivity of $P : \Omega^{k-2} \to \Omega^k$ for $k = 0, 1, 2, 3, 4$. This is automatic for $k = 0, 1$ since $\Omega^{k-2} = 0$ in these cases.

If $k = 2$, then $P : \Omega^0 \to \Omega^2$ is given by $Pf = (\xi \hk \ph) \wedge f = f (\xi \hk \ph)$. Suppose $Pf = 0$. Since $\xi \neq 0$, we have $\xi \hk \ph \neq 0$, and thus $f = 0$. So $P$ is injective for $k=2$.

If $k=3$, then $P : \Omega^1 \to \Omega^3$ is given by $P \alpha = (\xi \hk \ph) \wedge \alpha$. Suppose $P \alpha = 0$. Taking the wedge product of $P \alpha = 0$ with $\ps$ and using Lemma~\ref{lemma:identities2} gives
\begin{align*}
0 & = \ps \wedge (\xi \hk \ph) \wedge \alpha = 3 (\st \xi) \wedge \alpha \\
& = 3 g( \xi, \alpha) \vol.
\end{align*}
Thus $g(\xi, \alpha) = 0$. Similarly, taking the wedge product of $P \alpha = 0$ with $\ph$ and using Lemmas~\ref{lemma:identities2} and~\ref{lemma:identities3} gives
\begin{align*}
0 & = \ph \wedge (\xi \hk \ph) \wedge \alpha = -2 \big( \st (\xi \hk \ph) \big) \wedge \alpha \\
& = -2 \ps \wedge \xi \wedge \alpha = - 2 \st ( \xi \times \alpha ). 
\end{align*}
Thus $\xi \times \alpha = 0$. Taking the cross product of this with $\xi$ and using Lemma~\ref{lemma:identities3} gives
\begin{equation*}
- g(\xi, \xi) \alpha + g(\xi, \alpha) \xi = 0.
\end{equation*}
Since $g(\xi, \alpha) = 0$ and $\xi \neq 0$, we conclude that $\alpha = 0$. So $P$ is injective for $k=3$.

If $k=4$, then $P : \Omega^2 \to \Omega^4$ is given by $P \beta = (\xi \hk \ph) \wedge \beta$. Suppose $P \beta = 0$. This means
\begin{equation} \label{eq:regulartemp}
(\xi \hk \ph) \wedge \beta = 0.
\end{equation}
Write $\beta = \beta_7 + \beta_{14} \in \Omega^2_7 \oplus \Omega^2_{14}$, where by~\eqref{eq:forms-isom} we can write $\beta_7 = Y \hk \ph$ for some unique $Y$. Taking the wedge product of~\eqref{eq:regulartemp} with $\ph$ and using~\eqref{eq:forms-isom} and~\eqref{eq:fund-eq}, we have
\begin{align*}
0 & = (\xi \hk \ph) \wedge \ph \wedge \beta = (\xi \hk \ph) \wedge (- 2 \st \beta_7 + \st \beta_{14}) \\
& = - 2 (\xi \hk \ph) \wedge \st \beta_7 + 0 = (\xi \hk \ph) \wedge (Y \hk \ph) \wedge \ph = - 6 g(\xi, Y) \vol.
\end{align*}
Thus we have
\begin{equation} \label{eq:regulartemp2}
g(\xi, Y) = 0.
\end{equation}
Now we take the interior product of~\eqref{eq:regulartemp} with $\xi$. This gives $(\xi \hk \ph) \wedge (\xi \hk \beta) = 0$. By the injectivity of $P$ for $k=3$, we deduce that
\begin{equation} \label{eq:regulartemp2b}
\xi \hk \beta = 0.
\end{equation}
Using~\eqref{eq:regulartemp2b} and~\eqref{eq:forms-isom}, we can rewrite~\eqref{eq:regulartemp} as
\begin{equation*}
0 = \xi \hk (\ph \wedge \beta) = \xi \hk ( - 2 \st \beta_7 + \st \beta_{14} ).
\end{equation*}
Taking $\st$ of the above equation and using $\st ( \xi \hk \st \gamma) = \pm \xi \wedge \gamma$, where in general the sign depends on the dimension of the manifold and the degree of $\gamma$, we find that
\begin{equation} \label{eq:regulartemp3}
- 2 \xi \wedge \beta_7 + \xi \wedge \beta_{14} = 0.
\end{equation}
Equation~\eqref{eq:regulartemp3} implies that
\begin{equation} \label{eq:regulartemp4}
\xi \wedge \beta = \xi \wedge \beta_7 + \xi \wedge \beta_{14} = 3 \xi \wedge \beta_7.
\end{equation}
Taking the interior product of~\eqref{eq:regulartemp4} with $\xi$ and using~\eqref{eq:regulartemp2b} yields
\begin{equation} \label{eq:regulartemp5}
g(\xi, \xi) \beta = 3 g(\xi, \xi) \beta_7 - 3 \xi \wedge (\xi \hk \beta_7).
\end{equation}
By Lemma~\ref{lemma:identities3} we have $\xi \hk \beta_7 = \xi \hk Y \hk \ph = Y \times \xi$. Thus~\eqref{eq:regulartemp5} becomes
\begin{equation} \label{eq:regulartemp6}
g(\xi, \xi) \beta = 3 g(\xi, \xi) \beta_7 - 3 \xi \wedge (Y \times \xi).
\end{equation}
Now we take the wedge product of~\eqref{eq:regulartemp6} with $\ps$, use Lemma~\ref{lemma:identities3} again, and the fact that $\beta_{14} \wedge \ps = 0$ from~\eqref{eq:forms-isom}. We obtain
\begin{align*}
g(\xi, \xi) \beta_7 \wedge \ps & = 3 g(\xi, \xi) \beta_7 \wedge \ps - 3 \xi \wedge (Y \times \xi) \wedge \ps \\
& = 3 g(\xi, \xi) \beta_7 \wedge \ps - 3 \st (\xi \times (Y \times \xi)),
\end{align*}
which can be rearranged to give, using Lemma~\ref{lemma:identities3} and~\eqref{eq:regulartemp2}, that
\begin{equation} \label{eq:regulartemp7}
- 2 g(\xi, \xi) \beta_7 \wedge \ps = 3 \st (\xi \times (\xi \times Y)) = - 3 \st \big( g(\xi, \xi) Y \big).
\end{equation}
But from Lemma~\ref{lemma:identities2} we find $\beta_7 \wedge \ps = (Y \hk \ph) \wedge \ps = 3 \st Y$. Substituting this into~\eqref{eq:regulartemp7} and taking $\st$, we find that
\begin{equation*}
- 3 g(\xi, \xi) Y = -2 g(\xi, \xi) \st \big( 3 \st Y \big) = - 6 g(\xi, \xi) Y.
\end{equation*}
Since $\xi \neq 0$, we deduce that $Y = 0$ and thus $\beta_7 = 0$. Substituting back into~\eqref{eq:regulartemp5} then gives $g(\xi, \xi) \beta_{14} = 0$ and thus $\beta_{14} = 0$ as well. So $P$ is injective for $k=4$.
\end{proof}

\begin{cor} \label{cor:LBregular}
For any $k = 0, \ldots, 7$, we have
\begin{equation} \label{eq:LBregular}
(\im \mathcal L_B)^k = \st ( (\ker \mathcal L_B)^{7-k})^{\perp}.
\end{equation}
\end{cor}
\begin{proof}
By~\eqref{eq:derivationsstar} we have $(\im \mathcal{L}_B)^k = \st (\im \mathcal{L}^*_B)^{7-k}$, and because $\mathcal L_B$ is regular by Proposition~\ref{prop:LBregular}, we have $(\im \mathcal{L}^*_B)^{7-k} = ((\ker \mathcal{L}_B)^{7-k})^{\perp}$. The result follows.
\end{proof}

\begin{prop} \label{prop:Hph67}
We have $H^7_\ph \cong \mathcal{H}^7$ and $H^6_\ph \cong \mathcal{H}^6$.
\end{prop}
\begin{proof}
In the proof of Proposition~\ref{prop:Hph01}, we showed that $(\ker \mathcal{L}_B)^0 = \mathcal{H}^0$ and $(\ker \mathcal{L}_B)^1 = \mathcal{H}^1$. Thus using~\eqref{eq:LBregular} we have
\begin{align*}
(\im \mathcal{L}_B)^7 & = \st ((\ker \mathcal{L}_B)^0)^\perp = \st (\mathcal{H}^0)^{\perp} & & \\
& = (\im \dd)^7 \oplus (\im \ds)^7 \qquad & & \text{by the Hodge decomposition.}
\end{align*}
In exactly the same way we get $(\im \mathcal{L}_B)^6 = (\im \dd)^6 \oplus (\im \ds)^6$.

Moreover, since $\mathcal L_B$ has degree two, we have $(\ker \mathcal{L}_B)^6 = \Omega^6$ and $(\ker \mathcal{L}_B)^7 = \Omega^7$. Thus, we conclude that
\begin{equation*}
H^k_\ph = \frac{\Omega^k}{(\im \dd)^k \oplus (\im \ds)^k} \cong \mathcal{H}^k \qquad \text{ for $k = 6,7$.}
\end{equation*}
by the Hodge decomposition.
\end{proof}

\begin{prop} \label{prop:Hph5}
We have $H^5_\ph \cong \mathcal{H}^5$.
\end{prop}
\begin{proof}
In the proof of Proposition~\ref{prop:Hph2}, we showed that $(\ker \mathcal L_B)^2 = \mathcal{H}^2$, so using~\eqref{eq:LBregular} just as in the proof of Proposition~\ref{prop:Hph67} we deduce that
\begin{equation} \label{eq:Hph5temp}
(\im \mathcal{L}_B)^5 = (\im \dd)^5 \oplus (\im \ds)^5.
\end{equation}

Let $\alpha \in \Omega^6$. Then since $\ds = \st \dd \st$ on $\Omega^6$, we find from Figure~\ref{figure:d} that up to our usual identfications, $\ds \alpha = D^7_7 \alpha + D^7_{14} \alpha \in \Omega^5_7 \oplus \Omega^5_{14}$. Then Figure~\ref{figure:LB} and~\eqref{eq:d-relations} gives
\begin{equation*}
\mathcal L_B \ds \alpha = \mathcal L_B (D^7_7 \alpha + D^7_{14} \alpha) = 7 D^7_1 D^7_7 \alpha + 0 = 0,
\end{equation*}
so $(\im \ds)^5 \subset (\ker \mathcal L_B)^5$. We also have $\mathcal{H}^5 \subset (\ker \mathcal{L}_B)^5$ by~\eqref{eq:LcomposeL}. Using the Hodge decomposition of $\Omega^5$ we therefore have
\begin{equation*}
\mathcal{H}^5 \oplus (\im \ds)^5 \subseteq (\ker \mathcal{L}_B)^5 \subseteq \Omega^5 = \mathcal{H}^5 \oplus (\im \ds)^5 \oplus (\im \dd)^5.
\end{equation*}
Applying Lemma~\ref{lemma:linalg}(i) we deduce that
\begin{equation} \label{eq:Hph5temp2}
(\ker \mathcal L_B)^5 = \mathcal{H}^5 \oplus (\im \ds)^5 \oplus \big( (\im \dd)^5 \cap (\ker \mathcal L_B)^5 \big).
\end{equation}
Applying Lemma~\ref{lemma:linalg}(ii) to~\eqref{eq:Hph5temp},~\eqref{eq:Hph5temp2}, as subspaces of $\Omega^5 = \mathcal{H}^5 \oplus (\im \dd)^5 \oplus (\im \ds)^5$, we obtain
\begin{equation} \label{eq:Hph5temp3}
(\im \mathcal L_B)^5 \cap (\ker \mathcal L_B)^5 = (\im \ds)^5 \oplus \big( (\im \dd)^5 \cap (\ker \mathcal L_B)^5 \big).
\end{equation}
Therefore we find that
\begin{align*}
H^5_\ph & = \frac{(\ker \mathcal{L}_B)^5}{(\ker \mathcal{L}_B)^5 \cap (\im \mathcal{L}_B)^5} & & \\
& = \frac{\mathcal{H}^5 \oplus (\im \ds)^5 \oplus \big( (\im \dd)^5 \cap (\ker \mathcal{L}_B)^5 \big)}{(\im \ds)^5 \oplus (\im \dd)^5 \cap (\ker \mathcal{L}_B)^5} & & \text{by~\eqref{eq:Hph5temp2} and~\eqref{eq:Hph5temp3}} \\
& \cong \mathcal{H}^5 & & 
\end{align*}
as claimed. 
\end{proof}

Before we can compute $H^4_{\ph}$ we need two preliminary results.

\begin{lemma} \label{lemma:Hph4-1}
We have
\begin{equation} \label{eq:Hph4-1}
(\ker \mathcal{L}_B)^4 \cap \big( (\im \ds)^4 \oplus (\im \dd)^4 \big) = \big( (\ker \mathcal{L}_B)^4 \cap (\im \ds)^4 \big) \oplus \big( (\ker \mathcal{L}_B)^4 \cap (\im \dd)^4 \big).
\end{equation}
\end{lemma}
\begin{proof}
Let $\beta = \beta_7 + \beta_{14} \in \Omega^5_7 \oplus \Omega^5_{14}$, and $\gamma = \gamma_1 + \gamma_7 + \gamma_{27} \in \Omega^3_1 \oplus \Omega^3_7 \oplus \Omega^3_{27}$. We need to prove that
\begin{equation} \label{eq:Hph4-1temp}
\begin{aligned}
&\mathcal{L}_B \ds (\beta_7 + \beta_{14} ) + \mathcal{L}_B \dd (\gamma_1 + \gamma_7 + \gamma_{27}) = 0 \\
\iff & \mathcal{L}_B \ds (\beta_7 + \beta_{14}) = \mathcal{L}_B \dd (\gamma_1 + \gamma_7 + \gamma_{27}) = 0.
\end{aligned}
\end{equation}
From $\ds = - \st \dd \st$ on $\Omega^5$ and Figures~\ref{figure:d} and~\ref{figure:LB}, we have
\begin{align*}
\mathcal{L}_B\ds (\beta_7 + \beta_{14}) & = \mathcal{L}_B (-D^7_1 \beta_7 + \tfrac{3}{2} D^7_7 \beta_7 - D^7_{27} \beta_7 - D^{14}_7 \beta_{14} - D^{14}_{27} \beta_{14} ) \\
& = 3 D^1_7 (-D^7_1 \beta_7) - 6 D^7_7 ( \tfrac{3}{2} D^7_7 \beta_7  - D^{14}_7 \beta_{14} ) + 4 D^{27}_7 ( - D^7_{27} \beta_7 - D^{14}_{27} \beta_{14}) \\
& = - 3 D^1_7 D^7_1 \beta_7 - 9 D^7_7 D^7_7 \beta_7  + 6 D^7_7 D^{14}_7 \beta_{14} - 4 D^{27}_7 D^7_{27} \beta_7 - 4 D^{27}_7 D^{14}_{27} \beta_{14}.
\end{align*}
Using the relations in~\eqref{eq:d-relations}, the above expression simplifies to
\begin{equation} \label{eq:Hph4-1temp2}
\mathcal{L}_B\ds (\beta_7 + \beta_{14}) = -7 D^1_7 D^7_1 \beta_7.
\end{equation}
Similarly from Figures~\ref{figure:d} and~\ref{figure:LB} and $D^7_7 D^1_7 = 0$, we have
\begin{align*}
\mathcal{L}_B \dd (\gamma_1 + \gamma_7 + \gamma_{27}) & = \mathcal{L}_B (\tfrac{4}{3} D^7_1 \gamma_7 + ( - D^1_7 \gamma_1 - \tfrac{3}{2} D^7_7 \gamma_7 + D^{27}_7 \gamma_{27} ) + ( - D^7_{27} \gamma_7 + D^{27}_{27} \gamma_{27}) ) \\
& = 3 D^1_7 (\tfrac{4}{3} D^7_1 \gamma_7) - 6 D^7_7 ( - D^1_7 \gamma_1 - \tfrac{3}{2} D^7_7 \gamma_7 + D^{27}_7 \gamma_{27} ) + 4 D^{27}_7 ( - D^7_{27} \gamma_7 + D^{27}_{27} \gamma_{27}) ) \\
& = 4 D^1_7 D^7_1 \gamma_7 + 9 D^7_7 D^7_7 \gamma_7 - 6 D^7_7 D^{27}_7 \gamma_{27} - 4 D^{27}_7 D^7_{27} \gamma_7 + 4 D^{27}_7 D^{27}_{27} \gamma_{27}.
\end{align*}
Using the relations in~\eqref{eq:d-relations}, the above expression simplifies to
\begin{equation} \label{eq:Hph4-1temp3}
\mathcal{L}_B \dd (\gamma_1 + \gamma_7 + \gamma_{27}) = 18 D^7_7 D^7_7 \gamma_7 - 12 D^7_7 D^{27}_7 \gamma_{27}.
\end{equation}
Combining equations~\eqref{eq:Hph4-1temp2} and~\eqref{eq:Hph4-1temp3}, if $\mathcal{L}_B \ds (\beta_7 + \beta_{14}) + \mathcal{L}_B \dd (\gamma_1 + \gamma_7 + \gamma_{27}) = 0$, then we have 
\begin{equation*}
- 7 D^1_7 D^7_1 \beta_7 + 18 D^7_7 D^7_7 \gamma_7 - 12 D^7_7 D^{27}_7 \gamma_{27} = 0,
\end{equation*}
and thus, applying $D^7_1$ and using $D^7_1 D^7_7 = 0$, we deduce that
\begin{equation*}
7 D^7_1 D^1_7 D^7_1 \beta_7 = D^7_1 D^7_7 ( 18 D^7_7 \gamma_7 - 12 D^{27}_7 \gamma_{27} ) = 0.
\end{equation*}
Thus we have $D^7_1 D^1_7 D^7_1 \beta_7 = 0$. Applying Remark~\ref{rmk:IBP}, we deduce that $D^1_7 D^7_1 \beta_7 = 0$ and thus by~\eqref{eq:Hph4-1temp2} that $\mathcal L_B \ds (\beta_7 + \beta_{14}) = 0$. Thus we have established~\eqref{eq:Hph4-1temp} and consequently
\begin{equation*}
(\ker \mathcal{L}_B)^4 \cap \big( (\im \ds)^4 \oplus (\im \dd)^4 \big) = \big( (\ker \mathcal{L}_B)^4 \cap (\im \ds)^4 \big) \oplus \big( (\ker \mathcal{L}_B)^4 \cap (\im \dd)^4 \big)
\end{equation*}
as claimed.
\end{proof}

\begin{lemma} \label{lemma:Hph4-2}
We have
\begin{equation} \label{eq:Hph4-2}
(\im \dd)^4 \cap (\ker \mathcal{L}_B)^4 \cap (\im \mathcal{L}_B)^4 = 0.
\end{equation}
\end{lemma}
\begin{proof}
Let $\omega \in (\im \dd)^4 \cap (\ker \mathcal{L}_B)^4 \cap (\im \mathcal{L}_B)^4$. We write $\omega = \mathcal{L}_B (\alpha_7 + \alpha_{14})$ for some $\alpha_7 + \alpha_{14} \in \Omega^2_7 \oplus \Omega^2_{14}$. Using Figure~\ref{figure:LB}, we find
\begin{equation*} 
\begin{aligned}
\omega = \mathcal{L}_B (\alpha_7 + \alpha_{14}) & = (- 2 D^7_1 \alpha_7) + (- 3 D^{14}_7 \alpha_{14}) + (- 2 D^7_{27} \alpha_7 + D^{14}_{27} \alpha_{14}) \\
& = \omega_1 + \omega_7 + \omega_{27} \in \Omega^4_1 \oplus \Omega^4_7 \oplus \Omega^4_{27}.
\end{aligned}
\end{equation*}
That is, we have
\begin{equation} \label{eq:Hph4-2LBo}
\begin{aligned}
\omega_1 & = - 2 D^7_1 \alpha_7, \\
\omega_7 & = - 3 D^{14}_7 \alpha_{14}, \\
\omega_{27} & = - 2 D^7_{27} \alpha_7 + D^{14}_{27} \alpha_{14}.
\end{aligned}
\end{equation}
Using Figure~\ref{figure:LB} again, the equation $\mathcal{L}_B \mathcal{L}_B (\alpha_7 + \alpha_{14}) = \mathcal{L}_B \omega = 0$ gives
\begin{align*}
0 & = \mathcal{L}_B \big( - 2 D^7_1 \alpha_7 - 3 D^{14}_7 \alpha_{14} + (- 2 D^7_{27} \alpha_7 + D^{14}_{27} \alpha_{14} ) \big) \\
& = 3 D^1_7( - 2 D^7_1 \alpha_7) - 6 D^7_7 ( - 3 D^{14}_7 \alpha_{14} ) + 4 D^{27}_7 (- 2 D^7_{27} \alpha_7 + D^{14}_{27} \alpha_{14} ) \\
& = -6 D^1_7D^7_1\alpha_7 + 18 D^7_7D^{14}_7 \alpha_{14} - 8 D^{27}_7 D^7_{27} \alpha_7 + 4 D^{27}_7 D^{14}_{27} \alpha_{14}.
\end{align*}
Using the relations~\eqref{eq:d-relations}, we can rewrite the above expression in two different ways, both of which will be useful. These are
\begin{align}
-6 D^1_7 D^7_1 \alpha_7 - 8 D^{27}_7 D^7_{27} \alpha_7 + 24 D^7_7 D^{14}_7 \alpha_{14} & = 0, \label{eq:Hph4-2temp} \\
-14 D^1_7 D^7_1 \alpha_7 + 18 D^7_7 D^7_7 \alpha_7 + 24 D^7_7 D^{14}_7 \alpha_{14} & = 0. \label{eq:Hph4-2temp2}
\end{align}
Applying $D^7_1$ to~\eqref{eq:Hph4-2temp2} and using $D^7_1 D^7_7 = 0$, we deduce that
\begin{equation*}
14 D^7_1 D^1_7 D^7_1 \alpha_7 = (D^7_1 D^7_7)(18 D^7_7 \alpha_7 + 24 D^{14}_7 \alpha_{14}) = 0.
\end{equation*}
Thus we have $D^7_1 D^1_7 D^7_1 \alpha_7 = 0$. Applying Remark~\ref{rmk:IBP} twice, we deduce first that $D^1_7 D^7_1 \alpha_7 = 0$ and then that
\begin{equation}
D^7_1 \alpha_7 = 0. \label{eq:Hph4-2temp3}
\end{equation}
Comparing~\eqref{eq:Hph4-2temp3} and~\eqref{eq:Hph4-2LBo} we find that $\omega_1 = 0$. Since $\omega \in (\im \dd)^4$, it is $\dd$-closed. Using Figures~\ref{figure:d} and~\ref{figure:LB}, the conditions $\pi_7 \dd \omega = 0$ and $\mathcal{L}_B \omega = 0$ give, respectively,
\begin{align*} 
2 D^7_7 \omega_7 + \tfrac{4}{3} D^{27}_7 \omega_{27} & = 0, \\
-6 D^7_7 \omega_7 + 4 D^{27}_7 \omega_{27} & = 0.
\end{align*}
These two equations together force
\begin{equation} \label{eq:Hph4-2temp4}
D^7_7 \omega_7 = 0 \qquad \text{and} \qquad D^{27}_7 \omega_{27} = 0.
\end{equation}
Also, from~\eqref{eq:Hph4-2LBo} we have $\omega_7 = - 3 D^{14}_7 \alpha_{14}$, and thus since $D^7_1 D^{14}_7 = 0$ we deduce that
\begin{equation} \label{eq:Hph4-2temp5}
D^7_1 \omega_7 = 0.
\end{equation}
Combining the first equation in~\eqref{eq:Hph4-2temp4} with~\eqref{eq:Hph4-2temp5} we find by Theorem~\ref{thm:harmonic1} that, considered as a $1$-form, $\omega_7 \in \mathcal{H}^1$ and in particular
\begin{equation} \label{eq:Hph4-2temp6}
D^7_{14} \omega_7 = 0 \qquad \text{and} \qquad D^7_{27} \omega_7 = 0.
\end{equation}
From Figure~\ref{figure:d}, the condition $\pi_{14} \dd \omega = 0$ gives $- D^7_{14} \omega_7 + D^{27}_{14} \omega_{27} = 0$, which, by the first equation in~\eqref{eq:Hph4-2temp6} implies that
\begin{equation} \label{eq:Hph4-2temp7}
D^{27}_{14} \omega_{27} = 0.
\end{equation}
Recalling from~\eqref{eq:Hph4-2LBo} that $\omega_7 = - 3 D^{14}_7 \alpha_{14}$, substituting~\eqref{eq:Hph4-2temp3} into~\eqref{eq:Hph4-2temp} and using the first equation in~\eqref{eq:Hph4-2temp4} now gives
\begin{equation*}
0 = -8 D^{27}_7 D^7_{27} \alpha_7 - 8 D^7_7 \omega_7 = -8 D^{27}_7 D^7_{27} \alpha_7,
\end{equation*}
which by Remark~\ref{rmk:IBP} implies that
\begin{equation} \label{eq:Hph4-2temp8}
D^7_{27} \alpha_7 = 0.
\end{equation}
Combining~\eqref{eq:Hph4-2temp8} with~\eqref{eq:Hph4-2temp3} and using Theorem~\ref{thm:harmonic1}, we find that $\alpha_7$ is harmonic.

Recalling from~\eqref{eq:Hph4-2LBo} that $\omega_{27} = - 2 D^7_{27} \alpha_7 + D^{14}_{27} \alpha_{14}$, substituting~\eqref{eq:Hph4-2temp8} and taking $D^{27}_{27}$, we obtain by the relations in~\eqref{eq:d-relations} that
\begin{equation*}
D^{27}_{27} \omega_{27} = D^{27}_{27} D^{14}_{27} \alpha_{14} = D^7_{27} D^{14}_7 \alpha_{14}.
\end{equation*}
Substituting $D^{14}_7 \alpha_{14} = - \tfrac{1}{3} \omega_7$ from~\eqref{eq:Hph4-2LBo} into the above expression and using the second equation in~\eqref{eq:Hph4-2temp6}, we find that
\begin{equation} \label{eq:Hph4-2temp9}
D^{27}_{27} \omega_{27} = - \tfrac{1}{3} D^7_{27} \omega_{7} = 0.
\end{equation}
Combining the second equation in~\eqref{eq:Hph4-2temp4}, equation~\eqref{eq:Hph4-2temp7}, and~\eqref{eq:Hph4-2temp9}, with equation~\eqref{eq:Laplacian}, we deduce that $\omega_{27}$ is a harmonic $\Omega^4_{27}$ form. We already showed that $\omega_7$ is a harmonic $\Omega^4_7$ form, and that $\omega_1 = 0$. Thus we have $\omega \in\mathcal{H}^4$ and moreover we assumed that $\omega \in (\im \dd)^4$. By Hodge theory we conclude that $\omega = 0$ as claimed.
\end{proof}

\begin{prop} \label{prop:Hph4}
We have $H^4_{\ph} \cong \mathcal{H}^4 \oplus \big( (\im \dd)^4 \cap (\ker \mathcal{L}_B)^4 \big)$. 
\end{prop}
\begin{proof}
In the proof of Proposition~\ref{prop:Hph3}, we showed that
\begin{equation*}
H^3_{\ph} = (\ker \mathcal{L}_B)^3 = \mathcal{H}^3 \oplus \big( (\im \ds)^3 \cap (\ker \mathcal{L}_B)^3 \big).
\end{equation*}
We also have $\mathcal{H}^3 \subset (\ker \mathcal{L}_B)^3$ by~\eqref{eq:LBLKonH}. Thus
\begin{equation*}
\mathcal{H}^3 \subseteq (\ker \mathcal{L}_B)^3 \subseteq \mathcal{H}^3 \oplus (\im \ds)^3.
\end{equation*}
Taking orthogonal complements of the above chain of nested subspaces and using the Hodge decomposition $\Omega^3 = \mathcal{H}^3 \oplus (\im \dd)^3 \oplus (\im \ds)^3$, we find
\begin{equation*}
(\im \dd)^3 \oplus (\im \ds)^3 \supseteq ((\ker \mathcal{L}_B)^3)^{\perp} \supseteq (\im \dd)^3.
\end{equation*}
Taking the Hodge star of the above chain of nested subspaces and using $(\im \mathcal{L}_B)^4 = \st ((\ker \mathcal{L}_B)^3)^{\perp}$ from~\eqref{eq:LBregular} we obtain
\begin{equation*}
(\im \ds)^4 \subseteq (\im \mathcal{L}_B)^4 \subseteq (\im \ds)^4 \oplus (\im \dd)^4.
\end{equation*}
Applying Lemma~\ref{lemma:linalg}(i) to the above yields
\begin{equation} \label{eq:Hph4temp}
(\im \mathcal{L}_B)^4 = (\im \ds)^4 \oplus \big( (\im \dd)^4 \cap (\im \mathcal{L}_B)^4 \big).
\end{equation}

Now recall that $\mathcal{H}^4 \subseteq (\ker \mathcal{L}_B)^4$ by~\eqref{eq:LBLKonH}. Thus we have
\begin{equation*}
\mathcal{H}^4 \subseteq (\ker \mathcal{L}_B)^4 \subseteq \Omega^4 = \mathcal{H}^4 \oplus (\im \dd)^4 \oplus (\im \ds)^4.
\end{equation*}
Applying Lemma~\ref{lemma:linalg}(i) to the above and using Lemma~\ref{lemma:Hph4-1} gives
\begin{equation} \label{eq:Hph4temp2}
(\ker \mathcal{L}_B)^4 = \mathcal{H}^4 \oplus \big( (\im \ds)^4 \cap (\ker \mathcal{L}_B)^4\big) \oplus \big( (\im \dd)^4 \cap (\ker \mathcal{L}_B)^4 \big).
\end{equation}

Thus, applying Lemma~\ref{lemma:linalg}(ii) to~\eqref{eq:Hph4temp},~\eqref{eq:Hph4temp2}, as subspaces of $\Omega^4 = \mathcal{H}^4 \oplus (\im \dd)^4 \oplus (\im \ds)^4$, we obtain
\begin{equation} \label{eq:Hph4temp3}
(\ker \mathcal{L}_B)^4 \cap (\im \mathcal{L}_B)^4 = \big( (\im \ds)^4 \cap (\ker \mathcal{L}_B)^4 \big) \oplus \big( (\im \dd)^4 \cap (\ker \mathcal{L}_B)^4 \cap (\im \mathcal{L}_B)^4 \big).
\end{equation}
By Lemma~\ref{lemma:Hph4-2}, equation~\eqref{eq:Hph4temp3} simplifies to
\begin{equation} \label{eq:Hph4temp4}
(\ker \mathcal{L}_B)^4 \cap (\im \mathcal{L}_B)^4 = (\im \ds)^4 \cap (\ker \mathcal{L}_B)^4.
\end{equation}

Hence, by~\eqref{eq:Hph4temp2} and~\eqref{eq:Hph4temp4}, we have
\begin{align*}
\mathcal{H}^4_{\ph} & = \frac{(\ker \mathcal{L}_B)^4}{(\ker \mathcal{L}_B)^4 \cap (\im \mathcal{L}_B)^4} \\
& = \frac{\mathcal{H}^4 \oplus \big( (\im \ds)^4 \cap (\ker \mathcal{L}_B)^4 \big) \oplus \big( (\im \dd)^4 \cap (\ker \mathcal{L}_B)^4 \big)}{(\im \ds)^4 \cap (\ker \mathcal{L}_B)^4}\\
&\cong \mathcal{H}^4 \oplus (\im \dd)^4 \cap (\ker \mathcal{L}_B)^4
\end{align*}
as claimed.
\end{proof}

\begin{lemma} \label{lemma:k3}
We have $(\im \ds)^3 \cap (\ker \mathcal L_B)^3=(\im \ds)^3 \cap (\ker \mathcal{L}^*_B)^3$.
\end{lemma}
\begin{proof}
Let $\omega = \ds (\gamma_1 + \gamma_7 + \gamma_{27}) \in (\im \ds)^3$ where $\gamma_1 + \gamma_7 + \gamma_{27} \in \Omega^4_1 \oplus \Omega^4_7 \oplus \Omega^4_{27}$. From $\ds = \st \dd \st$ on $\Omega^5$ and Figures~\ref{figure:d} we find that
\begin{equation} \label{eq:k3temp}
\ds (\gamma_1 + \gamma_7 + \gamma_{27}) = \tfrac{4}{3} D^7_1 \gamma_7 + (- D^1_7 \gamma_1 - \tfrac{3}{2} D^7_7 \gamma_7 + D^{27}_7 \gamma_{27}) + (-D^7_{27} \gamma_7 + D^{27}_{27} \gamma_{27}).
\end{equation}
Using~\eqref{eq:k3temp} and Figure~\ref{figure:LB}, we have
\begin{align*}
\mathcal L_B \ds (\gamma_1 + \gamma_7 + \gamma_{27}) & = - 2 D^1_7 ( \tfrac{4}{3} D^7_1 \gamma_7 ) + 3 D^7_{14} ( - D^1_7 \gamma_1 - \tfrac{3}{2} D^7_7 \gamma_7 + D^{27}_7 \gamma_{27}) \\
& \qquad {} + (-\tfrac{8}{3} D^{27}_7 + D^{27}_{14}) (-D^7_{27} \gamma_7 + D^{27}_{27} \gamma_{27}) \\
& = (-\tfrac{8}{3} D^1_7 D^7_1 \gamma_7 + \tfrac{8}{3} D^{27}_7 D^7_{27} \gamma_7 - \tfrac{8}{3} D^{27}_7 D^{27}_{27} \gamma_{27}) \\
& \qquad {} + (-3 D^7_{14} D^1_7 \gamma_1 - \tfrac{9}{2} D^7_{14} D^7_7 \gamma_7 + 3 D^7_{14} D^{27}_7 \gamma_{27} - D^{27}_{14} D^7_{27} \gamma_7 + D^{27}_{14} D^{27}_{27} \gamma_{27}).
\end{align*}
Using the various relations in~\eqref{eq:d-relations}, the above expression simplifies to
\begin{equation} \label{eq:k3temp2}
\begin{aligned}
\mathcal L_B \ds (\gamma_1 + \gamma_7 + \gamma_{27}) & = -6 D^7_7 D^7_7 \gamma_7 + 4 D^7_7 D^{27}_7 \gamma_{27} - 6 D^7_{14} D^7_7 \gamma_7 + 4 D^7_{14} D^{27}_7 \gamma_{27} \\
& = 2 D^7_7 ( -3 D^7_7 \gamma_7 + 2 D^{27}_7 \gamma_{27}) + 2 D^7_{14} ( -3 D^7_7 \gamma_7 + 2 D^{27}_7 \gamma_{27}).
\end{aligned}
\end{equation}
Using $\mathcal L_B^* = - \st \mathcal{L}_B \st$ from~\eqref{eq:derivationsstar}, equation~\eqref{eq:k3temp}, and Figure~\ref{figure:LB} again, we also have that
\begin{align*}
\mathcal{L}^*_B \ds (\gamma_1 + \gamma_7 + \gamma_{27}) & = - 3 D^1_7 ( \tfrac{4}{3} D^7_1 \gamma_7 ) + 6 D^7_7 ( - D^1_7 \gamma_1 - \tfrac{3}{2} D^7_7 \gamma_7 + D^{27}_7 \gamma_{27}) \\
& \qquad {} - 4 D^{27}_7 (-D^7_{27} \gamma_7 + D^{27}_{27} \gamma_{27}) \\
= & -4 D^1_7 D^7_1 \gamma_7 - 6 D^7_7 D^1_7 \gamma_1 - 9 D^7_7 D^7_7 \gamma_7 + 6 D^7_7 D^{27}_7 \gamma_{27} \\
& \qquad {} + 4 D^{27}_7 D^7_{27} \gamma_7 - 4 D^{27}_7 D^{27}_{27} \gamma_{27}.
\end{align*}
Using the various relations in~\eqref{eq:d-relations}, the above expression simplifies to
\begin{equation} \label{eq:k3temp3}
\begin{aligned}
\mathcal{L}^*_B \ds (\gamma_1 + \gamma_7 + \gamma_{27}) & = -18 D^7_7 D^7_7 \gamma_7 + 12 D^7_7 D^{27}_7 \gamma_{27}) \\
& = 6 D^7_7 (-3 D^7_7 \gamma_7 + 2 D^{27}_7 \gamma_{27}).
\end{aligned}
\end{equation}
Thus for $\omega \in (\im \ds)^3$ we conclude that
\begin{align*}
\omega \in (\ker \mathcal L_B)^3 & \iff
\left\{
\begin{aligned}
D^7_7 (-3 D^7_7 \gamma_7 + 2 D^{27}_7 \gamma_{27}) & = 0, \\
D^7_{14} (-3 D^7_7 \gamma_7 + 2 D^{27}_7 \gamma_{27}) & = 0,
\end{aligned}
\right\} & & \text{by~\eqref{eq:k3temp2}} \\
& \iff D^7_7 (- 3 D^7_7 \gamma_7 + 2 D^{27}_7 \gamma_{27} = 0) & & \text{by Theorem~\ref{thm:harmonic1}} \\
& \iff \omega \in (\ker \mathcal{L}^*_B)^3 & & \text{by~\eqref{eq:k3temp3}}
\end{align*}
which is what we wanted to show.
\end{proof}

\begin{cor} \label{H34cor}
We have
\begin{align*}
H^3_{\ph} & = \mathcal{H}^3 \oplus \big( (\im \ds)^3 \cap (\ker \mathcal L_B)^3 \cap (\ker \mathcal{L}^*_B)^3 \big), \\
H^4_{\ph} & = \mathcal{H}^4 \oplus \big( (\im \dd)^4 \cap (\ker \mathcal L_B)^4 \cap (\ker \mathcal{L}^*_B)^4 \big).
\end{align*}
\end{cor}
\begin{proof}
Lemma~\ref{lemma:k3} says that
\begin{equation*}
(\im \ds)^3 \cap (\ker \mathcal L_B)^3 =(\im \ds)^3 \cap (\ker \mathcal L_B)^3 \cap (\ker \mathcal{L}^*_B)^3 = (\im \ds)^3 \cap (\ker \mathcal{L}^*_B)^3.
\end{equation*}
Applying $\st$ to the above equation and using~\eqref{eq:derivationsstar} gives
\begin{equation*}
(\im \dd)^4 \cap (\ker \mathcal L_B)^4 = (\im \dd)^4 \cap (\ker \mathcal L_B)^4 \cap (\ker \mathcal{L}^*_B)^4 = (\im \dd)^4 \cap (\ker \mathcal{L}^*_B)^4.
\end{equation*}
The claim now follows from Propositions~\ref{prop:Hph3} and~\ref{prop:Hph4}.
\end{proof}

\subsection{The main theorem on $\mathcal L_B$-cohomology} \label{sec:Hphthm}

We summarize the results of Section~\ref{sec:cohom} in the following theorem, which is intentionally stated in a way to mirror Theorem~\ref{thm:KLS}.

\begin{thm} \label{thm:Hph} The following relations hold.
\begin{itemize} \setlength\itemsep{-1mm}
\item $H^k_{\ph} \cong H^k_{dR}$ for $k=0,1,2,5,6,7$.
\item $H^k_{\ph}$ is infinite-dimensional for $k = 3,4$.
\item There is a canonical injection $\mathcal{H}^k \hookrightarrow H^k_{\ph}$ for all $k$.
\item The Hodge star induces isomorphisms $\st: H^k_{\ph} \cong H^{7-k}_{\ph}$.
\end{itemize}
\end{thm}
\begin{proof}
All that remains to show is that $H^3_{\ph}$ is indeed infinite-dimensional. But observe by~\eqref{eq:k3temp2} that for all $\alpha \in \Omega^4_1$, we have $\mathcal L_B \ds \alpha = 0$. Therefore, $\{ \ds \alpha : \alpha \in \Omega^4_1 \} \cong \im D^1_7 \cong (\im \dd)^1$ is an infinite-dimensional subspace of $(\im \ds)^3 \cap (\ker \mathcal L_B)^3 \subseteq H^3_{\ph}$.
\end{proof}

\section{An application to `almost' formality} \label{sec:formality}

In this section we consider an application of our results to the question of formality of compact torsion-free $\G$~manifolds. We discover a new topological obstruction to the existence of torsion-free $\G$-structures on compact manifolds, and discuss an explicit example in detail.

\subsection{Formality and Massey triple products} \label{sec:Massey}

Recall from~\eqref{eq:commdL} that $\dd$ commutes with $\mathcal L_B$. Hence $\dd$ induces a natural map
\begin{equation*}
\dd: H^k_{\ph} \to H^{k+1}_{\ph}.
\end{equation*}
Also, because $\mathcal L_B$ is a derivation, it is easy to check that the wedge product on $\Omega^{\bu}$ descends to $H^{\bu}_{\ph}$, with the Leibniz rule $\dd (\omega \wedge \eta) = (\dd \omega) \wedge \eta + (-1)^{|\omega|} \omega \wedge (\dd \eta)$ still holding on $H^{\bu}_{\ph}$. These two facts say that the complex $(H^{\bu}_{\ph}, \dd)$ is a \emph{differential graded algebra}, henceforth abbreviated \emph{dga}.

Additionally, because $[ \dd, \mathcal{L}_B ] = 0$, we also have that $( (\ker \mathcal L_B)^{\bu}, \dd)$ is a subcomplex of the de Rham complex $( \Omega^{\bu}, \dd)$. The natural injection and projection give homomorphisms of dga's
\begin{equation*}
(\Omega^{\bu}, \dd) \hookleftarrow ((\ker \mathcal L_B)^{\bu}, \dd) \twoheadrightarrow (H^{\bu}_{\ph}, \dd).
\end{equation*}

One goal of this section is to show that these two homomorphisms of dga's are both \emph{quasi-isomorphisms}. This means that they induce \emph{isomorphisms} on the cohomologies of the complexes. As mentioned in the introduction, some of the results in this section appeared earlier in work of Verbitsky~\cite{Verbitsky}. For example, our Proposition~\ref{prop:Verbitsky} is exactly~\cite[Proposition 2.21]{Verbitsky}, with the same proof. However, the proof of~\cite[Proposition 2.19]{Verbitsky} has several errors. The critical error is the following: first Verbitsky correctly shows that $\alpha - \Pi \alpha$ is an element of both $(\im \dd_c + \im \ds_c)$ and $(\im \ds_c)^{\perp}$. But then he incorrectly concludes that $\alpha - \Pi \alpha$ must be an element of $\im \dd_c$. This conclusion is only valid if $(\dd_c)^2 = 0$, which is not true in general. We give a correct proof of this result, which is our Proposition~\ref{prop:quasi-isom}. One consequence is the result about the Massey triple product in our Corollary~\ref{cor:Massey}, which appears to be new.

\begin{prop}[Verbitsky~\cite{Verbitsky}] \label{prop:Verbitsky}
The inclusion $((\ker \mathcal L_B)^{\bu}, \dd) \hookrightarrow (\Omega^{\bu}, \dd)$ is a quasi-isomorphism.
\end{prop}
\begin{proof}
This is proved in~\cite[Proposition 2.11]{Verbitsky}. We reproduce the short proof here for completeness and convenience of the reader. Since the differential for both complexes $\Omega^{\bu}$ and $(\ker \mathcal L_B)^{\bu}$ is the same exterior derivative $\dd$, we will omit it from the notation for simplicity.

It is well-known that the Hodge Laplacian $\Delta$ determines an eigenspace decomposition $\Omega^k = \oplus_{\lambda} \Omega^k_{\lambda}$ where the sum is over all eigenvalues $\lambda$ of $\Delta$, which form a discrete set of non-negative real numbers, and $\Omega^k_{\lambda} = \{ \alpha \in \Omega^k : \Delta \alpha = \lambda \alpha \}$ are the associated eigenspaces. Note that $\Omega^k_0 = \mathcal{H}^k$ is the space of harmonic $k$-forms. It is well-known that the cohomology of $\Omega^k_{\lambda}$ is trivial for $\lambda > 0$. This is because, if $\alpha \in \Omega^k_{\lambda}$ with $\lambda > 0$ and $\dd \alpha = 0$, then
\begin{equation} \label{eq:Verbitsky-temp}
\alpha = \tfrac{1}{\lambda} \Delta \alpha = \tfrac{1}{\lambda} (\dd \ds \alpha + \ds \dd \alpha) = \dd (\tfrac{1}{\lambda} \ds \alpha)
\end{equation}
is exact.

By~\eqref{eq:LBLKLap}, the operator $\mathcal L_B$ commutes with $\Delta$, and thus we obtain a decomposition
\begin{equation*}
(\ker \mathcal L_B)^k = \oplus_{\lambda} \big( \Omega^k_{\lambda} \cap (\ker \mathcal L_B)^k \big).
\end{equation*}
Note by~\eqref{eq:LBLKonH} that $\Omega^k_0 \cap (\ker \mathcal L_B)^k = \mathcal{H}^k \cap (\ker \mathcal L_B)^k = \mathcal{H}^k = \Omega^k_0$. Thus it remains to show that the cohomology of $\Omega^k_{\lambda} \cap (\ker \mathcal L_B)^k$ is also trivial for all $\lambda > 0$. But if $\alpha \in \Omega^k_{\lambda} \cap (\ker \mathcal L_B)^k$, we have $\mathcal L_B \alpha = 0$ and $\alpha = \dd (\tfrac{1}{\lambda} \ds \alpha)$ by~\eqref{eq:Verbitsky-temp}. Since $\mathcal L_B$ commutes with $\ds$ by~\eqref{eq:LBLKdscom}, we have $\mathcal L_B (\tfrac{1}{\lambda} \ds \alpha) = \tfrac{1}{\lambda} \ds \mathcal L_B \alpha = 0$, so the class of $\alpha$ in the cohomology of $(\ker \mathcal L_B)^k$ is indeed trivial.
\end{proof}

In Section~\ref{sec:cohom}, while computing $H^{\bu}_{\ph}$, we explicitly computed the complex $((\ker \mathcal L_B)^{\bu}, \dd)$. The results are collected in Figure~\ref{figure:complex1}. The isomorphisms displayed in Figure~\ref{figure:complex1} are explained in Corollary~\ref{cor:formality}.

\begin{figure}[H]
\begin{equation*}
\xymatrix {
\mathcal{H}^0 \ar@[purple][d]^{\purple{0}} \\
\mathcal{H}^1 \ar@[purple][d]^{\purple{0}} \\
\mathcal{H}^2 \ar@[purple][d]^{\purple{0}} \\
\mathcal{H}^3 \oplus \teal{\big( (\im\ds)^3 \cap (\ker \mathcal L_B)^3 \cap (\ker \mathcal L_B^*)^3 \big)} \ar@[purple][]!<-20ex,-2ex>;[d]!<-30ex,1ex>_{\purple{0}} \ar@[teal][]!<-5ex,-2ex>;[d]!<-10ex,1ex>_{\teal{\cong}} \\
\mathcal{H}^4 \oplus \teal{\big( (\im \dd)^4 \cap (\ker \mathcal L_B)^4 \cap (\ker \mathcal L_B^*)^4 \big)} \oplus \orange{\big( (\im \ds)^4 \cap (\ker \mathcal L_B)^4 \big)} \ar@[purple][]!<-30ex,-2ex>;[d]!<-20ex,1ex>^{\purple{0}} \ar@[orange][]!<20ex,-2ex>;[d]!<5ex,1ex>_{\orange{\cong}} \\
\mathcal{H}^5 \oplus \blue{(\im \ds)^5} \oplus \orange{\big( (\im \dd)^5 \cap (\ker \mathcal L_B)^5 \big)} \ar@[purple][]!<-17ex,-2ex>;[d]!<-10ex,1ex>^{\purple{0}} \ar@[blue][]!<-8ex,-2ex>;[d]!<-2.5ex,1ex>^{\blue{\cong}} \\
\mathcal{H}^6 \oplus \blue{(\im \dd)^6} \oplus \gray{(\im \ds)^6} \ar@[purple][]!<-10ex,-2ex>;[d]!<-3.8ex,1ex>^{\purple{0}} \ar@[gray][]!<6ex,-2ex>;[d]!<3ex,1ex>_{\gray{\cong}} \\
\mathcal{H}^7 \oplus \gray{(\im \dd)^7}
}
\end{equation*}
\caption{The complex $((\ker \mathcal L_B)^{\bu}, \dd)$.} \label{figure:complex1}
\end{figure}
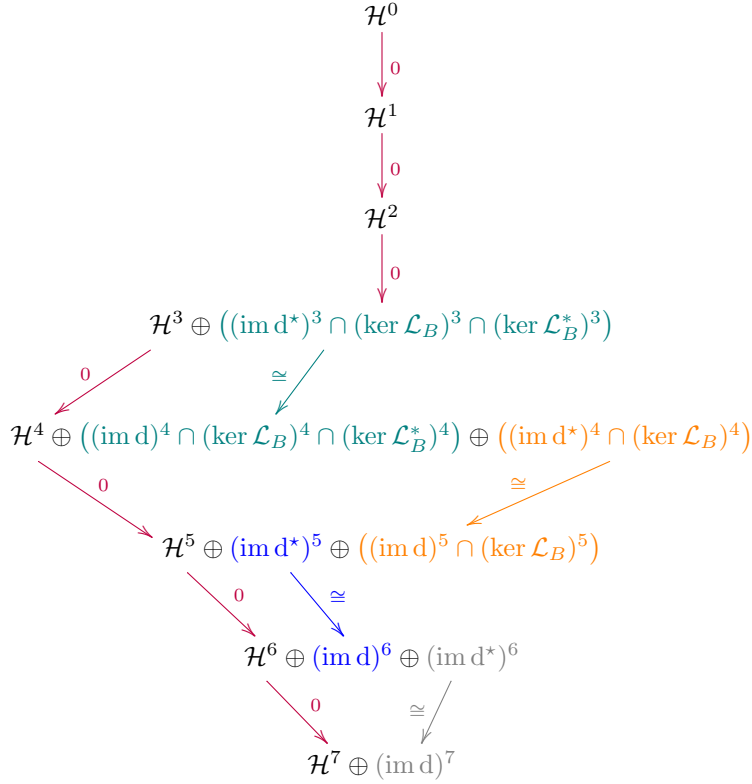

\begin{cor} \label{cor:formality}
For all $0 \leq k \leq 7$, we have $(\im \dd)^k \cap (\ker \mathcal L_B)^k = \dd (\ker \mathcal L_B)^{k-1}$.
\end{cor}
\begin{proof}
Let $\Omega^k = \mathcal{H}^k \oplus (\im \dd)^k \oplus (\im \ds)^k$ denote the Hodge decomposition of $\Omega^k$. For simplicity in this proof we will write $A^k = \mathcal{H}^k$, $B^k = (\im \dd)^k$, and $C^k = (\im \ds)^k$. Thus $\Omega^k = A^k \oplus B^k \oplus C^k$. We can see from Figure~\ref{figure:complex1} that for all $0 \leq k \leq 7$, we have $(\ker \mathcal L_B)^k = A^k \oplus \tilde B^k \oplus \tilde C^k$, where $\tilde B^k$ and $\tilde C^k$ are subspaces of $B^k$ and $C^k$, respectively. Depending on $k$, we can have $\tilde B^k = 0$, $\tilde B^k = B^k$, or $0 \subsetneq \tilde B^k \subsetneq B^k$ and similarly for $\tilde C^k$. By Hodge theory, $(\ker \dd)^k = A^k \oplus B^k$, so applying Lemma~\ref{lemma:linalg}(ii) we find that
\begin{equation} \label{eq:formalitycor1}
(\ker \mathcal L_B)^k \cap (\ker \dd)^k = A^k \oplus \tilde B^k.
\end{equation}
Applying $\dd$ to $(\ker \mathcal L_B)^{k-1}$, we have
\begin{equation} \label{eq:formalitycor2}
\dd (\ker \mathcal L_B)^{k-1} = \dd (\tilde C^{k-1}) \subseteq \tilde B^k.
\end{equation}
By Proposition~\ref{prop:Verbitsky}, the cohomology of $(\Omega^k, \dd)$ equals the cohomology of $( (\ker \mathcal L_B)^k, \dd)$. But by by Hodge theory the cohomology of $(\Omega^k, \dd)$ is $\mathcal{H}^k = A^k$, and equations~\eqref{eq:formalitycor1} and~\eqref{eq:formalitycor2} say that the cohomology of $( (\ker \mathcal L_B)^k, \dd)$ is $A^k \oplus \big( \tilde B^k / (\dd \tilde C^{k-1}) \big)$. Thus in fact we have $\dd \tilde C^{k-1} = \tilde B^k$, and since $\dd$ is injective on $C^k$, we deduce that
\begin{equation} \label{eq:formalitycor3}
\text{$\dd$ maps $\tilde C^{k-1}$ isomorphically onto $\tilde B^k$ for all $0 \leq k \leq 7$.}
\end{equation}
From $(\im \dd)^k \cap (\ker \mathcal L_B)^k = \tilde B^k$, and $\dd (\ker \mathcal L_B)^{k-1} = \dd (A^{k-1} \oplus \tilde B^{k-1} \oplus \tilde C^{k-1}) = \dd \tilde C^{k-1}$, we conclude that $(\im \dd)^k \cap (\ker \mathcal L_B)^k = \dd (\ker \mathcal L_B)^{k-1}$ as claimed.
\end{proof}

\begin{rmk} \label{rmk:Kahler}
Corollary~\ref{cor:formality} may be related to a $\G$-analogue of the \emph{generalized} $\partial \bar \partial$-lemma, called the $\dd \mathcal{L}_J$-lemma, introduced by the authors in~\cite{CKT} in the context of $\U{m}$-structures. See~\cite[Equation (3.27)]{CKT}.
\end{rmk}

\begin{prop} \label{prop:quasi-isom}
The quotient map $((\ker \mathcal L_B)^{\bu}, \dd) \twoheadrightarrow (H^{\bu}_{\ph}, \dd)$ is a quasi-isomorphism.
\end{prop}
\begin{proof}
We have a short exact sequence of chain complexes
\begin{equation*}
0 \to ((\ker \mathcal L_B)^{\bu} \cap (\im \mathcal L_B)^{\bu}, \dd) \to ((\ker \mathcal L_B)^{\bu}, \dd) \twoheadrightarrow (H^{\bu}_{\ph}, \dd) \to 0,
\end{equation*}
so it suffices to show that the cohomology of $((\ker \mathcal L_B)^{\bu} \cap (\im \mathcal L_B)^{\bu}, \dd)$ is trivial. In Section~\ref{sec:cohom}, while computing $H^{\bu}_{\ph}$, we explicitly computed the complex $((\ker \mathcal L_B)^{\bu} \cap (\im \mathcal L_B)^{\bu}, \dd)$. The results are collected in Figure~\ref{figure:complex2}. The isomorphisms in Figure~\ref{figure:complex2} are a subset of the isomorphisms from Figure~\ref{figure:complex1} and are coloured in the same way. It is clear from Figure~\ref{figure:complex2} that the cohomology of $((\ker \mathcal L_B)^{\bu} \cap (\im \mathcal L_B)^{\bu}, \dd)$ is trivial.
\end{proof}

\begin{figure}[H]
\begin{equation*}
\xymatrix {
0 \ar@[purple][d]^{\purple{0}} \\
0 \ar@[purple][d]^{\purple{0}} \\
0 \ar@[purple][d]^{\purple{0}} \\
0 \ar@[purple][d]^{\purple{0}} \\
\orange{(\im \ds)^4 \cap (\ker \mathcal L_B)^4} \ar@[orange][]!<-1ex,-2ex>;[d]!<4ex,1ex>^{\orange{\cong}} \\
\blue{(\im \ds)^5} \oplus \orange{\big( (\im \dd)^5 \cap (\ker \mathcal L_B)^5 \big)} \ar@[blue][]!<-11ex,-2ex>;[d]!<-6ex,1ex>^{\blue{\cong}} \\
\blue{(\im \dd)^6} \oplus \gray{(\im \ds)^6} \ar@[gray][]!<4.5ex,-2ex>;[d]!<-.5ex,1ex>_{\gray{\cong}} \\
\gray{(\im \dd)^7}
}
\end{equation*}
\caption{The complex $((\ker \mathcal L_B \cap \im \mathcal L_B)^{\bu}, \dd)$.} \label{figure:complex2}
\end{figure}

The next two definitions are taken from~\cite[Section 3.A]{Huybrechts}.

\begin{defn} \label{defn:equiv}
Let $(A, \dd_A)$ and $(B, \dd_B)$ be two differential graded algebras (dga's). We say that $A$ and $B$ are \emph{equivalent} if there exists a finite sequence of \emph{dga quasi-isomorphisms}
\begin{equation*}
\xymatrix {
& (C_1, \dd_{C_1}) \ar[ld] \ar[rd] & & \cdots \ar[ld] \ar [rd] & & (C_n, \dd_{C_n}) \ar[ld] \ar[rd] & \\
(A, \dd_A) & & (C_2, \dd_{C_2}) & & \cdots & & (B, \dd_B).\\
}
\end{equation*}
A dga $(A, \dd_A)$ is called \emph{formal} if it is equivalent to a dga $(B, \dd_B)$ with $\dd_B = 0$.
\end{defn}

It is well-known~\cite[Section 3.A]{Huybrechts} that a \emph{compact K\"ahler manifold} is formal. That is, the de Rham complex of a compact K\"ahler manifold is equivalent to a dga with zero differential. It is still an open question whether or not compact torsion-free $\G$~manifolds are formal. We show in Theorem~\ref{thm:almost-formal} below that compact torsion-free $\G$~manifolds are `almost formal' in the sense that the de Rham complex is equivalent to a dga which has \emph{only one nonzero differential}.

\begin{thm} \label{thm:almost-formal}
The de Rham complex of a compact torsion-free $\G$~manifold $(\Omega^{\bu}, \dd)$ is equivalent to $(H^{\bu}_{\ph}, \dd)$, which is a dga with all  differentials trivial except for $\dd: H^3_{\ph} \to H^4_{\ph}$.
\end{thm}
\begin{proof}
In Section~\ref{sec:cohom}, we explicitly computed the complex $(H^{\bu}_{\ph}, \dd)$. The results are collected in Figure~\ref{figure:complex3}. The isomorphism in Figure~\ref{figure:complex3} appeared already in Figure~\ref{figure:complex1} and is coloured in the same way. The zero maps in Figure~\ref{figure:complex3} are a consequence of $\mathcal{H}^k \subseteq (\ker \dd)^k$.
\end{proof}

\begin{figure}[H]
\begin{equation*}
\xymatrix {
\mathcal{H}^0 \ar@[purple][d]^{\purple{0}} \\
\mathcal{H}^1 \ar@[purple][d]^{\purple{0}} \\
\mathcal{H}^2 \ar@[purple][]!<1ex,-1ex>;[d]!<-22ex,1ex>_{\purple{0}} \\
\mathcal{H}^3 \oplus \teal{\big( (\im \ds)^3 \cap (\ker \mathcal L_B)^3 \cap (\ker \mathcal{L}^*_B)^3 \big)} \ar@[purple][]!<-19ex,-2ex>;[d]!<-19ex,1ex>_{\purple{0}} \ar@[teal][]!<1ex,-2ex>;[d]!<1ex,1ex>_{\teal{\cong}}  \\
\mathcal{H}^4 \oplus \teal{\big( (\im \dd)^4 \cap (\ker \mathcal L_B)^4 \cap (\ker \mathcal{L}^*_B)^4 \big)} \ar@[purple][]!<-22ex,-1ex>;[d]!<1ex,1ex>^{\purple{0}} \\
\mathcal{H}^5 \ar@[purple][d]^{\purple{0}} \\
\mathcal{H}^6 \ar@[purple][d]^{\purple{0}} \\
\mathcal{H}^7
}
\end{equation*}
\caption{The complex $(H^{\bu}_{\ph}, \dd)$.} \label{figure:complex3}
\end{figure}

One consequence of almost-formality is that \emph{most of the Massey triple products of the de Rham complex will vanish}. This is established in Corollary~\ref{cor:Massey} below.

\begin{defn} \label{defn:masseytriprod}
Let $(A,\dd_A)$ be a dga, and denote by $H^k (A)$ the degree $k$ cohomology of $A$ with respect to $\dd_A$. Let $[\alpha]\in H^p(A)$, $[\beta] \in H^q(A)$, $[\gamma] \in H^r(A)$ be cohomology classes satisfying
\begin{equation*}
[\alpha] [\beta] = 0 \in H^{p+q}(A) \qquad \text{and} \qquad [\beta] [\gamma] = 0 \in H^{q+r}(A).
\end{equation*}
Then $\alpha \beta =\dd f$ and $\beta \gamma = \dd g$ for some $f \in A^{p+q-1}$ and $g \in A^{q+r-1}$. Consider the class
\begin{equation*}
[f \gamma - (-1)^p \alpha g] \in H^{p+q+r-1}(A).
\end{equation*}
It can be checked that this class is well-defined up to an element of $H^{p+q-1} \cdot H^r + H^p \cdot H^{q+r-1}$. That is, it is well defined as an element of the quotient
\begin{equation*}
\frac{H^{p+q+r-1}(A)}{H^{p+q-1} \cdot H^r + H^p \cdot H^{q+r-1}}.
\end{equation*}
We call this element the \emph{Massey triple product} and write it as $\langle [\alpha], [\beta], [\gamma] \rangle$. It is easy to see that the Massey triple product is linear in each of its three arguments.
\end{defn}

In the following we only consider the case when $(A,\dd_A) = (\Omega^*, \dd)$ is the dga of smooth differential forms. If the dga $(A, \dd_A)$ is formal, then \emph{all the Massey triple products vanish} due to the naturality of the triple product (see~\cite[Proposition 3.A.33]{Huybrechts} for details). In fact, the proof of~\cite[Proposition 3.A.33]{Huybrechts} actually yields the following more general result.

\begin{cor} \label{cor:weak-Massey}
Let $(A, \dd_A)$ be a dga such that the differentials $\dd_A$ are all zero except for $\dd : A^{k-1} \to A^k$. Then if the Massey triple product $\langle [\alpha], [\beta], [\gamma] \rangle$ is defined and we have $|\alpha| + |\beta| \neq k$ and $|\beta| + |\gamma| \neq k$, then $\langle [\alpha], [\beta], [\gamma] \rangle = 0$.
\end{cor}

Combining Corollary~\ref{cor:weak-Massey} and Theorem~\ref{thm:almost-formal} yields the following.

\begin{cor} \label{cor:Massey}
Let $M$ be a compact torsion-free $\G$~manifold. Consider cohomology classes $[\alpha]$, $[\beta]$, and $[\gamma] \in H^{\bu}_{\DR}$. If the Massey triple product $\langle [\alpha], [\beta], [\gamma] \rangle$ is defined and we have $|\alpha| + |\beta| \neq 4$ and $|\beta| + |\gamma| \neq 4$, then $\langle [\alpha], [\beta], [\gamma] \rangle = 0$.
\end{cor}

In Theorem~\ref{thm:irreducibleMassey} in the next section we establish a stronger version of Corollary~\ref{cor:Massey} when the holonomy of the metric on $M$ is exactly $\G$.

\subsection{A new topological obstruction to existence of torsion-free $\G$-structures} \label{sec:new-obstruction}

A key feature of the criterion in Corollary~\ref{cor:Massey} is that it is \emph{topological}. That is, \emph{it does not depend on the differentiable structure on $M$}. Therefore it gives \emph{a new topological obstruction} to the existence of torsion-free $\G$-structures on compact $7$-manifolds. There are several previously known topological obstructions to the existence of a torsion-free $\G$-structure on a compact $7$-manifold. These obstructions are discussed in detail in~\cite[Chapter 10]{Joyce}. We summarize them here. Let $\ph$ be a torsion-free $\G$-structure on a compact manifold $M$ with induced metric $g_{\ph}$. Let $b^k_M = \dim H^k_{\DR} (M)$. Then
\begin{equation} \label{eq:known-obstructions}
\left. \begin{aligned}
& b^3_M \geq b^1_M + b^0_M, \\
& b^2_M \geq b^1_M, \\
& b^1_M \in \{ 0, 1, 3, 7 \}, \\
& \text{if $g_{\ph}$ is not flat, then $p_1 (M) \neq 0$, where $p_1 (M)$ is the first Pontryagin class of $TM$, \phantom{A}} \\
& \text{if $g_{\ph}$ has full holonomy $\G$, then the fundamental group $\pi_1 (M)$ is finite}.
\end{aligned} \right\}
\end{equation}
Note that the first three conditions are simply obstructions to the existence of torsion-free $\G$-structures. The fourth condition can be used to rule out \emph{non-flat} torsion-free $\G$-structures,  and the fifth condition can be used to rule out \emph{non-irreducible} torsion-free $\G$-structures. In fact, the third condition determines the reduced holonomy of $g_{\ph}$, which is $\{1\}$, $\SU{2}$, $\SU{3}$, or $\G$, if $b^1_M = 7$, $3$, $1$, or $0$, respectively.

\begin{thm} \label{thm:irreducibleMassey}
Let $M$ be a compact torsion-free $\G$~manifold with full holonomy $\G$, and consider cohomology classes $[\alpha]$, $[\beta]$, and $[\gamma] \in H^{\bu}_{\DR}$. If the Massey triple product $\langle [\alpha], [\beta], [\gamma] \rangle$ is defined, then $\langle [\alpha], [\beta], [\gamma] \rangle = 0$ except possibly in the case when $|\alpha| = |\beta| = |\gamma| = 2$.
\end{thm}
\begin{proof}
Recall that the hypothesis of full holonomy $\G$ implies that $b^1_M = 0$, so $H^1_{\DR} = \{0\}$. Suppose $|\alpha| = 1$. Then $[\alpha] \in H^1_{\DR}$, so $[\alpha] = 0$, and by linearity if follows that $\langle [\alpha], [\beta], [\gamma] \rangle = 0$. The same argument holds if $|\beta| = 1$ or $|\gamma| = 1$. Suppose $|\alpha| = 0$. Then $\alpha$ is a constant function. The condition $[\alpha \beta] = [\alpha][\beta] = 0$ forces the form $\alpha \beta$ to be exact, so either $\alpha = 0$ (in which case the Massey product vanishes), or $\beta$ is exact, so $[\beta] = 0$ and again the Massey product vanishes. A similar argument holds if $|\beta| = 0$ or $|\gamma| = 0$.

Thus we must have $|\alpha|, |\beta|, |\gamma| \geq 2$ if the Massey product has any chance of being nontrivial. Moreover, since $\langle [\alpha], [\beta], [\gamma] \rangle$ lies in a quotient of $H^{|\alpha| + |\beta| + |\gamma| - 1}_{\DR}$, we also need $|\alpha| + |\beta| + |\gamma| \leq 8$. Finally, Corollary~\ref{cor:Massey} tells us that we must have either $|\alpha| + |\beta| = 4$ or $|\beta| + |\gamma| = 4$. Hence the only possibilities for the triple $(|\alpha|, |\beta|, |\gamma|)$ to obtain a nontrivial Massey product are $(2,2,2)$, $(2,2,3)$, $(2,2,4)$, $(3,2,2)$, and $(4,2,2)$. For $(2,2,3)$ or $(3,2,2)$, the Massey product lies in a quotient of $H^6_{\DR}$, which is zero since $b^6_M = b^1_M = 0$. For $(2,2,4)$ or $(4,2,2)$ the Massey product lies inside $H^7_{\DR}/ (H^2_{\DR} \cdot H^6_{\DR} + H^3_{\DR} \cdot H^4_{\DR})$, but $H^3_{\DR} \cdot H^4_{\DR} = H^7_{\DR}$ since $\ph \wedge \ps = 7 \vol$ is a generator of $H^7_{\DR}$. Thus in this case the quotient space is zero. We conclude that the only possibly nontrivial Massey product corresponds to the case $(|\alpha|, |\beta|, |\gamma|) = (2,2,2)$.
\end{proof}

In the remainder of this section we will apply our new criterion to a particular nontrivial example. Consider a smooth compact connected oriented $7$-manifold $M$ of the form $M = W \times L$, where $W$ and $L$ are smooth compact connected oriented manifolds of dimensions $3$ and $4$, respectively. In order for $M$ to admit $\G$-structures, we must have $w_2 (M) = 0$, where $w_2 (M)$ is the second Stiefel-Whitney class of $TM$, by~\cite[p 348--349]{LM}.

Take $W$ to be the \emph{real Iwasawa manifold}, which is defined to be the quotient of the set
\begin{equation*}
\left \{ \begin{pmatrix} 1 & t_1 & t_2 \\ 0 & 1 & t_3 \\ 0 & 0 & 1 \end{pmatrix} : t_1, t_2, t_3 \in \R \right \} \cong \R^3
\end{equation*}
by the left multiplication of the group
\begin{equation*}
\left \{ \begin{pmatrix} 1 & a & b \\ 0 & 1 & c \\ 0 & 0 & 1 \end{pmatrix} : a, b, c \in \Z \right \}.
\end{equation*}
The manifold $W$ is a compact orientable $3$-manifold, so it is parallelizable and hence $w_2 (W) = 0$. Moreover, it is shown in~\cite[Example 3.A.34]{Huybrechts} that $b^1_W = 2$ and that
\begin{equation} \label{eq:W-Massey}
\text{ there exist $\alpha, \beta \in H^1_{\DR}(W)$ such that $\langle \alpha, \beta, \beta \rangle \neq 0$}.
\end{equation}

By the Whitney product formula, we have $w_2 (M) = w_2 (W) + w_2 (L)$. Thus if we choose $L$ to have vanishing $w_2$, then $w_2 (M)$ will vanish as required, and $M = W \times L$ will admit $\G$-structures.

\begin{thm} \label{thm:new-obstruction-example}
Let $L$ be a smooth compact connected oriented $4$-manifold with $w_2 (L) = 0$, and let $W$ be the real Iwasawa manifold described above. Then $M = W \times L$ admits $\G$-structures but cannot admit any \emph{torsion-free} $\G$-structures.
\end{thm}
\begin{proof}
Let $\pi : M \to W$ be the projection map. Consider the classes $\pi^* \alpha, \pi^* \beta \in H^1_{\DR}(M)$. By naturality of the Massey triple product, and since $p=q=r=1$, we have
\begin{equation*}
\langle \pi^* \alpha, \pi^* \beta, \pi^* \beta \rangle = \pi^* \langle \alpha, \beta, \beta \rangle \in \frac{H^2 (M)}{H^1 (M) \cdot H^1 (M)}.
\end{equation*}
Let $s: W \to W \times L$ be any section of $\pi$. Since $s^* \pi^* = (\pi \circ s)^* = \mathrm{Id}$, we deduce that
\begin{equation*}
\pi^*: \frac{H^2(W)}{H^1(W) \cdot H^1(W)} \to \frac{H^2(M)}{H^1(M) \cdot H^1(M)} \qquad \text{ is injective.}
\end{equation*}
Thus since $\langle \alpha, \beta, \beta \rangle \neq 0$ we have
\begin{equation*}
\langle \pi^* \alpha, \pi^* \beta, \pi^* \beta \rangle = \pi^* \langle \alpha, \beta, \beta \rangle \neq 0.
\end{equation*}
Since $| \pi^* \alpha | = | \pi^* \beta | = 1$ and $1 + 1 \neq 4$, we finally conclude by Corollary~\ref{cor:Massey} that $M$ \emph{does not admit a torsion-free $\G$-structure}.
\end{proof}

It remains to find an $L$ with $w_2 (L) = 0$ such that no previously known topological obstructions~\eqref{eq:known-obstructions} are violated, so that we have indeed established something new. We first collect several preliminary results that we will require.

By Poincar\'e duality $b^3_L = b^1_L$ and $b^2_W = b^1_W = 2$. The K\"unneth formula therefore yields
\begin{equation} \label{eq:M-Bettis}
\left. \begin{aligned}
b^1_M & = b^1_W + b^1_L = 2 + b^1_L, \\
b^2_M & = b^2_W + b^1_W b^1_L + b^2_L = 2 + 2 b^1_L + b^2_L, \\
b^3_M & = b^3_W + b^2_W b^1_L + b^1_W b^2_L + b^3_L = 1 + 2 b^1_L + 2 b^2_L + b^1_L = 1 + 3 b^1_L + 2 b^2_L. \phantom{A}
\end{aligned} \right\}
\end{equation}

\begin{rmk} \label{rmk:connect-sum}
Let $M$, $N$ be smooth compact oriented $n$-manifolds. There is a canonical way to make the connected sum $M \# N$ smooth, by smoothing around the $S^{n-1}$ with which we paste them together. With coefficients in either $R = \Z$ or $R = \Z / 2 \Z$, we have $H^k (M \# N, R) \cong H^k (M, R) \oplus H^k (N, R)$ for $k = 1, \ldots ,n-1$. This can be seen using the Mayer-Vietoris sequence. The isomorphism is induced by the map $p : M \# N \to M$ collapsing $N$, and the map $q: M \# N \to N$ collapsing $M$. For $k = n$, we have $H^k (M \# N)  \cong H^n (M, R) \cong H^n (N, R)$ with isomorphisms induced by $p$ and $q$ as before.
\end{rmk}

\begin{lemma} \label{lemma:4manifolds}
Let $L$ be a \emph{simply-connected} smooth compact oriented $4$-manifold, with intersection form
\begin{equation*}
Q: H_2 (L, \Z) \times  H_2 (L, \Z) \to \Z.
\end{equation*}
If the signature of $Q$ is $(p, q)$, let $\sigma (L) = p - q$. Then we have
\begin{itemize}
\item $w_2 (L) = 0$ if and only if $Q(a, a) \in 2 \Z$ for all $a \in H^2 (L, \Z)$;
\item $p_1(L) = 0$ if and only if $\sigma (L)$ is zero.
\end{itemize}
\end{lemma}
\begin{proof}
The first statement can be found in~\cite[Corollary 2.12]{LM}. The Hirzebruch signature theorem for $4$-manifolds, which can be found in~\cite[Theorem 1.4.12]{GS}, says that $p_1(L) = 3 \sigma(L)$. This immediately implies the second statement.
\end{proof}

Recall that K3 is the unique connected simply-connected smooth manifold underlying any compact complex surface with vanishing first Chern class. One way to define the K3 surface is by
\begin{equation*}
\text{K3} = \{ [z_0 : z_1 : z_2 : z_3] \in \C \PR^3 : z_0^4 + z_1^4 +z_2^4 + z_3^4 = 0 \}.
\end{equation*}
It is well-known (see~\cite[Page 75]{GS} or~\cite[Pages 127--133]{Sc}) that K3 has intersection form $Q_{\text{K3}} = -2 E_8 \oplus 3 H$, where $E_8$ is a certain even positive definite bilinear form, and $H = \begin{bmatrix} 0 & 1 \\ 1 & 0 \end{bmatrix}$, which is also even and has signature $0$. It follows that $Q_{\text{K3}}$ has signature $(3,19)$ and thus $\sigma(\text{K3}) = -16$. We also have that the Betti numbers of K3 are $b^1_{\text{K3}} = b^3_{\text{K3}} = 0$ and $b^2_{\text{K3}} = 22$.

\begin{prop} \label{prop:L-is-lemma}
Let $L = \text{K3} \, \# (S^1 \times S^3)$. Then $w_2 (L) = 0$, and for $M = W \times L$ where $W$ is the real Iwasawa manifold, none of the first four topological obstructions~\eqref{eq:known-obstructions} are violated. Thus $M$ cannot admit any torsion-free $\G$-structure.
\end{prop}
\begin{proof}
Since $b^1_{S^1 \times S^3} = b^3_{S^1 \times S^3} = 1$ and $b^2_{S^1 \times S^3} = 0$, Remark~\ref{rmk:connect-sum} tells us that the Betti numbers of $L$ are $b^1_L = b^3_L = 1$ and $b^2_L = 22$. In~\cite[Pages 20, 456]{GS} it is shown that $Q_{M \# N} = Q_M \oplus Q_N$, and consequently $\sigma(M \# N) = \sigma(M) + \sigma(N)$. Since $Q_{S^1 \times S^3} = 0$, we find that $Q_L$ is even and has nonzero signature. Thus by Lemma~\ref{lemma:4manifolds} we deduce that $p_1(L) \neq 0$ and $w_2 (L) = 0$. Now the equations~\eqref{eq:M-Bettis} tell us that the Betti numbers of $M$ are $b^1_M = 3$, $b^2_M = 26$, and $b^3_M = 48$. In particular, the first three conditions in~\eqref{eq:known-obstructions} are satisfied.

We now claim that $p_1 (M) \neq 0$. To see this, consider the inclusion $\iota : L \to M = W \times L$ into some vertical fibre $\{ \ast \} \times L$ of $M$ over $W$. Then $\iota^* (TM) = TL \oplus E$ where $W$ is the trivial rank $3$ real vector bundle over $L$. If $p_1 (TM) = 0$, then by naturality we he have $p_1 (TL) = \iota^* (p_1 (TM)) = 0$, which we showed was not the case. Thus, the fourth condition in~\eqref{eq:known-obstructions} is satisfied.
\end{proof}

\begin{rmk} \label{rmk:L-holonomy}
Because $b^1_M = 3$, it $M$ had any compact torsion-free $\G$-structure, it would have reduced holonomy $\SU{2}$. We have shown in Proposition~\ref{prop:L-is-lemma} that such a Riemannian metric cannot exist on $M$. It is not clear if there is any simpler way to rule out such a Riemannian metric on $M$.
\end{rmk}

Other examples of compact orientable spin $7$-manifolds that cannot be given a torsion-free $\G$-structure can likely be constructed similarly.   

\begin{rmk} \label{rmk:other-work}
The formality of compact $7$-manifolds with additional structure has been studied by several authors, in particular recently by Crowley--Nordstr\"om~\cite{CN} and Munoz--Tralle~\cite{MT}. Two of the results in~\cite{CN} are: there exist non-formal compact $7$-manifolds that have only trivial Massey triple products; and a non-formal compact manifold $M$ with $\G$ holonomy must have $b^2 (M) \geq 4$. One of the results in~\cite{MT} is that a compact simply-connected 7-dimensional Sasakian manifold is formal if and only if all its triple Massey products vanish.
\end{rmk}

\begin{rmk} \label{rmk:future}
A natural question is: can we actually establish \emph{formality} by extending our chain of quasi-isomorphisms? One idea is to quotient out the unwanted summands, but such a quotient map is not a dga morphism. One can also try to involve $\mathcal{L}_K$ or other operators that can descend to $H^{\bu}_{\ph}$, but the authors have so far had no success in this direction.
\end{rmk}

\end{document}